\documentclass[reqno]{amsart}
\usepackage{graphicx}
\usepackage[T1]{fontenc}
\usepackage{enumerate, amsmath, amsfonts, amssymb, amsthm, mathrsfs}
\usepackage{bbm}
\usepackage[lmargin=3.5cm,rmargin=3.5cm,top=2.9cm,bottom=2.9cm]{geometry}

\usepackage[all]{xy}
\usepackage{tikz}
\usetikzlibrary{arrows, decorations, shapes}
\usepackage{tikz-cd}
\usepackage{colonequals}
\usetikzlibrary{arrows}
\usepackage[colorlinks, citecolor=blue]{hyperref}
\usepackage[noabbrev,capitalise]{cleveref}
\usepackage{multirow}
\usepackage[normalem]{ulem}
\usepackage{cancel}
\usepackage{caption} 

\newtheorem{theorem}{Theorem}[section]
\newtheorem{corollary}[theorem]{Corollary}
\newtheorem{proposition}[theorem]{Proposition}
\newtheorem{lemma}[theorem]{Lemma}
\newtheorem*{theorem*}{Theorem}

\theoremstyle{definition}
\newtheorem{definition}[theorem]{Definition}
\newtheorem{example}[theorem]{Example}
\newtheorem{remark}[theorem]{Remark}
\newtheorem{question}{Question}

\newtheorem{thm}{Theorem}

\newtheorem{defn}{Definition}

\newcommand{\tc}{\mathcal{T}}

\newcommand{\ac}{\mathcal{A}}
\newcommand{\bc}{\mathcal{B}}

\newcommand{\mc}{\mathcal{M}}
\newcommand{\nc}{\mathcal{N}}
\newcommand{\uc}{\mathcal{U}}
\newcommand{\ec}{\mathcal{E}}

\newcommand{\fc}{\mathcal{F}}
\newcommand{\xc}{\mathcal{X}}
\newcommand{\yc}{\mathcal{Y}}
\newcommand{\vc}{\mathcal{V}}
\newcommand{\wc}{\mathcal{W}}

\newcommand{\Ext}{\operatorname{Ext}}

\newcommand{\Ker}{\operatorname{Ker}}

\newcommand{\Hom}{\operatorname{Hom}}

\newcommand{\add}{\operatorname{add}}

\newcommand{\proj}{\operatorname{proj}}

\title{Higher extension closure and $d$-exact categories}
\author[Kvamme]{Sondre Kvamme}
\address{Department of Mathematical Sciences\\ 
        NTNU\\ 
        NO-7491 Trondheim\\ 
        Norway}
\email{sondre.kvamme@ntnu.no}
\urladdr{https://sondkv.folk.ntnu.no}

\begin{document}


\begin{abstract}
We prove that any weakly idempotent complete $d$-exact category is equivalent to a $d$-cluster tilting subcategory of a weakly idempotent complete exact category, and that any weakly idempotent complete algebraic $(d+2)$-angulated category is equivalent to a $d$-cluster tilting subcategory of an algebraic triangulated category closed under $d$-shifts.  Furthermore, we show that the ambient exact category of a $d$-cluster tilting subcategory is unique up to exact equivalence, assuming it is weakly idempotent complete. This follows from a universal property of the inclusion of the $d$-cluster tilting subcategory. As a consequence of our theory we also get that any $d$-torsion class is $d$-cluster tilting in an extension-closed subcategory, and we recover the fact that any $d$-wide subcategory is $d$-cluster tilting in a unique wide subcategory. In the last part of the paper we rectify the description of the $d$-exact structure of a $d$-cluster tilting subcategory of a non-weakly idempotent complete exact category. 
\end{abstract}
	
	\maketitle

 \setcounter{tocdepth}{1}
 \tableofcontents
	
	\section{Introduction}

Higher homological algebra is a fundamental tool in higher Auslander--Reiten (AR) theory. Its chief actors are $(d+2)$-angulated, $d$-abelian, and $d$-exact categories \cite{GKO13,Jas16}. The first has a central role in the recent proof of the Donovan-Wemyss conjecture \cite{JKM22}, while the second and third play an important part in generalizing fundamental concepts from homological algebra \cite{J16,HJV20,HJ21,AMS22,FJS26,Gul25,AHJKPT25}. Their stage is higher AR theory, which itself has applications spanning algebraic geometry \cite{JKM22,HIMO23}, algebraic K-theory \cite{DJW19}, commutative algebra \cite{Iya07a,IY08,IT13}, combinatorics \cite{OT12}, and symplectic geometry \cite{DJL21}.

One of the main objects of study in higher AR theory are $d$-cluster tilting subcategories of triangulated, abelian, and exact categories. The concepts of $(d+2)$-angulated, $d$-abelian, and $d$-exact categories axiomatize the intrinsic properties of these subcategories. They resemble that of triangulated, abelian, and exact categories, with the difference roughly being that short exact sequences are replaced by longer exact sequences. This enables us to use intuition from classical homological algebra to study $d$-cluster tilting subcategories, giving us a powerful tool.

 A natural question is whether the converses hold. That is, whether every $(d+2)$-angulated, $d$-abelian, or $d$-exact category is equivalent to a $d$-cluster tilting subcategory of a triangulated, abelian, or exact category, respectively.  The answer is important, as it reveals whether the axiomatizations capture all the intrinsic properties of $d$-cluster tilting subcategories.
 
 For $d$-abelian categories an affirmative answer was given in \cite{Kva22,ENI22}, with prior results for projectively generated $d$-abelian categories \cite{Jas16} and $d$-abelian dualizing varieties \cite{IJ17}.
 
 In this paper we give an affirmative answer for weakly idempotent complete $d$-exact categories.

 \begin{thm}[\Cref{Cor:WeaklyIdemdExactdCT} and \Cref{Theorem:UniquenessAmbientExact}]\label{TheoremA}
    If $\mc$ is a weakly idempotent complete $d$-exact category, then there exists an equivalence $\mc\cong \nc$ of $d$-exact categories where $\nc$ is a $d$-cluster tilting subcategory of a weakly idempotent complete exact category. Furthermore, the exact category is unique up to exact equivalence.
 \end{thm}

 \Cref{TheoremA} is particularly important due to the ubiquity of $d$-exact categories. They capture situations in higher AR theory that $d$-abelian categories cannot. This includes $d$-representation infinite algebras \cite{HIO14,HIMO23}, $d$-complete algebras \cite{Iya11}, maximal Cohen-Macaulay modules over commutative local rings \cite{Han25} and orders \cite{Iya07a,Iya07}, hypersurface singularities \cite{BIKR08}, (non-local) commutative rings \cite{IW14}, and graded Cohen-Macaulay modules over graded commutative rings \cite{IT13,Han24}. They also occur as NCCR's of isolated singularities \cite{IW14} and in the higher dimensional generalization of McKay correspondence \cite{Iya07a,Iya07,IY08,IW14}.  Furthermore, they arise from Frobenius exact enhancements of triangulated categories with $d$-cluster tilting objects. This includes cluster categories of acyclic quivers \cite{Kel05a,BMRRT06}, subcategories of preprojective algebras \cite{BIRS09}, Amiot's cluster category \cite{Ami09}, generalized cluster categories from Jacobi-finite Ginzburg dg-algebras \cite{Ami09,Kel11}, and higher cluster categories of $\tau_d$-finite algebras \cite{IO13}. A challenge in higher AR theory is the lack of $d$-abelian categories, see \cite{HJS22,Vas23,ST24}, and we expect that $d$-exact categories will play a more fundamental role in developing the theory going forward. 

 Another example where $d$-exact categories take part is in enhancements of $(d+2)$-angulated categories. Following \cite{Jas16}, a $(d+2)$-angulated category is called algebraic if it has an enhancement by a Frobenius $d$-exact category. As a consequence of \Cref{TheoremA} we can give an affirmative answer to the question above for algebraic $(d+2)$-angulated categories.

\begin{thm}[\Cref{Theorem:Algd+2AngulatedisdCT}]\label{TheoremB}
    If $\fc$ is a weakly idempotent complete algebraic $(d+2)$-angulated category, then there exists an equivalence $\fc\cong \mathcal{G}$ of $(d+2)$-angulated categories where $\mathcal{G}$ is a $d$-cluster tilting subcategory of an algebraic triangulated category satisfying $\mathcal{G}[d]=\mathcal{G}$.
\end{thm}

There exists a different definition of algebraic $(d+2)$-angulated category in \cite{JKM22} for idempotent complete categories, using enhancements by DG categories. We deduce that it coincides with the one of \cite{Jas16} using \Cref{TheoremA} and \Cref{TheoremB}.

  Uniqueness in \Cref{TheoremA} is a consequence of the inclusion of $\nc$ into its ambient exact category being universal with respect to functors that send admissible $d$-exact sequences to acyclic complexes. Using this we construct an equivalence between the $2$-category of weakly idempotent complete $d$-exact categories and the $2$-category of $d$-cluster tilting subcategories of weakly idempotent complete exact categories. Hence, we can freely switch between these two perspectives without losing information, and depending on the problem choose the most advantageous approach. For example, to construct new examples it is often easier to look for $d$-exact structures \cite{Kla21,HAZ25,Kla24}. On the other hand, $d$-cluster tilting subcategories allow for a straightforward definition of admissible morphism and are therefore more convenient for proving e.g. the Auslander correspondence \cite{EN21}. 
  
  Showing the universal property for $d$-cluster tilting subcategories of exact categories is more complicated than for $d$-cluster tilting subcategories of abelian categories. The main issue is the lack of limits and colimits for exact categories. To circumvent this, we use the Quillen--Gabriel embedding to reduce the general case to that of a functor to an abelian category. The reduction relies on \Cref{Lemma:d-CTFiltrations}, which states how an exact category with a $d$-cluster tilting subcategory is generated by filtrations of objects with certain Ext-vanishing property. We believe this lemma is of independent interest.
  
One of the main difficulties in constructing the exact category in \Cref{TheoremA} is the absence of a well-developed theory of localizations for exact categories. This means that the proof for $d$-abelian categories in \cite{Kva22,ENI22}, by generalizing Auslander's formula, does not adapt to $d$-exact categories in a straightforward way. Furthermore, even though there is a version of Auslander's formula for exact categories \cite{HKR22}, it requires a notion of admissible morphism, and it is not clear how to define this for $d$-exact categories.

A different approach was used in \cite{Ebr21}, by generalizing Quillen-Gabriel embedding theorem from exact to $d$-exact categories. More precisely, they construct a fully faithful functor from a $d$-exact category $\mc$ into the abelian category $\mathcal{L}(\mc)$ of left exact functors $\mathcal{M}^{\operatorname{op}}\to \operatorname{Ab}$, and demonstrate that the image is $d$-rigid, i.e. that the Ext-groups vanish in degrees strictly between $0$ and $d$. This is a first step towards being $d$-cluster tilting. Unfortunately, $\mc$ cannot be $d$-cluster tilting in $\mathcal{L}(\mc)$, since $\mathcal{L}(\mc)$ is in some sense too big. Consequently, we are left with the problem of determining where $\mc$ could be $d$-cluster tilting. This leads to the central concept studied in this paper: higher extension closure.  

\begin{defn}\label{d-extension closure}
A $d$-rigid subcategory $\uc$ of an exact category $\ec$ is called \textit{$d$-extension closed} or \textit{closed under }$d$\textit{-extensions} in $\ec$ if any exact sequence
\[
0\to U\to E_1\to \cdots \to E_d\to U'\to 0
\]
with $U,U'\in \mathcal{U}$ is Yoneda-equivalent to an exact sequence
\[
0\to U\to U_1\to \cdots \to U_d\to U'\to 0
\]
where $U_i\in \uc$ for all $1\leq i\leq d$.
\end{defn} 

By \cite{EN23} the image of $\mc\to \mathcal{L}(\mc)$ is closed under $d$-extensions. The construction of the exact category in \Cref{TheoremA} is therefore a consequence of the following result.

\begin{thm}\label{TheoremC}[\Cref{Theorem:dCT} and \Cref{Theorem:UniquenessdCT}]
 Let $\uc$ be a weakly idempotent complete $d$-rigid and $d$-extension closed subcategory of an exact category $\ec$. Then there exists a unique extension closed subcategory $\ec'$ of $\ec$ such that $\uc$ is $d$-cluster tilting in $\ec'$ and the canonical map $$\Ext^d_{\ec'}(U,V)\to \Ext^d_{\ec}(U,V)$$ is an isomorphism for all $U,V\in \uc$.
\end{thm}
Since the map $\Ext^d_{\ec'}(U,V)\to \Ext^d_{\ec}(U,V)$ is an isomorphism, the $d$-exact structure on $\uc$ consists of the complexes of length $d+2$ in $\uc$ which are acyclic in $\ec$. For $d=1$ this recovers the exact structure on an extension closed subcategory of an exact category.

The construction of $\ec'$ in \Cref{TheoremC} is done in several steps. We first consider one-sided versions of $d$-cluster tilting, called right and left maximal $d$-rigid, see \Cref{Definition:MaxRigid}. We show that $\uc$ is contained as a right maximal $d$-rigid subcategory of
\begin{align*}
	\uc^{d}(\ec):=\{E\in \ec \mid \exists\text{ }0\to E\to U^1\to \cdots \to U^d\to 0 \text{ exact, }U^i\in \uc \text{, }1\leq i\leq d\}
\end{align*}
and a left maximal $d$-rigid subcategory of
\begin{align*}
    \uc_{d}(\ec):=\{E\in \ec \mid \exists\text{ }0\to U_d\to \cdots \to U_1\to E\to 0 \text{ exact, }U_i\in \uc \text{, }1\leq i\leq d\}.
	\end{align*}
Most of the work here involves showing that $\uc^{d}(\ec)$ and $\uc_{d}(\ec)$ are extension closed in $\ec$ and therefore inherit an exact structure. This relies on $\uc$ being $d$-extension closed, see \Cref{thm: extclosed}. Next we repeat the procedure iteratively, giving a sequence of subcategories $$\dots \subseteq \uc^d(\uc_d(\uc^d(\ec)))\subseteq \uc_d(\uc^d(\ec))\subseteq \uc^d(\ec) \subseteq \ec.$$ With a bit of work we show that the sequence stabilizes after $2$-steps. This implies that
\[
\uc_d(\uc^d(\ec))=\uc^d(\uc_d(\ec))
\]
and $\uc$ must be $d$-cluster tilting in this subcategory, since it is left and right maximal $d$-rigid.

Closure under $d$-extensions has previously been considered for subcategories of $d$-cluster tilting subcategories, see \cite{Fed19,Fed20}. In particular, it appears in the characterization of higher wide subcategories \cite{HJV20} and higher torsion classes \cite{J16,AHJKPT25}. Our results imply the following in these cases.

\begin{thm}\label{TheoremD}
Let $\mc$ be a $d$-cluster tilting subcategory of an abelian category $\ac$.
    \begin{enumerate}
    \item\label{TheoremD:1}   If $\wc$ is a $d$-wide subcategory of $\mc$, then there exists a unique wide subcategory $\bc$ of $\ac$ such that $\wc$ is a $d$-cluster tilting subcategory of $\bc$. Furthermore
    \begin{align*}
        \bc&=\{A\in \ac\mid \exists\text{ }0\to W_d\to \cdots \to W_1\to A\to 0 \text{ exact, }W_i\in \wc \text{, }1\leq i\leq d\} \\
        & =\{A\in \ac\mid \exists\text{ }0\to A\to W^1\to \cdots \to W^d \to 0 \text{ exact, }W^i\in \wc \text{, }1\leq i\leq d\}.
    \end{align*}
        \item\label{TheoremD:2} Assume $\ac$ is an abelian length category. If $\uc$ is a $d$-torsion class of $\mc$, then 
    \[
    \uc_d(\ac)=\{A\in \ac\mid \exists\text{ }0\to U_d\to \cdots \to U_1\to A\to 0 \text{ exact, }U_i\in \uc \text{, }1\leq i\leq d\}
    \]
    is closed under extensions and contains $\uc$ as a $d$-cluster tilting subcategory.
    \end{enumerate}
\end{thm}

\Cref{TheoremD} \eqref{TheoremD:1} generalizes \cite[Theorem A]{HJV20} by dropping the assumptions on functorially finiteness and on the ambient abelian category being the module category of a finite-dimensional algebra. \Cref{TheoremD} \eqref{TheoremD:2} is new.

We end the introduction by mentioning some open problems. The first is whether we can drop the weak idempotent completeness hypothesis in \Cref{TheoremA}. 

\begin{question}\label{Question:NonweaklyIdemComplete}
    Let $\mc$ be a non-weakly idempotent complete $d$-exact category. Does there exist an equivalence $\mc\cong \nc$ of $d$-exact categories where $\nc$ is a $d$-cluster tilting subcategory of an exact category?
\end{question}
 Note that uniqueness in \Cref{TheoremA} does not hold without the ambient exact category being weakly idempotent complete, see \Cref{Example:NonUniqueness}. We refer to \Cref{Theorem:dCTImpliesdExactNotWeakIdemPotent} for the $d$-exact structure of a $d$-cluster tilting subcategory of a non-weakly idempotent complete exact category. The description in \cite[Theorem 4.14]{Jas16} is not completely correct in this case, since acyclicity is not preserved by weak isomorphisms for non-weakly idempotent complete categories. In particular, \cite[Example 2.5]{Ebr21} is not a counterexample to \Cref{Question:NonweaklyIdemComplete}. 

Next we ask whether we can drop the algebraic assumption in \Cref{TheoremB}.

\begin{question}\label{Question:NonAlgebraic}
    Let $\fc$ be a weakly idempotent complete $(d+2)$-angulated category. Does there exist an equivalence $\fc\cong \mathcal{G}$ of $(d+2)$-angulated categories where $\mathcal{G}$ is a $d$-cluster tilting subcategory of a triangulated category satisfying $\mathcal{G}[d]=\mathcal{G}$?
\end{question}
Note that weak idempotent completeness is a necessary assumption, since any triangulated category is weakly idempotent complete, and hence any $d$-cluster tilting subcategory of a triangulated category inherits this property. 

Finally, we ask about uniqueness of the ambient triangulated category.

\begin{question}\label{Question:UniquenessTriang}
    Let $\fc_1$ and $\fc_2$ be $d$-cluster tilting subcategories of triangulated categories $\tc_1$ and $\tc_2$, and assume $\fc_1[d]=\fc_1$ and $\fc_2[d]=\fc_2$. If we have an equivalence $\fc_1\cong\fc_2$ of $(d+2)$-angulated categories, does there exists an equivalence $\tc_1\cong \tc_2$ of triangulated categories? 
\end{question}
Note that \Cref{Question:UniquenessTriang} cannot be deduced from the uniqueness in \Cref{TheoremA}, even for algebraic triangulated categories. The issue is that lifting the equivalence $\fc_1\cong\fc_2$ to a functor between the enhancements might not be possible.

The structure of the paper is as follows. In \Cref{Section:Preliminaries} we recall properties we need for exact categories, $d$-cluster tilting subcategories, $d$-exact categories, and the generalization of Quillen-Gabriel embedding theorem from \cite{Ebr21}. In \Cref{Section:MaximalRigidSubcat} we study right and left maximal $d$-rigid subcategories and give necessary and sufficient conditions for a $d$-rigid subcategory $\uc$ to be right maximal $d$-rigid in $\uc^d(\ec)$, see \Cref{thm: extclosed}. In \Cref{Section:Higher extension-closure} we assume in addition that $\uc$ is closed under $d$-extensions, and show that the closure of $\uc$ under additive complements is $d$-cluster tilting in $\uc_d(\uc^d(\ec))$, see \Cref{Theorem:dCT}. We also prove a uniqueness property of $\uc_d(\uc^d(\ec))$, see \Cref{Theorem:UniquenessdCT}. Together these results imply \Cref{TheoremC}. In the last part of \Cref{Section:Higher extension-closure} we prove \Cref{TheoremD}. 

In \Cref{Section:EmbeddingThm} we obtain an extension closed subcategory $\ec(\mc)$ of the category of left exact functors $\mathcal{L}(\mc)$ which contains the weak idempotent completion of $\mc$ as a $d$-cluster tilting subcategory, see \Cref{Cor:WeaklyIdemdExactdCT}. This implies half of \Cref{TheoremA}. As a consequence we show that axiom (E1) for $d$-exact categories is redundant, see \Cref{Corollary:E1IsRedundant}. In \Cref{Section:UniversalProperty} we prove the universal property of the ambient exact category of a $d$-cluster tilting subcategory, assuming it is weakly idempotent complete, see \Cref{Theorem:UnivProperty}. As a consequence we obtain an equivalence between the $2$-category of weakly idempotent complete $d$-exact categories and the $2$-category of $d$-cluster tilting subcategories of weakly idempotent complete exact categories, see \Cref{Theorem:2equivalence}. This implies the second half of \Cref{TheoremA} on uniqueness of the ambient exact category.

In \Cref{Section:(d+2)-Angulated} we consider algebraic $(d+2)$-angulated categories and prove \Cref{TheoremB}. We also show that the notions of algebraic $(d+2)$-angulated categories in \cite{Jas16} and in \cite{JKM22} coincide, see \Cref{(d+2)AngulatedDefCoincides}. 

In \Cref{Section:NonWeaklyIdemComp} we describe the $d$-exact structure of a $d$-cluster tilting subcategory of a non-weakly idempotent complete exact category, see \Cref{Theorem:dCTImpliesdExactNotWeakIdemPotent}. We also explain how \cite[Example 2.5]{Ebr21} does not give a counterexample to \Cref{Question:NonweaklyIdemComplete}, see \Cref{Example:CounterExdExNotdCTNotValid}.

	\subsection*{Conventions}
	All categories are assumed to be additive, i.e. enriched over abelian groups and admitting finite biproducts.  Subcategories are assumed to be closed under finite direct sums. Given a class $\xc$ of objects in an additive category $\mathcal{C}$, we write $\operatorname{add}\xc$ for the smallest subcategory of $\mathcal{C}$ containing $\xc$ and which is closed under summands and finite direct sums. For a finite-dimensional algebra $\Lambda$, we write $\operatorname{mod}\Lambda$ for the category of finitely generated left $\Lambda$-modules. The category of abelian groups is denoted $\operatorname{Ab}$. The suspension functor of a triangulated category $\tc$ is denoted $[1]\colon \tc\to \tc$, and its composites $[n]=\underbrace{[1]\circ\cdots\circ [1]}_{n\text{ times}}$.
	
	\section{Preliminaries}\label{Section:Preliminaries}
	
	\subsection{Exact categories}
	Here we summarize the basic concepts and results needed from the theory of exact categories. For more details on exact categories see \cite{Bue10} or \cite{Kel90}.
	
	A \textit{conflation category} is an additive category $\ec$ endowed with a class of kernel-cokernel pairs in $\ec$ closed under isomorphism, called \textit{conflations}. The kernel part of a conflation is called an \textit{inflation}, and the cokernel part is called a \textit{deflation}. An \textit{exact category} (in the sense of Quillen) is a conflation category satisfying
	\begin{itemize}
		\item[\textbf{R0}] For any object $E\in \ec$, the identity $1_E$ is an inflation.
		\item[\textbf{L0}] For any object $E\in \ec$, the identity $1_E$ is a deflation.
		\item[\textbf{R1}] Inflations are closed under composition.
		\item[\textbf{L1}] Deflations are closed under composition.
		\item[\textbf{R2}] The pullback of a deflation exists, and deflations are stable under pullbacks.
		\item[\textbf{L2}] The pushout of an inflation exists, and inflations are stable under pushouts.
	\end{itemize}
     Axioms \textbf{R2} and \textbf{L2} tell us that given the solid part of the diagrams 
	\[
	\begin{tikzcd}[column sep=20, row sep=20]
	E_2 \arrow[r,dashed] \arrow[d,dashed,"h"] & E_1  \arrow[d,"f"] \\
	E_4 \arrow[r] & E_3   
	\end{tikzcd} \quad \text{and} \quad 
	\begin{tikzcd}[column sep=20, row sep=20]
	F_2 \arrow[r] \arrow[d,"g"] & F_1 \arrow[d,dashed,"k"] \\
	F_4\arrow[r,dashed,] & F_3  
	\end{tikzcd}
	\]
	with $f$ a deflation and $g$ an inflation, then the dashed arrows exists making the first square cartesian, i.e. a pullback square, and the second square cocartesian, i.e. a pushout square. Furthermore, the morphisms $h$ and $k$ will again be a deflation and an inflation, respectively. In this case both squares must be bicartesian, i.e. cartesian and cocartesian, see \mbox{\cite[Proposition 2.12]{Bue10}}.

 An \textit{exact functor} is a functor $F\colon \ec\to \ec'$ between exact categories $\ec,\ec'$ that preserves conflations, i.e. whenever $0\to E_3\to E_2\to E_1\to 0$ is a conflation in $\ec$, then 
     \[
     0\to F(E_3)\to F(E_2)\to F(E_1)\to 0
     \]
     is a conflation in $\ec'$. An \textit{exact equivalence} is an exact functor which is also an equivalence of categories.

 A subcategory $\ec'$ of an exact category $\ec$ is \textit{extension-closed} or \textit{closed under extensions} if for any conflation
    \[
    0\to E_3\to E_2\to E_1\to 0
    \]
    in $\ec$ with $E_1,E_3\in \ec'$, then $E_2\in \ec'$. In this case, $\ec'$ inherits an exact structure from $\ec$, where the conflations in $\ec'$ are the conflations in $\ec$ whose objects all belong to $\ec'$. This makes the inclusion functor $\ec'\to \ec$ into an exact functor. Note that any exact category is equivalent to an extension-closed subcategory of an abelian category, and if the abelian category is endowed with the exact structure consisting of all its kernel-cokernel pairs and the subcategory with the induced exact structure, then this is an exact equivalence, see \cite[Proposition A.2]{Kel90}.  Throughout this paper we always endow an extension-closed subcategory with the induced exact structure from the ambient exact category.

    An additive category $\ec$ is \textit{weakly idempotent complete} if any split epimorphism has a kernel, or equivalently if any split monomorphism has a cokernel, see \cite[Lemma 7.1]{Bue10}. If $\ec$ is exact then this implies that all split epimorphisms are deflations and all split monomorphisms are inflations, see \cite[Corollary 7.5]{Bue10}.  Note that a subcategory $\ec'$ of a weakly idempotent complete category $\ec$ is weakly idempotent complete if and only if it is \textit{closed under additive complements}, i.e. $E\oplus E'\in \ec'$ and $E\in \ec'$ implies $E'\in \ec'$.
    
    An additive category $\ec$ is \textit{idempotent complete} if all idempotents $e$ split, i.e. if there exists morphisms $r$ and $i$ such that $e=i\circ r$ and $r\circ i=\operatorname{id}_E$ for some object $E$ in $\ec$. In this case we call $(i,r)$ a \textit{splitting} of $e$.  Note that a subcategory $\ec'$ of an idempotent complete category $\ec$ is idempotent complete if and only if it is \textit{closed under direct summands}, i.e. $E\oplus E'\in \ec'$ implies $E\in \ec'$ and $E'\in \ec'$. 
    
    Any idempotent complete category must be weakly idempotent complete. Indeed, splittings $(r\colon E\to F,s\colon F\to E)$ and $(r'\colon E\to F',s'\colon F'\to E)$ of the idempotents $e$ and $1-e$, respectively, give rise to mutual inverse isomorphisms
    \[
    \begin{pmatrix}
        s&s'
    \end{pmatrix}\colon F\oplus F' \xrightarrow{\cong} E \quad \text{and} \quad \begin{pmatrix}
        r\\r'
    \end{pmatrix}\colon E\xrightarrow{\cong} F\oplus F'.
    \] 
    In particular, $F'$ is a kernel of $r$.  Since any split epimorphism $r$ occurs in a splitting $(i,r)$ of some idempotent, it must always have a kernel if the category is idempotent complete.

    An additive category $\ec$ has a fully faithful functor into an idempotent complete category  $\tilde{\ec}$ (resp. weakly idempotent complete category $\hat{\ec}$), which is universal with respect to all additive functors into idempotent complete categories (resp.  weakly idempotent complete categories), see \cite[Section 6 and Remark 7.8]{Bue10} and \cite[Proposition A.10]{HR19}. The category $\tilde{\ec}$ is called the \textit{idempotent completion} of $\ec$, and the category $\hat{\ec}$ is called the \textit{weak idempotent completion} of $\ec$. We often identify $\ec$ as full subcategories of $\tilde{\ec}$ and $\hat{\ec}$. Any object in $\tilde{\ec}$ must be a direct summand of an object in $\ec$, see the construction in \cite[Remark 6.3]{Bue10}. Similarly, for any object $F$ in $\hat{\ec}$ we can find an object $E\in \ec$ such that $F\oplus E \in \ec$, see \cite[Proposition A.11]{HR19}. If $\ec$ is exact, then $\tilde{\ec}$ and $\hat{\ec}$ can be endowed with an exact structure such that  the functors $\ec\to \tilde{\ec}$ and $\ec\to \hat{\ec}$ are exact and satisfy a universal property with respect to all exact functors into idempotent complete exact categories (resp. weakly idempotent complete exact categories), see \cite[Section 6 and Remark 7.8]{Bue10}.

Given an additive category $\ec$, let $C^b(\ec)$ denote the category of bounded chain complexes of $\ec$. Objects and morphisms in $C^b(\ec)$ are written as $X_\bullet$ and $f_\bullet\colon X_\bullet\to Y_\bullet$, so that $X_\bullet$ is given by the complex
\[
\cdots \xrightarrow{d_{i+2}^X} X_{i+1}\xrightarrow{d_{i+1}^X} X_i\xrightarrow{d_{i}^X} X_{i-1}\xrightarrow{d_{i-1}^X} \cdots
    \]
and $f_\bullet$ is given by a commutative diagram
       \begin{equation*}
	\begin{tikzcd}
	\cdots \arrow[r] & X_{i+1}\arrow[r,"d_{i+1}^X"] \arrow[d,"f_{i+1}"] & X_i \arrow[r,"d_i^X"] \arrow[d,"f_i"] & X_{i-1} \arrow[r,"d_{i-1}^X"] \arrow[d,"f_{i-1}"] \arrow[r]& \cdots  \\
	 \cdots \arrow[r]& Y_{i+1}\arrow[r,"d_{i+1}^Y"] & Y_i \arrow[r,"d_i^Y"]  & Y_{i-1} \arrow[r,"d_{i-1}^Y"] \arrow[r] & \cdots .
	\end{tikzcd} 
	\end{equation*}
where the rows are complexes. Given a morphism $f_\bullet$ as above, its \textit{cone} $C_\bullet$ is the complex
\[
\cdots \xrightarrow{} X_{i}\oplus Y_{i+1}\xrightarrow{\begin{pmatrix}-d_{i}^X&0\\f_{i}&d_{i+1}^Y\end{pmatrix}} X_{i-1}\oplus Y_{i}\xrightarrow{\begin{pmatrix}-d_{i-1}^X&0\\f_{i-1}&d_{i}^Y\end{pmatrix}} \cdots 
\]
where $X_{i-1}\oplus Y_i$ is in degree $i$. The inclusion and the projection
\[
Y_i\xrightarrow{\begin{pmatrix}0\\ \operatorname{id}_{Y_i}\end{pmatrix}} X_{i-1}\oplus Y_{i} \quad \quad X_{i-1}\oplus Y_{i} \xrightarrow{\begin{pmatrix}\operatorname{id}_{X_{i-1}}&0\end{pmatrix}} X_{i-1}
\]
induce morphisms $Y_\bullet\to C_\bullet$ and $C_\bullet\to X_\bullet[1]$ of chain complexes, where $X_\bullet[1]_i=X_{i-1}$ denotes the shift of $X_\bullet$. Hence, we get a sequence $$X_\bullet\xrightarrow{f_\bullet} Y_\bullet\to C_\bullet\to X_\bullet[1]$$ of morphisms of chain complexes, which forms a distinguished triangle in triangulated structure of the bounded homotopy category $K^b(\ec)$ of $\ec$.

Now assume $\ec$ is an exact category. A complex
   \[
X_\bullet=(\cdots \to X_{1}\to X_0\to X_{-1}\to \cdots )
    \]
    in $\ec$ is called \textit{acyclic} or \textit{exact} if there exists conflations 
    \[
    0\to Z_i(X_\bullet)\to X_i\to Z_{i-1}(X_\bullet)\to 0
    \]
    such that the morphism $X_i\to X_{i-1}$ is equal to the composite $X_i\to Z_{i-1}(X_\bullet)\to X_{i-1}$ for $i\in \mathbb{Z}$. The subcategory of bounded acyclic complexes is denoted by $\operatorname{Ac}^b(\ec)$. 
    
    Note that acyclicity might not be preserved under homotopy equivalence. In fact, if $(i,r)$ is a splitting of an idempotent $e\colon E'\to E'$, then the complex
    \[
    0\to E\xrightarrow{i}E'\xrightarrow{\operatorname{id}_{E'}-e}E'\xrightarrow{r}E\to 0
    \]
    is nullhomotopic, but not acyclic unless $r$ is a deflation.   In general, $\operatorname{Ac}^b(\ec)$ is closed under isomorphisms in $K^b(\ec)$ if and only if $\ec$ is weakly idempotent complete, see \mbox{\cite[Corollary 10.14]{Bue10}}. In this case, $\operatorname{Ac}^b(\ec)$ is also closed under direct summands in $K^b(\ec)$, see \mbox{\cite[Corollary 10.14]{Bue10}}.

   The cone of a morphism between acyclic complexes is always acyclic \cite[Lemma 1.1]{Nee90} even if $\ec$ is not weakly idempotent complete. Hence, $\operatorname{Ac}^b(\ec)$ is always a triangulated subcategory of $K^b(\ec)$ (not necessarily closed under isomorphism). The \textit{bounded derived category} of $\ec$ is defined as the Verdier quotient $$D^b(\ec)\colonequals K^b(\ec)/\operatorname{Ac}^b(\ec).$$ 
   Note that the inclusion $\ec\to \hat{\ec}$ into the weak idempotent completion $\hat{\ec}$ of $\ec$ induces a derived equivalence $D^b(\ec)\to D^b(\hat{\ec})$ \cite[Remark 1.12.3]{Nee90}.

Given an integer $n>0$, the $n$\textit{th Yoneda Ext-group} of two objects $X,Y$ in $\ec$ is defined to be 
	\[
	\Ext^n_\ec(X,Y)\colonequals \Hom_{D^b(\ec)}(X,Y[n]).
	\]
	The elements of $\Ext^n_\ec(X,Y)$ can be identified with equivalence classes of exact sequences
	\[
	0\to Y\to E_n\to \cdots \to E_1\to X\to 0
	\]
	with $n$ middle terms, where two sequences 
 \[
	0\to Y\to E_n\to \cdots \to E_1\to X\to 0 \quad \text{and} \quad 0\to Y\to E'_n\to \cdots \to E'_1\to X\to 0
\]
are equivalent if we can find a commutative diagram
\begin{equation*}
		\begin{tikzcd}
		0 \arrow[r] & Y\arrow[r] & E_n \arrow[r,""]  & E_{n-1} \arrow[r,""] & \cdots \arrow[r,""] & E_1\arrow[r,""] & X\arrow[r]  & 0 \\
		0 \arrow[r,""] & Y \arrow[u,equal] \arrow[d,equal] \arrow[r,""] & E_n'' \arrow[u,""] \arrow[d,""] \arrow[r,""] & E_{n-1}'' \arrow[u,""] \arrow[d,""] \arrow[r,""] & \cdots  \arrow[r,""] & E_1'' \arrow[u,""] \arrow[d,""] \arrow[r,""] & X \arrow[r] \arrow[u,equal] \arrow[d,equal] & 0 \\
		0 \arrow[r,""] & Y \arrow[r,""] & E_n' \arrow[r,""]  & E_{n-1}' \arrow[r,""]  & \cdots  \arrow[r,""]  & E_1'  \arrow[r,""] & X  \arrow[r]  & 0
		\end{tikzcd} 
		\end{equation*}
where the middle row is an exact sequence. This is equivalent to the existence of a similar commutative diagram, but where the direction of the vertical arrows are reversed. For more details see \cite[Proposition A.13]{Pos11} or \cite[Chapter 6]{FS10}. For a subcategory $\xc$ and an object $E$ of $\ec$, we write $\Ext^i_\ec(\xc,E)=0$ (respectively, $\Ext^i_\ec(E,\xc)=0$) to denote that $\Ext^i_\ec(X,E)=0$ (respectively, $\Ext^i_\ec(E,X)=0$) for all $X\in \xc$.

\subsection{Cluster tilting subcategories}

Fix an exact category $\ec$ and a subcategory $\mc$ of $\ec$.

 	  A \textit{right $\mc$-approximation} of an object $E$ in $\ec$ is a morphism $M\to E$ with $M\in \mc$ such that $\Hom_\ec(M',M)\to \Hom_\ec(M',E)$ is an epimorphism for all $M'\in \mc$.
 	 A \textit{left $\mc$-approximation} of $E$ is a morphism $E\to M$ with $M\in \mc$ such that $\Hom_\ec(M,M')\to \Hom_\ec(E,M')$ is an epimorphism for all $M'\in \mc$.
 	 $\mc$ is called \textit{contravariantly finite} if any object in $\ec$ has a right $\mc$-approximation, \textit{covariantly finite} if any object in $\ec$ has a left $\mc$-approximation, and \textit{functorially finite} if it is both contravariantly and covariantly finite.

    $\mc$ is called \textit{generating} if for any $E\in \ec$ there exists a deflation $M\to E$ with $M\in \mc$. Dually, $\mc$ is called \textit{cogenerating} if for any $E\in \ec$ there exists an inflation $E\to M$ with $M\in \mc$.

 If $\mc$ is generating and contravariantly finite, then for any $E\in \ec$ there exists a right $\mc$-approximation $M\to E$ which is a deflation, see e.g. proof of \cite[Proposition 4.4]{Kva21}. If $\ec$ is in addition weakly idempotent complete, then any right $\mc$-approximation must be a deflation. To see this, let $M\to E$ be a right $\mc$-approximation, and let $M'\to E$ be a deflation with $M'\in \mc$. Then $M\to E$ must factor through $M'\to E$, and by a strong version of the Obscure axiom it follows that $M\to E$ is a deflation, see \cite[Proposition 7.6]{Bue10}. Dual statements hold for cogenerating and covariantly finite subcategories. 

\begin{definition}\cite[Definition 2.2]{Iya07a}, \cite[Definition 4.13]{Jas16} 
The subcategory $\mc$ is called $d$-\textit{cluster tilting} if it is generating, cogenerating, functorially finite, and satisfies
 	\begin{align*}
 	\mc &=\{E\in \ec\mid \Ext^i_\ec(E,\mc)=0 \text{ for all }0<i<d\} \\
 		&=\{E\in \ec\mid \Ext^i_\ec(\mc,E)=0 \text{ for all }0<i<d\}.
 	\end{align*}
\end{definition}

Note that $\mc\subseteq \ec$ is $1$-cluster tilting if and only if $\mc=\ec$.

We often use the following reformulation of $d$-cluster tilting. Here $\mc$ is called $d$-\textit{rigid} if $\Ext^i_{\ec}(M,M')=0$ for all $0<i<d$ and $M,M'\in \mc$.
 
 \begin{proposition}\label{Reformulation:d-CT}
 	The subcategory $\mc$ is $d$-cluster tilting if and only if it is $d$-rigid, closed under direct summands, and for any $E\in \ec$ there exists exact sequences
  \begin{align*}
     & 0\to E\to M^1\to \cdots \to M^d\to 0 \\
     & 0\to M_d\to \cdots \to M_1\to E\to 0
  \end{align*}
  where $M_1,\dots, M_d\in \mc$ and $M^1,\dots, M^d\in \mc$.
 \end{proposition}	
 
 \begin{proof}
 See \cite[Proposition 4.4]{Kva21}.
 \end{proof}

The following result shows that the $d$-cluster tilting property is preserved under (weak) idempotent completions.

\begin{proposition}\label{Prop:dCTFor(Weak)IdemComp}
    Let $\mc$ be a $d$-cluster tilting subcategory of $\ec$. The following hold.
    \begin{enumerate}
    \item\label{Prop:dCTFor(Weak)IdemComp:1} The weak idempotent completion of $\mc$ is a $d$-cluster tilting subcategory of the weak idempotent completion of $\ec$.
    \item\label{Prop:dCTFor(Weak)IdemComp:2} The idempotent completion of $\mc$ is a $d$-cluster tilting subcategory of the idempotent completion of $\ec$.
    \end{enumerate}
\end{proposition}

\begin{proof}
 Let $\hat{\mc}$ and $\hat{\ec}$ denote the weak idempotent completions and $\tilde{\mc}$ and $\tilde{\ec}$ the idempotent completions of $\mc$ and $\ec$. By the universal property of $\hat{\mc}$ and $\tilde{\mc}$, the inclusions $\mc\to \ec\to \hat{\ec}$ and $\mc\to \ec\to \tilde{\ec}$ extends to exact functors
 \[
 \hat{\mc}\to \hat{\ec} \quad \text{and} \quad \tilde{\mc}\to \tilde{\ec}.
 \]
  These functors are also fully faithful, since $\mc\to \ec\to \hat{\ec}$ and $\mc\to \ec\to \tilde{\ec}$ are fully faithful and any object in $\hat{\mc}$ and $\tilde{\mc}$ is a direct summand of an object in $\mc$.
 Hence, we can identify $\hat{\mc}$ and $\tilde{\mc}$ with full subcategories of $\hat{\ec}$ and $\tilde{\ec}$. 
 
 Now by \cite[Remark 1.12.3]{Nee90} and \cite[Theorem 2.8]{BS01} the canonical maps
 \[
 \Ext^i_\ec(E,F)\to \Ext^i_{\hat{\ec}}(E,F) \quad \text{and} \quad \Ext^i_\ec(E,F)\to \Ext^i_{\tilde{\ec}}(E,F)
 \]
are isomorphisms for all $E,F\in \ec$ and all $i>0$. Since
\[
\mc =\{E\in \ec\mid \Ext^i_\ec(E,\mc)=0 \text{ for }0<i<d\}=\{E\in \ec\mid \Ext^i_\ec(\mc,E)=0 \text{ for }0<i<d\}
\]
and any object in $\hat{\ec}$ or $\tilde{\ec}$ (respectively $\hat{\mc}$ or $\tilde{\mc}$) is a direct summand of an object in $\ec$ (respectively $\mc$), we get that
\begin{align*}
& \hat{\mc} =\{E\in \hat{\ec}\mid \Ext^i_{\hat{\ec}}(E,\hat{\mc})=0 \text{ for }0<i<d\}=\{E\in \hat{\ec}\mid \Ext^i_{\hat{\ec}}(\hat{\mc},E)=0 \text{ for }0<i<d\} \\
& \tilde{\mc} =\{E\in \tilde{\ec}\mid \Ext^i_{\tilde{\ec}}(E,\tilde{\mc})=0 \text{ for }0<i<d\}=\{E\in \tilde{\ec}\mid \Ext^i_{\tilde{\ec}}(\tilde{\mc},E)=0 \text{ for }0<i<d\}.
\end{align*}
Since generating, cogenerating, and functorially finiteness holds for $\hat{\mc}$ and $\tilde{\mc}$ if they hold for $\mc$, this proves the claim.
\end{proof}

	\subsection{Higher exact categories}\label{Subsection:d-ExactCat}
 Here we recall basic properties of $d$-exact categories, following \cite{Jas16}. Throughout we work in an additive category $\mc$. 

 A \textit{left $d$-exact sequence} is a complex $X_{d+1}\to X_d\to \cdots \to X_{0}$ in $\mc$ for which \begin{align*}
		 0\to \Hom_{\mc}(X,X_{d+1})\to \Hom_{\mc}(X,X_d)\to \cdots \to \Hom_{\mc}(X,X_{0}) 
  \end{align*}
  is exact for all $X\in \mc$. In this case $X_{d+1}\to \cdots \to X_{1}$ is called a \textit{$d$-kernel} of $X_1\to X_{0}$. A \textit{right $d$-exact sequence} is a complex $Y_{d+1}\to Y_d\to \cdots \to Y_{0}$ in $\mc$ for which
  \begin{align*}
      0\to \Hom_{\mc}(Y_{0},X)\to \cdots \to \Hom_{\mc}(Y_d,X) \to \Hom_{\mc}(Y_{d+1},X)
  \end{align*}
  is exact for all $X\in \mc$. In this case $Y_d\to \cdots \to Y_{0}$ is called a \textit{$d$-cokernel} of $Y_{d+1}\to Y_{d}$. A \textit{$d$-exact sequence} is a complex which is both a left and right $d$-exact sequence. For a class $\xc$ of $d$-exact sequences in $\mc$, we call its members  \textit{$\xc$-admissible $d$-exact sequences}. Given an $\xc$-admissible $d$-exact sequence
\[
X_0\to X_1\to \cdots \to X_d\to X_{d+1}
\]
the morphism $X_0\to X_1$ is called an $\xc$\textit{-admissible monomorphism} and the morphism $X_d\to X_{d+1}$ an $\xc$\textit{-admissible epimorphism}.

  Let $X_\bullet$ and $Y_\bullet$ be $d$-exact sequences. A \textit{morphism $h_\bullet\colon X_\bullet\to Y_\bullet$ of $d$-exact sequences} is just a morphism of complexes, i.e. a commutative diagram
   \begin{equation*}
	\begin{tikzcd}[column sep=15]
	X_{d+1}\arrow[r,""] \arrow[d,"h_{d+1}"] & X_d \arrow[r,""] \arrow[d,"h_d"] &  \cdots \arrow[r,""]  & X_{1} \arrow[r,""] \arrow[d,"h_{1}"] & X_0 \arrow[d,"h_{0}"]  \\
	 Y_{d+1} \arrow[r,""] & Y_{d} \arrow[r,""]  & \cdots \arrow[r,""] & Y_{1} \arrow[r,""]  & Y_{0}. 
	\end{tikzcd} 
	\end{equation*}
We say that $h_\bullet$ is a \textit{weak isomorphism} if there exists $0\leq i\leq d+1$ such that $h_{i+1}$ and $h_{i}$ are isomorphisms (where $h_{d+2}\colonequals h_0)$.

 Let $Y_\bullet=Y_d\xrightarrow{} Y_{d-1}\xrightarrow{} \cdots \xrightarrow{} Y_{0}$ be a complex in $\mc$. A $d$\textit{-pushout} of $Y_\bullet$ along a morphism $f\colon Y_d\to Z_d$ is a complex  $Z_\bullet= Z_d\xrightarrow{}Z_{d-1}\xrightarrow{}\cdots \xrightarrow{}Z_0$ in $\mc$ and a morphism of complexes $f_\bullet\colon Y_\bullet\to Z_\bullet$ satisfying $f_d=f$ and such that its cone
\[
Y_d\xrightarrow{} Y_{d-1}\oplus Z_d\xrightarrow{} \cdots \xrightarrow{}Y_0\oplus Z_{1}\xrightarrow{} Z_0
\]
is a right $d$-exact sequence. A $d$\textit{-pullback} of $Y_\bullet$ along a morphism $g\colon X_0\to Y_0$ is a complex $X_\bullet=X_d\xrightarrow{}X_{d-1}\xrightarrow{}\cdots \xrightarrow{}X_0$ in $\mc$ and a morphism of complexes $g_\bullet\colon X_\bullet\to Y_\bullet$ satisfying $g_0=g$ and such that its cone
\[
X_d\xrightarrow{} X_{d-1}\oplus Y_d\xrightarrow{} \cdots \xrightarrow{}X_0\oplus Y_{1}\xrightarrow{} Y_0
\]
is a left $d$-exact sequence.

\begin{definition}\cite[Definition 4.2]{Jas16}
    A \textit{$d$-exact structure} on $\mc$ is a class $\xc$ of $d$-exact sequences, closed under weak isomorphisms, and which satisfies the following axioms. 
    \begin{enumerate}
        \item[(E0)] The sequence $0\to 0\to \cdots \to 0$ with $(d+2)$-terms is in $\xc$.
        \item[(E1)] $\xc$-admissible monomorphisms are closed under composition.
        \item[(E1)$^{\operatorname{op}}$] $\xc$-admissible epimorphisms are closed under composition.
        \item[(E2)] Given a $\xc$-admissible $d$-exact sequence $Y_{d+1}\to \cdots \to Y_{0} $ and a morphism \mbox{$Y_{d+1}\to Z_{d+1}$}, there exists a $d$-pushout 
         \begin{equation*}
	\begin{tikzcd}[column sep=15]
	Y_{d+1}\arrow[r,""] \arrow[d] & Y_{d} \arrow[r,""] \arrow[d,""]  & \cdots \arrow[r,""] & Y_{2} \arrow[r,""] \arrow[d,""]& Y_1 \arrow[d,""]  \\
	 Z_{d+1} \arrow[r,""] & Z_{d} \arrow[r,""]   & \cdots \arrow[r,""] & Z_{2} \arrow[r,""] & Z_1 
	\end{tikzcd} 
	\end{equation*}
 where $Z_{d+1}\to Z_{d}$ is an $\xc$-admissible monomorphism.
 \item[(E2)$^{\operatorname{op}}$] Given a $\xc$-admissible $d$-exact sequence $Y_{d+1}\to \cdots \to Y_{0}$ and a morphism $X_{0}\to Y_{0}$, there exists a $d$-pullback 
         \begin{equation*}
	\begin{tikzcd}[column sep=15]
	X_d\arrow[r,""] \arrow[d,""] & X_{d-1} \arrow[r,""] \arrow[d,""] & \cdots \arrow[r,""] & X_{1} \arrow[d,""] \arrow[r,""] & X_{0} \arrow[d]  \\
	 Y_d \arrow[r,""] & Y_{d-1} \arrow[r,""] & \cdots \arrow[r,""] &  Y_{1}   \arrow[r,""] & Y_{0} 
	\end{tikzcd} 
	\end{equation*}
 where $X_1\to X_{0}$ is an $\xc$-admissible epimorphism.
    \end{enumerate}
A $d$\textit{-exact category} is an additive category together with a $d$-exact structure.  
\end{definition} 

We often omit the $d$-exact structure $\xc$ if it is clear from the context, and simply say admissible monomorphism, admissible epimorphisms, and admissible $d$-exact sequences. Note that a \mbox{$1$-exact} category is the same as an exact category in the sense of Quillen.

A \textit{$d$-exact functor} between $d$-exact categories $(\mc,\xc)$ and $(\nc,\yc)$ is a functor $F\colon \mc\to \nc$  which sends $\xc$-admissible $d$-exact sequences to $\yc$-admissible $d$-exact sequences. In other words, if $X_{d+1}\to \cdots \to X_{0}$ is an $\xc$-admissible $d$-exact sequence, then $F(X_{d+1})\to \cdots \to F(X_{0})$ is a $\yc$-admissible $d$-exact sequence. An \textit{equivalence of $d$-exact categories} is a $d$-exact functor $F\colon \mc\to \nc$  which has a quasi-inverse $G\colon \nc\to \mc$ which is also $d$-exact.

We need the following lemma on $d$-pushouts and $d$-pullbacks, cf. \cite[Proposition 2.12]{Bue10}.  

\begin{proposition}\label{Prop:dPushoutofdExact}
    Let $\mc$ be a $d$-exact category, let $X_{d+1}\to \cdots \to X_{0}$ be an admissible $d$-exact sequence. Then 
     \begin{equation*}
	\begin{tikzcd}[column sep=15]
	X_{d+1}\arrow[r,""] \arrow[d,""] & X_d \arrow[r,""] \arrow[d,""] & \cdots \arrow[r,""] & X_1 \arrow[d,""]  \\
	 Y_{d+1} \arrow[r,""] & Y_{d} \arrow[r,""]  & \cdots \arrow[r,""] & Y_1
	\end{tikzcd} 
	\end{equation*}
is a $d$-pushout if and only if there exists a morphism $Y_1\to X_{0}$ such that
\begin{equation*}
	\begin{tikzcd}[column sep=15]
	X_{d+1}\arrow[r,""] \arrow[d,""] & X_d \arrow[r,""] \arrow[d,""] & \cdots \arrow[r,""] & X_1 \arrow[d,""] \arrow[r]& X_{0} \arrow[d,equal]  \\
	 Y_{d+1} \arrow[r,""] & Y_{d} \arrow[r,""]  & \cdots \arrow[r,""] & Y_1\arrow[r]& X_{0}
	\end{tikzcd} 
	\end{equation*}
 is commutative, and the lower row is an admissible $d$-exact sequence.
\end{proposition}

\begin{proof}
This is well-known for $d=1$, see e.g. \cite[Proposition 2.12]{Bue10}, so we let $d\geq 2$. Assume we have a $d$-pushout as in the statement. We claim that $Y_{d+1}\to Y_d$ is an admissible monomorphism. Indeed,  there exists a $d$-pushout $X_\bullet\to Z_\bullet$ where $Z_{d+1}=Y_{d+1}$ and $Z_{d+1}\to Z_d$ is an admissible monomorphism by axiom (E2) for $d$-exact categories. Furthermore, by \cite[Proposition 2.13]{Jas16} there exists a morphism of complexes $Y_\bullet\to Z_{\bullet}$ where $Y_{d+1}\to Z_{d+1}$ is the identity. In particular, we have a commutative diagram
 \begin{equation*}
	\begin{tikzcd}[column sep=15]
	Y_{d+1}\arrow[r,""] \arrow[d,equal] & Y_d \arrow[d,""]  \\
	 Z_{d+1} \arrow[r,""] & Z_{d}  
	\end{tikzcd} 
	\end{equation*}
 Since $Z_{d+1}\to Z_d$ is an admissible monomorphism, the morphism $Y_{d+1}\to Y_d$ must be an admissible monomorphism by \cite[Theorem B]{Kla26}. This proves the claim. The remainder of the proposition then follows from \cite[Proposition 4.8]{Jas16}.
\end{proof}

The following result shows that $d$-cluster tilting subcategories are $d$-exact categories. Here we need to assume the ambient exact category is weakly idempotent complete for the description in \cite[Theorem 4.4]{Jas16} to hold. We discuss the non-weakly idempotent complete case in \Cref{Section:NonWeaklyIdemComp}.

\begin{theorem}\label{Theorem:dCTImpliesdExactWeakIdemPotent}\cite[Theorem 4.14]{Jas16}
Let $\mc$ be a $d$-cluster tilting subcategory of a weakly idempotent complete exact category $\ec$. Then the class of complexes 
    \[
0\to M_{d+1}\to \cdots \to M_{0}\to 0
    \]
    in $\mc$ which are acyclic in $\ec$ gives a $d$-exact structure on $\mc$. 
\end{theorem}

\begin{proof}
Let $\xc$ be the class of complexes in the theorem. By the proof of \mbox{\cite[Theorem 4.14]{Jas16}} the elements of $\xc$ are $d$-exact sequences, and $\xc$ satisfies axioms (E0)-(E2) and (E1)$^{\operatorname{op}}$-(E2)$^{\operatorname{op}}$. Hence, it only remains to show that $\xc$ is closed under weak isomorphisms. Let $h_\bullet\colon X_\bullet\to Y_\bullet$ be a weak isomorphism of $d$-exact sequences. It suffices to show that if $X_\bullet$ or $Y_\bullet$ is in $\xc$, then both of them are in $\xc$. We assume $Y_\bullet$ is in $\xc$, the other direction is proved dually. 

 By definition of weak isomorphisms we have that $h_i$ and $h_{i+1}$ are isomorphisms for $i\leq d+1$. If this holds for $i\leq d$, then $h_\bullet$ is a homotopy equivalence by \cite[Proposition 2.7]{Jas16}. Since $\ec$ is weakly idempotent complete, acyclic complexes are closed under homotopy equivalences.  Hence, $X_\bullet$ must be in $\xc$. 

It remains to prove that $X_\bullet$ is in $\xc$ when $h_{d+1}$ and $h_{d+2}\colonequals h_0$ are isomorphisms. In this case, we have a commutative diagram
\begin{equation*}
	\begin{tikzcd}[column sep=15]
	X_{d+1}\arrow[r,""] \arrow[d,"\cong"] & X_d \arrow[r,""] \arrow[d,""]  & \cdots \arrow[r,""] & X_1 \arrow[d,""] \arrow[r]& X_{0} \arrow[d,"\cong"]  \\
	 Y_{d+1} \arrow[r,""] & Y_{d} \arrow[r,""]  & \cdots \arrow[r,""] & Y_1\arrow[r]& Y_{0}
	\end{tikzcd} 
	\end{equation*}
 where the top row is a $d$-exact sequence and the bottom row is a $\xc$-admissible $d$-exact sequence. Since $\xc$ consists of acyclic complexes, the morphism $Y_{1}\to Y_{0}$ must be a deflation in $\ec$, and hence the composite $Y_1\to Y_{0}\xrightarrow{\cong}X_{0}$ must also be a deflation. Since $\ec$ is weakly idempotent complete, the morphism $X_1\to X_{0}$ is therefore a deflation, see \cite[Proposition 7.6]{Bue10}. Let $K$ denote its kernel. Since $\mc$ is $d$-cluster tilting, we can choose an exact sequence 
 \[
 0\to Z_{d+1}\to \cdots \to Z_{2}\to K\to 0
 \]
 with $Z_i\in \mc$ for all $i$, see \Cref{Reformulation:d-CT}. Then the induced complex
 \[
Z_{d+1}\to \cdots \to Z_{2}\to X_1\to X_{0}
 \]
 is acyclic in $\ec$, and is therefore a $\xc$-admissible $d$-exact sequence. Furthermore, we get a commutative diagram
 \begin{equation*}
	\begin{tikzcd}[column sep=15]
	 Z_{d+1} \arrow[d,dashed] \arrow[r,""] & Z_{d}  \arrow[d,dashed]\arrow[r,""] & \cdots \arrow[r,""] & Z_{2}\arrow[r]  \arrow[d,dashed]& X_1\arrow[r] \arrow[d,equal] & X_{0} \arrow[d,equal]  \\
	 X_{d+1} \arrow[r,""] & X_{d} \arrow[r,""] &  \cdots \arrow[r,""] & X_{2}\arrow[r]& X_1\arrow[r]& X_{0}
	\end{tikzcd} 
	\end{equation*}
  where the dashed arrows exist by the $d$-kernel property. This gives a homotopy equivalence by \cite[Proposition 2.7]{Jas16}, and so $X_\bullet$ must be contained in $\xc$. The claim follows.
\end{proof}

\begin{remark}\label{AdmMonoEpiWeakIdemComp}
    Let $\mc$ and $\ec$ be as in \Cref{Theorem:dCTImpliesdExactWeakIdemPotent}. Then a morphism in $\mc$ is an admissible monomorphism (resp. admissible epimorphism) if and only if it is an inflation (resp. deflation) in $\ec$. Similarly, a sequence
    \[
    0\to M_{d+1}\to M_d\to \cdots \to M_{0} \quad (\text{resp. } N_{d+1}\to \cdots \to N_1\to N_{0}\to 0)
    \]
    in $\mc$ is left $d$-exact (resp. right $d$-exact) if and only if it is exact in $\ec$.
\end{remark}

\subsection{Left exact functors}\label{Subsection:The category of left exact functors}

Here we recall properties of the category of left exact functors of a $d$-exact category, following \cite{Ebr21,EN23}. We start with localizations of abelian categories. 

Let $\ac$ be an abelian category. Fix a \textit{Serre subcategory} $\mathcal{S}$ of $\ac$, i.e. a subcategory closed under extensions, subobjects and quotients. Let $\ac/\mathcal{S}$ denote the localization of $\ac$ by the class of morphisms whose kernel and cokernel lies in $\mathcal{S}$. Then $\ac/\mathcal{S}$ is an abelian category and the localization functor $\ac\to \ac/\mathcal{S}$ is an exact functor, see \cite[Proposition 2.2.6]{Kra22}. The Serre subcategory $\mathcal{S}$ is called a \textit{localizing subcategory} if the functor $\ac\to \ac/\mathcal{S}$ has a right adjoint. In this case the right adjoint induces an equivalence between $\ac/\mathcal{S}$ and the subcategory
\[
\mathcal{S}^{\perp_1}=\{A\in \ac\mid \Hom_\ac(S,A)=0=\Ext^1_\ac(S,A) \text{ for all }S\in \mathcal{S}\}
\]
with quasi-inverse given by the composite $\mathcal{S}^{\perp_1}\to \ac\to \ac/\mathcal{S}$, see \cite[Lemma 2.2.10]{Kra22}. Note that if $\ac$ has injective envelopes and $\mathcal{S}$ is a localizing subcategory, then $\ac/\mathcal{S}\cong \mathcal{S}^{\perp_1}$ has injective envelopes and the inclusion $\mathcal{S}^{\perp_1}\to \ac$ preserves injective envelopes, see \cite[Corollary 2.2.15]{Kra22}. For more information on Serre and localizing subcategories we refer to \cite[Section 2.2]{Kra22}.
 
 Let $\operatorname{Ab}$ denote the category of abelian groups. Given an additive category $\mc$, the category of additive functors from $\mc^{\operatorname{op}}$ to $\operatorname{Ab}$ is denoted $\operatorname{Mod}\mc$. This is an abelian category with injective envelopes, and where kernels and cokernels are computed objectwise. The \textit{Yoneda embedding} is denoted 
 $$Y\colon\mc\to \operatorname{Mod}\mc \quad \quad X \mapsto \mc(-,X).$$ It is a fully faithful functor by the Yoneda Lemma.  A functor $F$ in $\operatorname{Mod}\mc$ is \textit{finitely presented} if we can find an exact sequence
 \[
 \mc(-,X)\to \mc(-,Y)\to F\to 0
 \]
 for some objects $X,Y\in \mc$. The subcategory of finitely presented functors is denoted $\operatorname{mod}\mc$. Note that $\operatorname{mod}\mc$ is an essentially small category if $\mc$ is an essentially small category. 
 
 Assume $\mc$ is a $d$-exact category. A \textit{left exact functor} $F\colon \mc^{\operatorname{op}}\to \operatorname{Ab}$ is an additive functor such that for any admissible $d$-exact sequence $X_{d+1}\xrightarrow{g_{d+1}} \cdots \xrightarrow{g_1} X_{0}$ 
the sequence
\[
0\to F(X_{0})\xrightarrow{F(g_1)}F(X_1)\xrightarrow{F(g_{2})}F(X_{2})
\]
is exact. The subcategory of left exact functors is denoted $\mathcal{L}(\mc)$. A \textit{weakly effaceable functor} $F\colon \mc^{\operatorname{op}}\to \operatorname{Ab}$ is an additive functor such that for any object $X\in \mc$ and any element $x\in F(X)$ there exists an admissible epimorphism $Y\xrightarrow{g} X$ satisfying $F(g)(x)=0$. The subcategory of weakly effaceable functors is denoted $\operatorname{Eff}(\mc)$.

\begin{proposition}\cite[Proposition 3.5]{Ebr21}\label{Proposition:Ebrahimi1}
    Let $\mc$ be a $d$-exact category. The following hold.
    \begin{enumerate}
    \item $\operatorname{Eff}(\mc)$ is a localizing subcategory of $\operatorname{Mod}\mc$.
    \item $\mathcal{L}(\mc)=\{F\in \operatorname{Mod}\mc\mid \Hom_{\operatorname{Mod}\mc}(\operatorname{Eff}(\mc),F)=0=\Ext^1_{\operatorname{Mod}\mc}(\operatorname{Eff}(\mc),F)\}$.
    \end{enumerate}
\end{proposition}

Note that the essential image of the Yoneda embedding $Y\colon \mc\to \operatorname{Mod}\mc$
must lie in $\mathcal{L}(\mc)$ by definition of right $d$-exact sequences. 

\begin{theorem}\label{Proposition:YonedaImageEbrahimi}
Let $\mc$ be a $d$-exact category. The following hold.
\begin{enumerate}
    \item\label{Proposition:YonedaImageEbrahimi:1} The essential image of $Y$ is $d$-rigid in $\mathcal{L}(\mc)$.
    \item\label{Proposition:YonedaImageEbrahimi:2}  The essential image of $Y$ in $\mathcal{L}(\mc)$ is closed under $d$-extensions.
    \item\label{Proposition:YonedaImageEbrahimi:3} A complex $X_{d+1}\to \cdots \to X_{0}$ in $\mc$ is an admissible $d$-exact sequence if and only if 
    \[
    0\to Y(X_{d+1})\to \cdots \to Y(X_{0})\to 0
    \]
    is an exact sequence in $\mathcal{L}(\mc)$.
\end{enumerate}
\end{theorem}

\begin{proof}
  Part \eqref{Proposition:YonedaImageEbrahimi:1} is \cite[Proposition 3.7 (ii)]{Ebr21}, part \eqref{Proposition:YonedaImageEbrahimi:2} is \cite[Theorem 3.5]{EN23}, while part \eqref{Proposition:YonedaImageEbrahimi:3} follows from \cite[Proposition 3.7 (i)]{Ebr21} and \cite[Proposition 3.2]{EN23}.
\end{proof}

\section{Maximal rigid subcategories}\label{Section:MaximalRigidSubcat}

\subsection{Properties and examples}
Fix an exact category $\ec$. Here we study maximal right (and left) $d$-rigid subcategories of $\ec$.  Their definition is a one-sided version of \Cref{Reformulation:d-CT}, and is weaker than being $d$-cluster tilting. However, they are still maximal rigid in the sense that there is no larger $d$-rigid subcategory containing them, see \Cref{Prop:RightMaxImplyLeftdCT} below.

		\begin{definition}\label{Definition:MaxRigid}
		Let $\uc$ be a $d$-rigid subcategory of $\ec$.
		\begin{enumerate}
			\item $\uc$ is called \textit{right maximal $d$-rigid} if it is closed under direct summands and for any $E\in \ec$ there exists an exact sequence
			\[
			0\to E\to U^1\to \cdots \to U^d\to 0
			\]
			in $\ec$ with $U_i\in \uc$ for $1\leq i\leq d$.
			\item $\uc$ is called \textit{left maximal $d$-rigid} if it is closed under direct summands and for any $E\in \ec$ there exists an exact sequence
			\[
			0\to U_d\to \cdots \to U_1\to E\to 0
			\]
			in $\ec$ with $U_i\in \uc$ for $1\leq i\leq d$.
		\end{enumerate}
  $\uc$ is called \textit{weakly left} or respectively, \textit{weakly right maximal $d$-rigid} if it satisfies the conditions above, but without necessarily being closed under direct summands.
  \end{definition}

  \begin{remark}
     A subcategory is $d$-cluster tilting if and only if it is left and right maximal $d$-rigid by \Cref{Reformulation:d-CT}.
 \end{remark}

  \begin{remark}
	$\uc$ is left maximal $d$-rigid if and only if $(\uc,\uc)$ is a left $(d-1)$-cotorsion pair in the sense of \cite{HM21a}. A similar statement holds for right maximal $d$-rigid and right $(d-1)$-cotorsion pairs.
\end{remark}

 \begin{example}\label{Example:LowGlobalDim}
 Let $k$ be a field, let $\Lambda$ be a finite-dimensional $k$-algebra, and let $\operatorname{proj}\Lambda$ be the category of finitely generated projective $\Lambda$-modules. Then $\operatorname{proj}\Lambda$ is left maximal $d$-rigid in $\operatorname{mod}\Lambda$ if and only if the global dimension of $\Lambda$ is $<d$, and $\operatorname{proj}\Lambda$ is right maximal $d$-rigid in $\operatorname{mod}\Lambda$ if and only if the global dimension of $\Lambda$ is $0$.
 \end{example}

\begin{example}\label{Example:VosnexProperty}
    Let $k$ be a field, and let $\Lambda$ be a finite-dimensional  $k$-algebra of global dimension $d$ with the vosnex property \cite[Notation 3.5]{IO13}. Then the subcategory
    \[
    \hat{\Lambda}\colonequals \add\{\tau_d^{-i}(\Lambda)\mid i\geq 0\}
    \]
    of $\operatorname{mod} \Lambda$ is left maximal $d$-rigid by \cite[Proposition 5.3 and Lemma 5.5]{IO13}, where $$\tau_d^-=\Ext^d_\Lambda(D(\Lambda),-)$$ is 
    the inverse $d$-Auslander--Reiten translation and $D(\Lambda)=\Hom_k(\Lambda,k)$. 
\end{example}

\begin{example}\label{Example:SubCatsOfPreProj}
     Let $\Pi=\Pi(Q)$ be the completion of the preprojective algebra of a non-Dynkin quiver $Q$ with vertices $1,2,\dots,n$. For each vertex $i\in \{1,2,\dots,n\}$ let $e_i$ denote the corresponding idempotent and let $I_i=\Pi(1-e_i)\Pi$ be the corresponding ideal of $\Pi$. 

     Assume $w=s_{i_1}s_{i_2}\cdots s_{i_k}$ is a reduced presentation of an element in the Weyl-group $W$ of $Q$, and let $I_w=I_{i_1}I_{i_2}\dots  I_{i_k}$ be the corresponding product of ideals. Then the algebra $\Pi/I_w$ is finite-dimensional by \cite[Theorem III.1.6]{BIRS09} and 
     \[
     \mc_w=\operatorname{add}\{\Pi/(I_{i_1}I_{i_2}\dots I_{i_j})\mid 0<j\leq k\} 
     \]
     is a left maximal $2$-rigid in $\operatorname{mod}\Pi/I_w$ by \cite[Lemma III.2.1 and Proposition III.2.5 (a)]{BIRS09}.

     Assume $s_{i_1},s_{i_2},\dots$ is an infinite sequence of reflections in the Weyl group $W$ of $Q$ such that the product $s_{i_1}s_{i_2}\cdots s_{i_k}$ is a reduced presentation of an element in $W$ for each $k$, and such that each number in $\{1,2,\dots,n\}$ occurs an infinite number of times as an index. Then the additive closure of the infinite sum
     \[
     \mc=\operatorname{add}\{\Pi/(I_{i_1}I_{i_2}\dots I_{i_j})\mid j>0\}
     \]
    is a left maximal $2$-rigid subcategory in the category $\operatorname{fl}\Pi$ of finite-dimensional modules over $\Pi$ by the proof of \cite[Theorem III.2.7]{BIRS09}. 

    It is not clear that $\mc_w$ or $\mc$ above are $2$-cluster tilting in $\operatorname{mod}\Pi/I_w$ or $\operatorname{fl}\Pi$, respectively, since they are not shown to be cogenerating. However, the subcategory
    \[
    \operatorname{Sub}\Pi/I_w\colonequals\{X\in \operatorname{mod}\Pi/I_w\mid \text{there exists a monomorphism }X\to P\text{ with }P\in \proj \Pi/I_w\}
    \]
    is extension closed in $\operatorname{mod}\Pi/I_w$ and $\mc_w$ is a $2$-cluster tilting subcategory in $\operatorname{Sub}\Pi/I_w$, by \cite[Theorem III.2.6]{BIRS09}. It follows that 
    \[
    \operatorname{Sub}\mc\colonequals\{X\in \operatorname{fl}\Pi\mid \text{there exists a monomorphism }X\to M\text{ with }M\in \mc\}
    \]
    is extension closed in $\operatorname{fl}\Pi$ and $\mc$ is $2$-cluster tilting in $\operatorname{Sub}\mc$, since  
    \[
    \mc=\cup_{j>0}\mc_{w_j} \quad \text{and} \quad \operatorname{Sub}\mc=\cup_{j>0}\operatorname{Sub}\Lambda/I_{w_j}
    \] 
    as subcategories of $\operatorname{fl}\Pi$ where $w_j=s_{i_1}s_{i_2}\cdots s_{i_j}$.
 \end{example}

Recall that $\uc$ is closed under additive complements if for any $W\in \ec$ and $U\in \uc$ such that $W\oplus U\in \uc$, then $W\in \uc$. 
The following lemma shows that we can replace being closed under direct summands with this weaker assumption in \Cref{Definition:MaxRigid}.

\begin{lemma}\label{Lemma:UClosedUnderDirectSummand}
	Assume $\uc$ is weakly right maximal $d$-rigid and closed under additive complements in $\ec$. Then $\uc$ is closed under direct summands, and is therefore right maximal $d$-rigid.
\end{lemma}

\begin{proof}
Assume $E\oplus E'\in \uc$. We want to show that $E,E'\in \uc$. Choose an exact sequence
	\[
	0\to E\to U^1\to \cdots \to U^{d-1}\to U^d\to 0
	\]
	and let $E^j$ be the kernel of $U^j\to U^{j+1}$ for $0<j<d$, so that $E^1=E$. Hence, we get conflations
	\[
	0\to E^j\to U^j\to E^{j+1}\to 0
	\]
	for $0<j<d$, where $E^d\colonequals U^d$. Since $E\oplus E'\in \uc$ and $\Ext^i_\ec(\uc,\uc)=0$ for $0<i<d$, it follows that $\Ext^i_\ec(\uc,E)=0$ for $0<i<d$. Hence applying $\Hom_\ec(U^d,-)$ to the top sequence gives an exact sequence
	\[
	0\to \Hom_{\ec}(U^d,E)\to \Hom_{\ec}(U^d,U^1)\to \cdots \to \Hom_{\ec}(U^d,U^{d-1})\to \Hom_{\ec}(U^d,U^d)\to 0.
	\]
	In particular, since the rightmost map is surjective, the conflation
	\[
	0\to E^{d-1}\to U^{d-1}\to U^d\to 0
	\]
	must be split. Hence, $E^{d-1}\in \uc$ since $\uc$ is closed under additive complements. Thus, we have an exact sequence
	\[
	0\to E\to U^1\to \cdots \to U^{d-2}\to E^{d-1}\to 0
	\]
	where all terms are in $\uc$ except possibly $E$. Repeating this procedure, we can show recursively for $i\geq 1$ that $E^{d-i}\in \uc$. In particular, for $i=d-1$ this gives that $E\in \uc$, which proves the claim.  
\end{proof}

The following result shows that right maximal $d$-rigid implies left $d$-cluster tilting in the sense of \cite[Definition 5.1]{IO13}. The dual statement on left maximal $d$-rigid and right $d$-cluster tilting also holds. The proof is similar to the last part of the proof of \cite[Theorem 5.2]{IO13}. 

\begin{proposition}\label{Prop:RightMaxImplyLeftdCT}
		If $\uc$ is right maximal $d$-rigid, then 
		\[
		\uc=\{E\in \ec\mid \Ext^i_\ec(\uc,E)=0 \text{ for all }0<i<d\}.
		\]
\end{proposition}

\begin{proof}
	 Assume $E\in \ec$ satisfies $\Ext^i_\ec(\uc,E)=0$ for all $0<i<d$.  Choose an exact sequence
	\[
	0\to E\to U^1\to \cdots \to U^{d-1}\to U^d\to 0
	\]
	in $\ec$ with $U^i\in \uc$ for $1\leq i\leq d$. Applying $\Hom_{\ec}(U^d,-)$ to this and using that $\Ext^i_\ec(U^d,E)=0$ for $0<i<d$, we get an exact sequence
	\[
	0\to \Hom_\ec(U^d,E)\to \Hom_\ec(U^d,U^1)\to \cdots \to \Hom_\ec(U^d,U^{d-1})\to \Hom_\ec(U^d,U^d)\to 0.
	\]
	Hence, the deflation $U^{d-1}\to U^d$ is split, and so its kernel must be in $\uc$ since $\uc$ is closed under direct summands. Repeating the argument in a similar way as in the proof of \Cref{Lemma:UClosedUnderDirectSummand}, we get that $E\in \uc$. This shows that
	$$\{E\in \ec\mid \Ext^i_\ec(\uc,E)=0 \text{ for all }0<i<d\}\subseteq \uc.$$
	Since the other inclusion is clear, the claim follows.
\end{proof}

\begin{remark}
  Let $\Lambda$ be an algebra with the vosnex property as in \Cref{Example:VosnexProperty}.  By the dual of \Cref{Prop:RightMaxImplyLeftdCT} the subcategory $\hat{\Lambda}\colonequals \add\{\tau_d^{-i}(\Lambda)\mid i\geq 0\}$ is right $d$-cluster tilting, i.e.
    \[
    \hat{\Lambda}=\{E\in \operatorname{mod}\Lambda\mid \Ext^i_\Lambda(E,\hat{\Lambda})=0 \text{ for all }0<i<d\}.
    \]
    This also follows from \cite[Theorem 5.2]{IO13}.
\end{remark}

\subsection{Realizing rigid subcategories as maximal rigid}\label{Subsection:FiniteRes&Cores}
	 
	 Fix $\uc$ to be a $d$-rigid subcategory of $\ec$, and let 
	\begin{align*}
	\uc^{j}(\ec):=\{E\in \ec \mid \exists\text{ }0\to E\to U^1\to \cdots \to U^j\to 0 \text{ exact, }U^i\in \uc \text{, }1\leq i\leq j\} \\
	\uc_{j}(\ec):=\{E\in \ec \mid \exists\text{ }0\to U_j\to \cdots \to U_1\to E\to 0 \text{ exact, }U_i\in \uc \text{, }1\leq i\leq j\}.
	\end{align*}
The goal of this subsection is to investigate when $\uc$ is right maximal $d$-rigid in $\uc^{d}(\ec)$. The main challenge involves determining when $\uc^{d}(\ec)$ is closed under extensions in $\ec$, so that it becomes an exact category. We first show that this always holds for $j<d$.
	
	\begin{lemma}\label{Lemma:ExtClosureProperty-Rigid} $\uc^{j}(\ec)$ is closed under extensions in $\ec$ for $j<d$. 
	\end{lemma}
	
	\begin{proof}
Follows from the dual of \cite[Lemma 3.7]{HM21a}, since $\Ext^1_\ec(\uc^j(\ec),\uc)=0$ for $j<d$.
	\end{proof}
	
	\begin{lemma}\label{ClosureKernelDeflation}
		Assume $d\geq 2$, and let $$0\to E_3\to E_2\to E_1\to 0$$ be a conflation with $E_2\in \uc^{d}(\ec)$ and $E_1\in \uc^{d-1}(\ec)$. Then $E_3\in \uc^{d}(\ec)$.
	\end{lemma}
	
	\begin{proof}
		Choose an exact sequence
		\[
		0\to E_2\to U^1\to \cdots \to U^d\to 0
		\]
		with $U^1,\dots, U^{d}\in \uc$. Let $C$ be the cokernel of $E_2\to U^1$, and let $F$ be the cokernel of the composite $E_3\to E_2\to U^1$. Then by Noether's third isomorphism theorem we get an exact sequence
		\[
		0\to E_1\to F\to C\to 0
		\]
		see \cite[Lemma 3.5]{Bue10}. By construction $C\in \uc^{d-1}(\ec)$ and by assumption $E_1\in \uc^{d-1}(\ec)$. Hence $F\in \uc^{d-1}(\ec)$ by \Cref{Lemma:ExtClosureProperty-Rigid}, and so $E_3\in \uc^{d}(\ec)$. This proves the claim.
	\end{proof}

 \begin{proposition}\label{Prop:ClosusureAdditiveComplements}
    Assume $d\geq 2$. Then $\uc^d(\ec)$ is closed under additive complements in $\ec$.
\end{proposition}

\begin{proof}
    Assume we have objects $E\in \ec$ and $F\in \uc^d(\ec)$ such that $E\oplus F\in \uc^d(\ec)$. By definition, we can find a conflation
    \[
    0\to F\to U\to F'\to 0
    \]
    with $U\in \uc$ and $F'\in \uc^{d-1}(\ec)$. Adding the trivial conflation $0\to E\xrightarrow{1}E\to 0\to 0$ to it, we get a conflation
    \[
    0\to E\oplus F\to E\oplus U\to F'\to 0.
    \]
    Now since $E\oplus F\in \uc^d(\ec)$, we can find a conflation $0\to E\oplus F\to V\to E'\to 0$ with $V\in \uc$ and $E'\in \uc^{d-1}(\ec)$. Taking the pushout of $E\oplus F\to E\oplus U$ along $E\oplus F\to V$, we get a commutative diagram
    \[
		\begin{tikzcd}[column sep=20, row sep=20]
		0\arrow[r] &E\oplus F \arrow[r] \arrow[d] & E \oplus U\arrow[r] \arrow[d] & F' \arrow[d,equal]\arrow[r] & 0 \\
		0\arrow[r] &V \arrow[r, ""] & G \arrow[r]  & F' \arrow[r]  & 0
		\end{tikzcd}
		\]
  where the rows are conflations. Since $\uc^{d-1}(\ec)$ is closed under extensions by \Cref{Lemma:ExtClosureProperty-Rigid}, it follows that $G\in \uc^{d-1}(\ec)$. Now since the cokernel of $E\oplus U\to G$ is equal to the cokernel of $E\oplus F\to V$, we have a conflation
  \[
  0\to E\oplus U\to G\to E'\to 0.
  \]
  Since $G\in \uc^{d-1}(\ec)$ and $E'\in \uc^{d-1}(\ec)$, it follows from \Cref{ClosureKernelDeflation} that $E\oplus U\in \uc^d(\ec)$. Now consider the (split) conflation
  \[
  0\to E\to E\oplus U\to U\to 0.
  \]
  Since $U\in \uc\subseteq \uc^{d-1}(\ec)$ and $E\oplus U\in \uc^d(\ec)$, we get that $E\in \uc^d(\ec)$ by \Cref{ClosureKernelDeflation}.
\end{proof}	

Next we turn to the question of when $\uc^d(\ec)$ is closed under extensions. We first show that if it holds, then we obtain a category where $\uc$ is maximal $d$-rigid. We need the following lemma. 

\begin{lemma} \label{lem: replaceterms}
	Assume $\uc$ is a cogenerating subcategory of $\ec$, and let $j\geq 1$ be an integer. Then any exact sequence 
	\begin{align*}
	0 \to E_{j+1} \xrightarrow{} E_j \xrightarrow{} E_{j-1} \to \dots \to E_{2} \to E_1 \xrightarrow{} E_{0} \to 0
	\end{align*}
 of length $j$ is Yoneda equivalent to an exact sequence 
	\begin{align*}
	0 \to E_{j+1} \xrightarrow{} U_j \xrightarrow{} U_{j-1} \to \dots \to U_{2} \to X \xrightarrow{} E_{0} \to 0
	\end{align*}
	where $U_2, \dots, U_{j} \in \uc$.
\end{lemma}

\begin{proof}
We claim that any exact sequence
\begin{align}\label{ExactSeq}
0 \to E_{j+1} \xrightarrow{} U_j \xrightarrow{} U_{j-1} \to \dots \to U_{j-i+1}\to E'_{j-i}\to E_{j-i-1}\to \cdots \xrightarrow{} E_{0} \to 0
\end{align}
where $0\leq i<j-1$ and $U_j,\dots, U_{j-i+1}\in \uc$ is Yoneda equivalent to an exact sequence 
\begin{align*}
0 \to E_{j+1} \xrightarrow{} U_j \xrightarrow{} U_{j-1} \to \dots \to U_{j-i+1}\to U_{j-i}\to E'_{j-i-1}\to E_{j-i-2}\to \cdots  \xrightarrow{} E_{0} \to 0
\end{align*}
where also $U_{j-i}\in \uc$.
Starting with an exact sequence
\begin{align*}
	0 \to E_{j+1} \xrightarrow{} E_j \xrightarrow{} \dots \to E_{1}\xrightarrow{} E_{0} \to 0
	\end{align*}
and applying the claim iteratively for $i=0,1,\dots,j-2$ gives the result.

To prove the claim, assume we are given an exact sequence as in \eqref{ExactSeq}, and let $K$ and $C$ be the kernel and cokernel of $E_{j-i}'\to E_{j-i-1}$. Since $\uc$ is cogenerating, we can choose a conflation $0 \to E'_{j-i} \xrightarrow{} U_{j-i} \to C' \to 0$ with $U_{j-i}\in \uc$. Taking the pushout $E_{j-i-1}'$ of $E'_{j-i} \xrightarrow{} E_{j-i-1}$ along $E'_{j-i} \xrightarrow{} U_{j-i}$ yields a commutative diagram 
	\begin{equation*}
	\begin{tikzcd}[column sep=15]
	0 \arrow[r] & K\arrow[r] \arrow[d, equal] & E_{j-i}' \arrow[r] \arrow[d,hook] & E_{j-i-1}\arrow[r,""] \arrow[d, hook] & C \arrow[r] \arrow[d,equal] & 0 \\
	0 \arrow[r,""] & K \arrow[r] & U_{j-i} \arrow[r] \arrow[d, two heads] & E'_{j-i-1} \arrow[r] \arrow[d, two heads] & C \arrow[r] & 0 \\
	& & C' \arrow[r,equal] & C' & &  
	\end{tikzcd} 
	\end{equation*}
	with rows being exact and the two middle columns being conflations. By attaching the exact sequences in the two top rows to the exact sequences  $0\to E_{j+1}\to U_j\to \cdots \to U_{j-i+1}\to K\to 0$ and $0\to C\to E_{j-i-2}\to \cdots \to E_{0}\to 0$, we get a commutative diagram 
	\begin{equation*}
	\begin{tikzcd}[column sep=10]
	0 \arrow[r] & E_{j+1}\arrow[r] \arrow[d,equal] & U_j \arrow[r,] \arrow[d,equal] & \cdots \arrow[r] & U_{j-i+1} \arrow[r,""] \arrow[d,equal] & E_{j-i}' \arrow[r] \arrow[d,hook] & E_{j-i-1} \arrow[r] \arrow[d,hook]& E_{j-i-2} \arrow[r] \arrow[d,equal] & \cdots \arrow[r] & E_{0} \arrow[r] \arrow[d,equal] & 0 \\
	0 \arrow[r,""] & E_{j+1} \arrow[r] & U_{j} \arrow[r] & \cdots \arrow[r] & U_{j-i+1} \arrow[r]  & U_{j-i}\arrow[r] & E_{j-i-1}'\arrow[r]& E_{j-i-2} \arrow[r]  & \cdots \arrow[r] &  E_{0} \arrow[r] & 0.
	\end{tikzcd} 
	\end{equation*}
	This proves the claim.
\end{proof}

\begin{theorem}\label{Theorem:dExtImpliesRightMaxdRigid}
	 If $\uc^d(\ec)$ is extension-closed, then $\uc$ is weakly right maximal $d$-rigid in $\uc^d(\ec)$.
	 \end{theorem}
 
 \begin{proof}
 For any object $E$ in $\uc^d(\ec)$ we have an exact sequence 
 	\[
 	0\to E\to U^1\to \cdots \to U^d\to 0
 	\]
 	in $\uc^d(\ec)$ where $U^1,\dots, U^d\in \uc$, by definition of $\uc^d(\ec)$. Hence, we only need to show $d$-rigidity of $\uc$ in $\uc^d(\ec)$. Let $U,V\in \uc$, let $j$ be an integer satisfying $0<j<d$, and let $\delta$ be an element in $\Ext^j_{\uc^d(\ec)}(V,U)$ . Then $\delta$ can be represented by an exact sequence in $\uc^d(\ec)$ of the form
 	\begin{equation*}
 	0 \to U \xrightarrow{} U_j \xrightarrow{} U_{j-1} \to \dots \to U_{2} \to X \xrightarrow{f} V \to 0
 	\end{equation*}
 	where $U_2, \dots, U_{j} \in \uc$ by \cref{lem: replaceterms}. Now $d$-rigidity of $\uc$ in $\ec$ and a dimension shifting argument applied to 
 	\[
 	0\to U\to U_j\to \cdots \to U_{2}\to \operatorname{Ker}f\to 0
 	\]
 	shows that $\Ext^1_\ec(V,\Ker f)=0$. Hence the conflation 
    \[
    0\to \operatorname{Ker}f \to X\to V\to 0
    \]
    is split exact, and thus $\delta=0$. Since $\delta$ was arbitrary this shows that $\Ext^{j}_{\uc^d(\ec)}(V,U)=0$, which proves the claim.
 \end{proof}

To determine when $\uc^d(\ec)$ is extension-closed, we need the following lemma. 

	\begin{lemma}\label{Lemma:YonedaEquivMorphism}
		Let 
		\[
		0\to U\to E_d\to U_{d-1}\to \dots \to U_1\to U'\to 0
		\]
		be an exact sequence where $U',U_1,\dots, U_{d-1},U\in \uc$. Assume this sequence is Yoneda equivalent to an exact sequence
		\[
		0\to U\to V_d\to \dots \to V_1\to U'\to 0
		\]
		where $V_1,\dots, V_d\in \uc$. Then there exists morphisms $E_d\to V_d$ and $U_i\to V_i$ for $1\leq i\leq d-1$ making the diagram
		\begin{equation*}
		\begin{tikzcd}
		0 \arrow[r] & U\arrow[r] \arrow[d, equal] & E_d \arrow[r,""] \arrow[d, ""] & U_{d-1} \arrow[r,""] \arrow[d, ""] & \cdots \arrow[r,""] & U_1\arrow[r,""] \arrow[d,""] & U'\arrow[r] \arrow[d, equal] & 0 \\
		0 \arrow[r,""] & U \arrow[r,""] & V_d \arrow[r,""] & V_{d-1} \arrow[r,""] & \cdots  \arrow[r,""] & V_1 \arrow[r,""] & U' \arrow[r] & 0
		\end{tikzcd} 
		\end{equation*}
		commutative.
	\end{lemma}
	
	\begin{proof}
		To simplify notation we write  $(A,B):=\operatorname{Hom}_\ec(A,B)$ and $(A,B)^d:=\operatorname{Ext}^d_\ec(A,B)$ for the Hom and Ext-groups in $\ec$. Let $p\colon U_1\to U'$ and $q\colon V_1\to U'$ denote the rightmost morphisms in the exact sequences. 
		
		Since the two exact sequences in the lemma are Yoneda equivalent, there must exist a commutative diagram with exact rows
		\begin{equation*}
		\begin{tikzcd}
		0 \arrow[r] & U\arrow[r] & E_d \arrow[r,""]  & U_{d-1} \arrow[r,""] & \cdots \arrow[r,""] & U_1\arrow[r,"p"] & U'\arrow[r]  & 0 \\
		0 \arrow[r,""] & U \arrow[u,equal] \arrow[d,equal] \arrow[r,""] & F_d \arrow[u,"\phi_d"] \arrow[d,"\psi_d"] \arrow[r,""] & F_{d-1} \arrow[u,"\phi_{d-1}"] \arrow[d,"\psi_{d-1}"] \arrow[r,""] & \cdots  \arrow[r,""] & F_1 \arrow[u,"\phi_1"] \arrow[d,"\psi_1"] \arrow[r,""] & U' \arrow[r] \arrow[u,equal] \arrow[d,equal] & 0 \\
		0 \arrow[r,""] & U \arrow[r,""] & V_d \arrow[r,""]  & V_{d-1} \arrow[r,""]  & \cdots  \arrow[r,""]  & V_1 \arrow[r,"q"] & U'  \arrow[r]  & 0
		\end{tikzcd} 
		\end{equation*}
		where $F_1,\cdots,F_d\in \ec$, see e.g \cite[Proposition A.13]{Pos11}. Applying $(-,V_d)$ to the top part of the diagram, we get the following morphism of complexes
		\begin{equation*}
		\begin{tikzcd}[column sep=10]
		(U',V_d)^d \arrow[d,equal]  & (U,V_d)\arrow[l] \arrow[d,equal] & (E_d,V_d) \arrow[l,""] \arrow[d,"-\circ \phi_d"]  & (U_{d-1},V_d) \arrow[l,""] \arrow[d,"-\circ \phi_{d-1}"] & \cdots \arrow[l,""] & (U_1,V_d)\arrow[l,""] \arrow[d,"-\circ \phi_1"]  & (U',V_d)\arrow[l] \arrow[d,equal] & \arrow[l] 0 \\
		(U',V_d)^d  & (U,V_d) \arrow[l,""] & (F_d,V_d) \arrow[l,""] & (F_{d-1},V_d) \arrow[l,""] & \cdots  \arrow[l,""] & (F_1,V_d) \arrow[l,""]  & (U',V_d) \arrow[l] & \arrow[l] 0. 
		\end{tikzcd} 
		\end{equation*}
		Here $(U,V_d)\to (U',V_d)^d$ comes from the natural transformation $(U,-)\to (U',-)^d$ determined by the Yoneda equivalence class of our exact sequences. Note that the top row is exact since $\operatorname{}^i(\uc,V_d)=0$ for $0<i<d$. The bottom row is a complex, but is in general not exact.
		
		Now consider the element $\psi_d\in (F_d,V_d)$. Since the bottom row is a complex, this map gets sent to zero under the composite $(F_d,V_d)\to (U,V_d)\to (U',V_d)^d$. Since the top row is exact at $(U,V_d)$, there must exist an element $\gamma_d\in (E_d,V_d)$ whose image in $(U,V_d)$ is the same as $\psi_d$. This is equivalent to the leftmost square of the diagram
		\begin{equation}\label{Equation:ImportantCommDiag}
		\begin{tikzcd}
		0 \arrow[r] & U\arrow[r] \arrow[d, equal] & E_d \arrow[r,""] \arrow[d, "\gamma_d"] & U_{d-1} \arrow[r,""] \arrow[d,dashed, "\gamma_{d-1}"] & \cdots \arrow[r,""] & U_1\arrow[r,"p"] \arrow[d,dashed,"\gamma_1"] & U'\arrow[r] \arrow[d, dashed, "\gamma_{0}"] & 0 \\
		0 \arrow[r,""] & U \arrow[r,""] & V_d \arrow[r,""] & V_{d-1} \arrow[r,""] & \cdots  \arrow[r,""] & V_1 \arrow[r,"q"] & U' \arrow[r] & 0
		\end{tikzcd} 
		\end{equation}
		being commutative, where the two rows are our exact sequences. Next we want to construct the morphisms $\gamma_i$ for $d-1\geq i\geq 0$ making the diagram commutative as indicated. First note that we have an exact sequence
		\begin{equation*}
		\begin{tikzcd}[column sep=10]
		(U,V_{d-1}) & (E_d,V_{d-1}) \arrow[l,""]  & (U_{d-1},V_{d-1}) \arrow[l,""]& \cdots \arrow[l,""] & (U_1,V_{d-1})\arrow[l,""]  & (U',V_{d-1})\arrow[l,"-\circ p"] & \arrow[l] 0 
		\end{tikzcd} 
		\end{equation*}
		obtained by applying $(-,V_{d-1})$ to the top sequence in \eqref{Equation:ImportantCommDiag}. Since the composite \mbox{$E_d\xrightarrow{\gamma_d}V_d\to V_{d-1}$} lies in the kernel of $(E_d,V_{d-1})\to (U,V_{d-1})$, it must be in the image of \mbox{$(U_{d-1},V_{d-1})\to (E_d,V_{d-1})$}. In other words, there must exist $\gamma_{d-1}\colon U_{d-1}\to V_{d-1}$ making the second leftmost square in \eqref{Equation:ImportantCommDiag} commutative. Repeating this argument, we construct morphisms $\gamma_i\colon U_i\to V_i$ for $d-1\geq i\geq 1$ and $\gamma_{0}\colon U'\to U'$ making the diagram \eqref{Equation:ImportantCommDiag} commutative. 
  
  Now let $$(-,\gamma_{0})^d\colon {}^d(U',U)\to {}^d(U',U)$$ be the induced morphism on Ext-groups. If $\delta\in {}^d(U',U)$ denotes the Yoneda equivalence class of our exact sequence, then the existence of the commutative diagram \eqref{Equation:ImportantCommDiag} implies that $(-,\gamma_{0})^d(\delta)=\delta$, see e.g. \cite[Chapter III, Proposition 5.1]{Maclane63}. 
		
		Next apply $(U',-)$  to the lower exact sequence in \eqref{Equation:ImportantCommDiag}. Since ${}^i(U',\uc)=0$ for $0<i<d$, we get an exact sequence
		\[
		0\to (U',U)\to (U',V_d)\to (U',V_{d-1})\to \cdots \to (U',V_1)\xrightarrow{q \circ -} (U',U')\to (U',U)^d
		\]
		where the rightmost morphism sends a map $f\colon U'\to U'$ to $(-,f)^d(\delta)$. Now since $(-,\gamma_{0})^d(\delta)=\delta$, the morphism $1_{U'}-\gamma_{0}$ gets sent to $0$ by the map $(U',U')\to (U',U)^d$. Hence, there exists a morphism $\alpha\colon U'\to V_1$ such that $\gamma_{0}+q\circ \alpha=1_{U'}$. Therefore, if we set $\gamma_i'=\gamma_i$ for $2\leq i\leq d$ and $\gamma_1'=\gamma_1+\alpha\circ p$, we get a commutative diagram
		\begin{equation*}
		\begin{tikzcd}
		0 \arrow[r] & U\arrow[r] \arrow[d, equal] & E_d \arrow[r,""] \arrow[d, "\gamma_d'"] & U_{d-1} \arrow[r,""] \arrow[d, "\gamma_{d-1}'"] & \cdots \arrow[r,""] & U_1\arrow[r,"p"] \arrow[d,"\gamma_1'"] & U'\arrow[r] \arrow[d, equal] & 0 \\
		0 \arrow[r,""] & U \arrow[r,""] & V_d \arrow[r,""] & V_{d-1} \arrow[r,""] & \cdots  \arrow[r,""] & V_1 \arrow[r,"q"] & U' \arrow[r] & 0.
		\end{tikzcd} 
		\end{equation*}
		This proves the claim.
	\end{proof}

	We now give a characterization of when $\uc^{d}(\ec)$ is closed under extensions.
	
	\begin{theorem} \label{thm: extclosed}
	The subcategory	$\uc^{d}(\ec)$ is closed under extensions if and only if any exact sequence
		\[
		0\to U\to E\to U_{d-1}\to \cdots \to U_1\to U'\to 0
		\]
		with $U,U_1,\dots, U_{d-1},U'\in \uc$ is Yoneda equivalent to an exact sequence with all terms in $\uc$.
	\end{theorem}
	
	\begin{proof}
	This is obviously true if $d=1$, so we assume $d\geq 2$.
 
 We prove the "if" direction. First we show that if $0\to U\to E_2'\to E_1'\to 0$ is a conflation with $U\in \uc$ and $E_1'\in \uc^{d}(\ec)$, then $E_2'\in \uc^{d}(\ec)$. To this end, choose an exact sequence $0\to E_1'\to U^1\to U^2\to \cdots \to U^d\to 0$ with $U^i\in \uc$ for $1\leq i\leq d$. Then we have an exact sequence
		\[
		0\to U\to E_2'\to U^1\to U^2\to \cdots \to U^d\to 0
		\]
		By assumption and \cref{Lemma:YonedaEquivMorphism} we have a commutative diagram
		\begin{equation*}
		\begin{tikzcd}
		0 \arrow[r] & U\arrow[r] \arrow[d, equal] & E_2' \arrow[r,""] \arrow[d, ""] & U^1 \arrow[r,""] \arrow[d, ""] & \cdots \arrow[r,""] & U^{d-1}\arrow[r,""] \arrow[d,""] & U^d\arrow[r] \arrow[d, equal] & 0 \\
		0 \arrow[r,""] & U \arrow[r,""] & V_{d} \arrow[r,""] & V_{d-1} \arrow[r,""] & \cdots  \arrow[r,""] & V_1 \arrow[r,""] & U^d \arrow[r] & 0
		\end{tikzcd} 
		\end{equation*}
		with exact rows and where $V_i \in \uc$ for $i=1,\dots, d$. If we consider this diagram as a morphism of complexes, then its cone must be acyclic by  \cite[Lemma 1.1]{Nee90}. Removing the identity maps at $U$ and $U_d$ in the cone, we get the complex 
		\[
		0\to E_2'\to U^1\oplus V_d\to \cdots \to U^{d-1}\oplus V_{2}\to V_1\to 0.
		\]
		which also must be acyclic. This shows that $E_2'\in \uc^{d}(\ec)$.
		
		Now assume $0\to E_3\to E_2\to E_1\to 0$ is a conflation with $E_1,E_3\in \uc^{d}(\ec)$. Choose a conflation
		\[
		0\to E_3\to U\to F\to 0
		\]
		with $U\in \uc$ and $F\in \uc^{d-1}(\ec)$. Taking the pushout of $E_3\to U$ along $E_3\to E_2$, we get a commutative diagram
		\[
		\begin{tikzcd}[column sep=20, row sep=20]
		& 0\arrow[d]  & 0\arrow[d] & \\
		0\arrow[r] &E_3 \arrow[r] \arrow[d, ""] & E_2 \arrow[r] \arrow[d] & E_1 \arrow[d, equal]\arrow[r] & 0 \\
		0\arrow[r] &U \arrow[r, ""] \arrow[d] & E \arrow[r] \arrow[d] & E_1 \arrow[r] & 0 \\
		& F \arrow[r, equal] \arrow[d] & F \arrow[d] & \\ 
		& 0  & 0 &
		\end{tikzcd}
		\]
		where the rows and columns are conflations. Since $U\in \uc$ and $E_1\in \uc^{d}(\ec)$, it follows from the argument above that $E\in \uc^{d}(\ec)$. Since $F\in \uc^{d-1}(\ec)$, we get that $E_2\in \uc^{d}(\ec)$ by applying \Cref{ClosureKernelDeflation} to the middle column. This shows that $\uc^{d}(\ec)$ is closed under extensions.
		
		Now we prove the "only if" direction of the claim. Fix an exact sequence 
		\begin{equation}\label{Equation:Fixedexactseq}
		0\to U\to E_d\to U_{d-1}\to \cdots \to U_1\to U'\to 0
		\end{equation}
		with $U,U'\in \uc$ and $U_1,\dots ,U_{d-1}\in \uc$.
		We prove by  induction on $0\leq i\leq d-1$ that \eqref{Equation:Fixedexactseq} is Yoneda equivalent to an exact sequence
		\[
		0\to U\to V_d\to \cdots \to V_{d-i+1}\to E_{d-i}\to V_{d-i-1}\to \cdots \to V_1\to U' \to 0
		\]  
		where $V_1,\dots, V_{d-i-1},V_{d-i+1},\dots ,V_d\in \uc$ and $E_{d-i}\in \uc^{d-i}(\ec)$. The case $i=d-1$ then proves the claim, since $E_1\in \uc^{1}(\ec)=\uc$.
		
		To see that the case $i=0$ holds we only need to check that $E_d\in \uc^{d}(\ec)$. For this, let $C$ be the cokernel of $U\to E_d$, and note that it is in $\uc^{d}(\ec)$. Since we have an exact sequence
		\[
		0\to U\to E_d\to C\to 0
		\]
		it follows that $E_d\in \uc^{d}(\ec)$ since $\uc^{d}(\ec)$ is extension-closed.
		
		Now assume the claim holds for $i<d-1$. We show that it holds for $i+1$. By the induction hypothesis \eqref{Equation:Fixedexactseq} is Yoneda-equivalent to an exact sequence
		\[
		0\to U\to V_d\to \cdots \to V_{d-i+1}\to E_{d-i}\to V_{d-i-1}\to \cdots \to V_1\to U' \to 0
		\]  
		where $V_1,\dots, V_{d-i-1},V_{d-i+1},\dots ,V_d\in \uc$ and $E_{d-i}\in \uc^{d-i}(\ec)$. Choose a conflation $$0\to E_{d-i}\to V\to E'_{d-i-1} \to 0$$ with $V\in \uc$ and $E'_{d-i-1}\in \uc^{d-i-1}(\ec)$. Let $K$ be the image of $V_{d-i+1}\to E_{d-i}$, let $C$ be the cokernel of $V_{d-i+1}\to E_{d-i}$, and let $C'$ be the cokernel of the composite $V_{d-i+1}\to E_{d-i}\to V$. Then we get a commutative diagram with exact rows
		\[
		\begin{tikzcd}[column sep=20, row sep=20]
		0\arrow[r] &K \arrow[r] \arrow[d, equal] & E_{d-i} \arrow[r] \arrow[d] & C \arrow[d]\arrow[r] & 0 \\
		0\arrow[r] &K \arrow[r, ""] & V \arrow[r]  & C' \arrow[r]  & 0.  
		\end{tikzcd}
		\]
		Note that the right hand square is bicartesian, so $C\to C'$ is an inflation with cokernel equal to $E_{d-i-1}'$. Taking the pushout of $C\to V_{d-i-1}$ along $C\to C'$ we get the following commutative diagram with exact rows 
		\[
		\begin{tikzcd}[column sep=20, row sep=20]
		0\arrow[r] &C \arrow[r] \arrow[d] & V_{d-i-1} \arrow[r] \arrow[d] & C'' \arrow[d,equal]\arrow[r] & 0 \\
		0\arrow[r] &C' \arrow[r, ""] & E_{d-i-1} \arrow[r]  & C'' \arrow[r]  & 0. 
		\end{tikzcd}
		\]
		Again, since the left hand square is bicartesian, the morphism $V_{d-i-1}\to E_{d-i-1}$ is an inflation with cokernel $E_{d-i-1}'$. Since $E'_{d-i-1}\in \uc^{d-i-1}(\ec)$ and $V_{d-i-1}\in \uc$, the object $E_{d-i-1}$ must be in $\uc^{d-i-1}(\ec)$ by \Cref{Lemma:ExtClosureProperty-Rigid}. Gluing these diagrams to the exact sequences
		\begin{align*}
		0\to U\to V_d\to \cdots \to V_{d-i+1}\to K\to 0 \\
		0\to C''\to V_{d-i-2}\to \cdots \to V_1\to U'\to 0
		\end{align*}
		we get a commutative diagram with exact rows
		\begin{equation*}
		\begin{tikzcd}[column sep=15]
		0 \arrow[r] & U\arrow[r] \arrow[d, equal] & \cdots \arrow[r,""] & V_{d-i+1}\arrow[r,""] \arrow[d,equal] & E_{d-i}\arrow[r] \arrow[d] &V_{d-i-1} \arrow[r] \arrow[d] & V_{d-i-2} \arrow[r]\arrow[d,equal] & \cdots \arrow[r] &U'\arrow[r]\arrow[d,equal] & 0 \\
		0 \arrow[r,""] & U \arrow[r,""] & \cdots  \arrow[r,""] & V_{d-i+1} \arrow[r] & V \arrow[r] & E_{d-i-1}\arrow[r] & V_{d-i-2} \arrow[r] & \cdots \arrow[r,""] & U' \arrow[r] & 0.
		\end{tikzcd} 
		\end{equation*}
		Since the lower exact sequence is Yoneda equivalent to the upper exact sequence, it is also Yoneda equivalent to our original exact sequence. This proves the claim.
	\end{proof}
 
We can deduce that $d$-cluster tilting subcategories are closed under $d$-extensions from \Cref{thm: extclosed}, recovering \cite[Theorem 1.2]{EN23} and \cite[A.1]{Iya07a}. 

\begin{proposition}\label{Corollary:d-CTClosedUnderdExt}
	If $\uc$ is generating and weakly right maximal $d$-rigid, then $\uc$ is closed under $d$-extensions. In particular, a $d$-cluster tilting subcategory is closed under $d$-extensions.
\end{proposition}

\begin{proof}
	Let $X,Y\in \uc$. By the dual of \Cref{lem: replaceterms} any element in $\Ext^d_\ec(X,Y)$ can be represented by an exact sequence
	\[
	0\to Y\to E_d\to U_{d-1}\to \cdots \to U_1\to X\to 0
	\]
	in $\ec$ with $U_i\in \uc$ for $1\leq i\leq d-1$. By assumption $\uc^d(\ec)=\ec$ since $\uc$ is weakly right maximal $d$-rigid in $\ec$. In particular, $\uc^d(\ec)$ is extension closed in $\ec$. Hence, the sequence above is Yoneda equivalent to an exact sequence where all the terms are in $\uc$ by \cref{thm: extclosed}. This proves the claim.
\end{proof}

	\section{Higher extension closure}\label{Section:Higher extension-closure}

    Throughout this section we fix an exact category $\ec$, an integer $d\geq 1$, and a $d$-rigid subcategory $\uc$. The goal of this section is to show that whenever $\uc$ is $d$-extension closed in $\ec$ as in \Cref{d-extension closure}, then its additive closure is a $d$-cluster tilting subcategory of some extension closed subcategory of $\ec$. The construction involves the subcategories $\uc^{d}(\ec)$ and $\uc_{d}(\ec)$ considered in the previous section. Combining the main result of this section with \Cref{Corollary:d-CTClosedUnderdExt}, we see that $d$-cluster tilting subcategories provide in a sense all the examples of $d$-extension closed subcategories.

	\subsection{Preservation of higher extensions}

We investigate the properties of $\uc^{d}(\ec)$ and $\uc_{d}(\ec)$ when $\uc$ is closed under $d$-extensions.

\begin{proposition} \label{cor: extclosure}
	Assume $\uc$ is $d$-extension closed in $\ec$. The following hold.
	\begin{enumerate}
		\item $\uc^{d}(\ec)$ and $\uc_{d}(\ec)$ are closed under extensions in $\ec$.
		\item $\uc$ is weakly right maximal $d$-rigid in $\uc^d(\ec)$ and weakly left maximal $d$-rigid in $\uc_d(\ec)$.
	\end{enumerate}
\end{proposition}
\begin{proof}
	The statements for $\uc^{d}(\ec)$ follow directly from \cref{thm: extclosed} and \Cref{Theorem:dExtImpliesRightMaxdRigid}, while the statements for $\uc_{d}(\ec)$ follow by the dual results.
\end{proof}

Next we investigate the higher extensions of $\uc$ in $\uc^d(\ec)$. 

  \begin{proposition}\label{dExtClosedInU^d}
      Assume $\uc$ is $d$-extension closed in $\ec$. Then $\uc$ is $d$-extension-closed in $\uc^d(\ec)$.
  \end{proposition}

  \begin{proof}
      Let $U,V\in \uc$ be arbitrary. By \cref{lem: replaceterms} any $d$-extension of $U$ and $V$ in $\uc^d(\ec)$ can be represented by an exact sequence of the form
\begin{equation*}
0 \to U \xrightarrow{} U_d \xrightarrow{} U_{d-1} \to \dots \to U_{2} \to X \xrightarrow{} V \to 0
\end{equation*}
where $U_2, \dots, U_{d} \in \uc$. Since $\uc$ is closed under $d$-extensions as a subcategory of $\ec$, this sequence is Yoneda equivalent in $\ec$ to an exact sequence 
\begin{align*}
0 \to U \xrightarrow{} V_d \xrightarrow{} V_{d-1} \to \dots \to V_{2} \xrightarrow{} V_1 \xrightarrow{} V \to 0
\end{align*}
where $V_i \in \uc$ for all $i$. Now this also represents an exact sequence in $\uc^d(\ec)$, since the cokernel of $V_{i+1}\to V_{i}$ must lie in $\uc^{i}(\ec)\subseteq \uc^d(\ec)$ for $i\geq 0$ (where $V_{d+1}=U$). By the dual of \cref{Lemma:YonedaEquivMorphism} we have a commutative diagram
\begin{equation*}
\begin{tikzcd}
0 \arrow[r] & U\arrow[r]  & U_d \arrow[r,""] & U_{d-1} \arrow[r,""] & \cdots \arrow[r,""] & X\arrow[r,""]  & V\arrow[r] & 0 \\
0 \arrow[r,""] & U \arrow[r,""] \arrow[u, equal] & V_d \arrow[r,""] \arrow[u] & V_{d-1} \arrow[r,""] \arrow[u] & \cdots  \arrow[r,""] & V_1 \arrow[r,""] \arrow[u] & V \arrow[r] \arrow[u, equal] & 0.
\end{tikzcd} 
\end{equation*}
This implies in particular that the sequences are Yoneda equivalent in $\uc^d(\ec)$, and hence $\uc$ is closed under $d$-extensions in $\uc^d(\ec)$.
  \end{proof}

  \begin{proposition}\label{Theorem:Ext^dIso}
      Assume $\uc$ is $d$-extension closed in $\ec$. Then the canonical map $$\Ext^d_{\uc^d(\ec)}(U,V)\to \Ext^d_\ec(U,V)$$ is an isomorphism for all $U,V\in \uc$.  
  \end{proposition}

  \begin{proof}
      We first prove injectivity of the map. Let $\delta$ denote an element in $\Ext^{d}_{\uc^d(\ec)}(V,U)$, and suppose $\delta$ gets sent to $0$ in $\Ext_\ec^d(V,U)$. By \Cref{dExtClosedInU^d} we may assume $\delta$ is represented by an exact sequence in $\uc^d(\ec)$
\begin{align*}
0 \to U \xrightarrow{} V_d \xrightarrow{} V_{d-1} \to \dots \to V_{2} \xrightarrow{} V_1 \xrightarrow{} V \to 0
\end{align*}
where $V_i \in \uc$ for all $i$. Applying $\Hom_{\ec}(V,-)$ and using that $\Ext^i_\ec(V,\uc)=0$ for $0<i<d$, we get an exact sequence
\[
0\to \Hom_{\ec}(V,U)\to \cdots \to \Hom_{\ec}(V,V_{1}) \to \Hom_{\ec}(V,V)\to \Ext^d_{\ec}(V,U).
\]
Since $\delta$ is $0$ in $\Ext_\ec^d(V,U)$,  the rightmost map above must be $0$. Hence, 
$$\Hom_{\ec}(V,V_{1}) \to \Hom_{\ec}(V,V)$$
must be surjective. This implies that the deflation $V_1 \to V$ is a split deflation. However, if this holds, then $\delta$ must be $0$ also as an element in $\Ext_{\uc^d(\ec)}^d(V,U)$ and thus we are done.

To prove surjectivity of the map, note that $\uc$ being closed under $d$-extensions in $\ec$ implies that any element of $\Ext_\ec^d(V,U)$ has a representative of the form 
\begin{align*}
0 \to U \xrightarrow{} V_d \xrightarrow{} V_{d-1} \to \dots \to V_{2} \xrightarrow{} V_1 \xrightarrow{} V \to 0
\end{align*}
where $V_i \in \uc$ for all $i$. Since the cokernel of $V_{i+1}\to V_{i}$ must lie in $\uc^{i}(\ec)\subseteq \uc^d(\ec)$, such a sequence must also represent an element in $\Ext_{\uc^d(\ec)}^d(V,U)$. This proves the claim.
  \end{proof}

 \subsection{Higher extension closure implies cluster tilting}\label{Subsection:d-Extension closed and d-cluster tilting}
 
The goal of this subsection is to show that if $\uc$ is $d$-extension-closed and closed under additive complements, then $\uc$ is a $d$-cluster tilting subcategory of $\uc_d(\uc^d(\ec))$. Note that $\uc_d(\uc^d(\ec))$ is extension closed in $\ec$ by \Cref{cor: extclosure}, since $\uc$ is $d$-extension closed in $\uc^d(\ec)$ by \Cref{dExtClosedInU^d}. Hence, we will always consider it as a fully exact subcategory of $\ec$.

First, we need the following lemma. 
 
 \begin{lemma}\label{Lemma:TowardsdCT}
 Assume $\ec'$ is a subcategory of $\ec$ closed under extensions and additive complements, such that
 	\begin{itemize}
 		\item $\uc$ is a cogenerating subcategory of $\ec'$.
 		\item $\uc$ is closed under $d$-extensions in $\ec'$.
 	    \item The canonical map $\Ext^d_{\ec'}(U,V)\to \Ext^d_\ec(U,V)$ is an isomorphism for all $U,V\in \uc$.
 	\end{itemize} 
 	Then for any conflation in $\ec$
 	\[
 	0\to E\to U\to G\to 0
 	\]
 	with $E\in \ec'$ and $U\in \uc$ and $G\in \uc^d(\ec)$, we must have that $G\in \ec'$.
 \end{lemma}
 
 \begin{proof}
 	Let $0\to E\xrightarrow{f} U\to G\to 0$ be an exact sequence as in the lemma. Since $\uc$ is cogenerating in $\ec'$, we can find a conflation $0\to E\xrightarrow{g} U'\to E'\to 0$ with $U'\in \uc$ and $E'\in \ec'$. Set $V=U\oplus U'$, let $\pi_1\colon V\to U$ and $\pi_2\colon V\to U'$ denote the projections, and let $h\colon E\to V$ be the morphism induced by $f\colon E\to U$ and $g\colon E\to U'$. Then, we have a commutative diagram 
 	\[
 	\begin{tikzcd}
 	0\arrow[r] &E \arrow[r,"h"] \arrow[d, equal] & V \arrow[r] \arrow[d,"\pi_2"] & E'' \arrow[d]\arrow[r] & 0 \\
 	0\arrow[r] &E \arrow[r, "g"] & U' \arrow[r]  & E' \arrow[r]  & 0  
 	\end{tikzcd}
 	\]
 	where $E''$ denotes the cokernel of $h$. Since the rightmost square is bicartesian, the morphism $E''\to E'$ must be a deflation, see e.g. \cite[Exercise 2.19]{Bue10}. Furthermore, the kernel of $E''\to E'$ must be isomorphic to the one of $\pi_2$, i.e. to $U$. Hence $E''$ is an extension of two objects in $\ec'$, and must therefore be contained in $\ec'$.
 	
 	Next, consider the commutative diagram  
 	\[
 \begin{tikzcd}
 0\arrow[r] &E \arrow[r,"h"] \arrow[d, equal] & V \arrow[r] \arrow[d,"\pi_1"] & E'' \arrow[d]\arrow[r] & 0 \\
 0\arrow[r] &E \arrow[r, "f"] & U \arrow[r]  & G \arrow[r]  & 0. 
 \end{tikzcd}
 \]
 Since the right hand square is bicartesian, we have a conflation
 \[
 0\to V\to U\oplus E''\to G\to 0. 
 \]	
 Since $G\in \uc^d(\ec)$, we can find an exact sequence 
 	\[
 	0\to G\to U^1\to \cdots \to U^d\to 0
 	\] 
 	in $\ec$ with $U^1,\dots,U^d\in \uc$. Combining this with the conflation above, we get an exact sequence
 	\[
 	0\to V\to U\oplus E''\to U^1\to U^2\to \cdots \to U^d\to 0.
 	\] Since $\Ext^d_{\ec'}(U^d,V)\cong \Ext^d_\ec(U^d,V)$, this sequence is Yoneda equivalent to an exact sequence in $\ec'$, and since $\uc$ closed under $d$-extensions in $\ec'$ and $U^d,V\in \uc$, we can assume the sequence is of the form
 	\[
 	0\to V\to V^1\to \cdots \to V^{d}\to U^d\to 0
 	\]
 	where $V^1,\dots, V^{d}\in \uc$. Now by \Cref{Lemma:YonedaEquivMorphism} we have a commutative diagram 
 	\begin{equation*}
 	\begin{tikzcd}
 	0 \arrow[r] & V\arrow[r] \arrow[d, equal] & U\oplus E'' \arrow[r,""] \arrow[d, ""] & U^1 \arrow[r,""] \arrow[d, ""] & \cdots \arrow[r,""] & U^{d-1}\arrow[r,""] \arrow[d,""] & U^d\arrow[r] \arrow[d, equal] & 0 \\
 	0 \arrow[r,""] & V \arrow[r,""] & V^1 \arrow[r,""] & V^2 \arrow[r,""] & \cdots  \arrow[r,""] & V^d \arrow[r,""] & U^d \arrow[r] & 0.
 	\end{tikzcd} 
 	\end{equation*}
 	In particular, taking the cokernel of the two leftmost horizontal morphisms, we get a commutative diagram 
 		\[
 	\begin{tikzcd}[column sep=10, row sep=20]
 	0\arrow[r] &V \arrow[r] \arrow[d, equal] & U\oplus E'' \arrow[r] \arrow[d,""] & G \arrow[d]\arrow[r] & 0 \\
 	0\arrow[r] &V \arrow[r] & V^1 \arrow[r]  & C \arrow[r]  & 0  
 	\end{tikzcd}
 	\]
 	where $C\in \ec'$. Since the rightmost square is bicartesian, we have a conflation
 	\[
 	0\to U\oplus E''\to G\oplus V^1\to C\to 0
 	\]
 	Since $\ec'$ is closed under extensions, it follows that $G\oplus V^1\in \ec'$. Finally, since $\ec'$ is closed under additive complements and $V^1\in \ec'$ it follows that $G\in \ec'$.
 \end{proof}

We can prove the main result of this section.

\begin{theorem}\label{Theorem:dCT}
	 Assume $\uc$ is $d$-extension closed in $\ec$. The following hold
  \begin{enumerate}
      \item\label{Theorem:dCT:1} $\uc$ is weakly right and weakly left maximal $d$-rigid in $\uc_d(\uc^d(\ec))$.
      \item\label{Theorem:dCT:2} If $\uc$ is closed under additive complements, then $\uc$ is $d$-cluster tilting in $\uc_d(\uc^d(\ec))$.
  \end{enumerate} 
\end{theorem}

\begin{proof}
	This clearly holds for $d=1$, so assume $d\geq 2$. Consider $\uc^d(\uc_d(\uc^d(\ec)))$ as a subcategory of $\ec$. By repeated applications of \Cref{Prop:ClosusureAdditiveComplements}, \Cref{cor: extclosure}, \Cref{dExtClosedInU^d} and \Cref{Theorem:Ext^dIso} and their duals, we get that $\uc^d(\uc_d(\uc^d(\ec)))$ is closed under extensions and additive complements in $\ec$, that $\uc$ is $d$-rigid and closed under $d$-extensions in $\uc^d(\uc_d(\uc^d(\ec)))$, and that the canonical map $\Ext^d_{\uc^d(\uc_d(\uc^d(\ec)))}(U,V)\to \Ext^d_{\ec}(U,V)$ is an isomorphism for all $U,V\in \uc$. This implies that $\uc^d(\uc_d(\uc^d(\ec)))$ satisfies the conditions of \Cref{Lemma:TowardsdCT}. Hence, by repeated applications of that lemma we get that whenever
	\[
	0\to U_d\to \cdots \to U_1\to E\to 0
	\]
	is an exact sequence in $\uc^d(\ec)$ where $U_i\in \uc$ for all $1\leq i\leq d$, then $E\in \uc^d(\uc_d(\uc^d(\ec)))$. In other words, we have $$\uc_d(\uc^d(\ec))\subseteq \uc^d(\uc_d(\uc^d(\ec))).$$ Since the other inclusion follows by definition, we get that $$\uc_d(\uc^d(\ec))=\uc^d(\uc_d(\uc^d(\ec))).$$ Hence, $\uc$ must be both weakly right and weakly left maximal $d$-rigid in $\uc_d(\uc^d(\ec))$. If $\uc$ is closed under additive complements in $\ec$, then $\uc$ must be closed under direct summands and therefore be $d$-cluster tilting in $\uc_d(\uc^d(\ec))$ by \Cref{Lemma:UClosedUnderDirectSummand}.
 \end{proof}

 We have the following uniqueness result for $\uc_d(\uc^d(\ec))$.

 \begin{theorem}\label{Theorem:UniquenessdCT}
Assume  $\uc$ is $d$-extension closed in $\ec$ for $d\geq 2$. A subcategory $\ec'$ of $\ec$ is equal to $\uc_d(\uc^d(\ec))$ if and only if 
  \begin{enumerate}
      \item\label{Theorem:UniquenessdCT:1} $\ec'$ is closed under extensions and additive complements in $\ec$.
      \item\label{Theorem:UniquenessdCT:2} $\uc$ is weakly right and weakly left maximal $d$-rigid in $\ec'$.
      \item\label{Theorem:UniquenessdCT:3} The canonical map $\Ext^d_{\ec'}(U,V)\to \Ext^d_\ec(U,V)$ is an isomorphism for all $U,V\in \uc$.
  \end{enumerate}
 \end{theorem}

 \begin{proof}
      By \Cref{Prop:ClosusureAdditiveComplements}, \Cref{cor: extclosure}, \Cref{dExtClosedInU^d}, \Cref{Theorem:Ext^dIso} and \Cref{Theorem:dCT} the properties \eqref{Theorem:UniquenessdCT:1} to \eqref{Theorem:UniquenessdCT:3} hold for $\uc_d(\uc^d(\ec))$. Conversely, assume $\ec'$ is a subcategory satisfying these properties. Since $\uc$ is weakly right and weakly left maximal $d$-rigid, $\uc$ must be closed under $d$-extensions in $\ec'$ by \Cref{Corollary:d-CTClosedUnderdExt}. Hence, the conditions of \Cref{Lemma:TowardsdCT} hold for $\ec'$. Therefore, given an exact sequence
	\[
	0\to U_d\to \cdots \to U_1\to E\to 0
	\]
	 in $\uc^d(\ec)$ where $U_i\in \uc$ for all $1\leq i\leq d$, we must have that $E\in \ec'$. This shows that $\uc_d(\uc^d(\ec))\subseteq \ec'$. Since $\uc$ is weakly right and weakly left maximal $d$-rigid in $\ec'$, we must have that $\ec'\subseteq \uc_d(\uc^d(\ec))$. Hence, $\ec'=\uc_d(\uc^d(\ec))$, which proves the claim.
\end{proof}

We get the following corollary.

\begin{corollary}\label{Corollary:Equality}
	If $\uc$ is $d$-extension closed in $\ec$, then $\uc_d(\uc^d(\ec))=\uc^d(\uc_d(\ec))$.
\end{corollary}

\begin{proof}
	Since the criteria \eqref{Theorem:UniquenessdCT:1}-\eqref{Theorem:UniquenessdCT:3} in \Cref{Theorem:UniquenessdCT} are self-dual, they are also satisfied by $\uc^d(\uc_d(\ec))$. Hence, $\uc^d(\uc_d(\ec))=\uc_d(\uc^d(\ec))$ by uniqueness.
\end{proof}

\subsection{Higher torsion classes}
We explain  how higher torsion classes in the sense of \cite{J16} can be realized as $d$-cluster tilting subcategories of extension-closed subcategories, using the results in the previous subsection. Fix a $d$-cluster tilting subcategory $\mc$ of an abelian category $\ac$.

A $d$\textit{-torsion class} of $\mc$ is a subcategory $\uc$ of $\mc$ closed under direct summands, $d$-extensions, and $d$-quotients. Here closure under $d$-quotients means that any morphism $X\to U$ with $X\in \mc$ and $U\in \uc$ has a $d$-cokernel in $\mc$
\[
U\to U^1\to U^2\to \cdots \to U^d\to 0
\]
where $U^i\in \uc$ for all $i$. This definition of $d$-torsion class is different to the one in \cite{J16}, but by \cite[Theorem 1.1]{AHJKPT25} they coincide when $\ac$ is an abelian length category.

\begin{theorem}\label{Theorem:HigherTorsion}
    Let $\uc$ be a $d$-torsion class of $\mc$. Then $\uc$ is a $d$-cluster tilting subcategory of
    \[
    \uc_d(\ac)=\{A\in \ac\mid \exists\text{ }0\to U_d\to \cdots \to U_1\to A\to 0 \text{ exact, }U_i\in \uc \text{, }1\leq i\leq d\}.
    \]
\end{theorem}

\begin{proof}
    Let $0\to U_d\to \cdots \to U_1\to A\to 0$ be an exact sequence as in the definition of $\uc_d(\ac)$. Since $\uc$ is closed under $d$-quotients, we can find a $d$-cokernel $U_1\to V^1\to \cdots \to V^d\to 0$ of $U_2\to U_1$. This gives an exact sequence
    \[
    0\to A\to V^1\to \cdots \to V^d\to 0
    \]
    which shows that $\uc_d(\ac)\subseteq \uc^d(\ac)$. Applying $\uc_d(-)$ on both sides, we get that $$\uc_d(\uc_d(\ac))\subseteq \uc_d(\uc^d(\ac)).$$ Now clearly $\uc_d(\uc_d(\ac))=\uc_d(\ac)$, and hence $\uc_d(\ac)=\uc_d(\uc^d(\ac))$. The claim now follows from \Cref{Theorem:dCT}.
\end{proof}
\subsection{Higher wide subcategories}

Here we prove that any higher wide subcategory in the sense of \cite{HJV20} is $d$-cluster tilting in some abelian category. 

Fix a $d$-cluster tilting subcategory $\mc$ of an abelian category $\ac$. A $d$\textit{-wide subcategory} is a subcategory $\wc$ of $\mc$ closed under direct summands, $d$-extensions, $d$-cokernels, and $d$-kernels. Here closure of $d$-kernels and $d$-cokernels means that any morphism in $\wc$ has a $d$-kernel and $d$-cokernel in $\mc$ where all the terms are in $\wc$.

\begin{theorem}\label{Theorem:HigherWide}
    Let $\wc$ be a $d$-wide subcategory of $\mc$. Then there exists a unique wide subcategory $\bc$ of $\ac$ such that $\wc$ is a $d$-cluster tilting subcategory of $\bc$. Furthermore
    \begin{align*}
        \bc&=\{A\in \ac\mid \exists\text{ }0\to W_d\to \cdots \to W_1\to A\to 0 \text{ exact, }W_i\in \wc \text{, }1\leq i\leq d\} \\
        & =\{A\in \ac\mid \exists\text{ }0\to A\to W^1\to \cdots \to W^d \to 0 \text{ exact, }W^i\in \wc \text{, }1\leq i\leq d\}.
    \end{align*}
\end{theorem}

\begin{proof}
   By a similar argument as in the proof of \Cref{Theorem:HigherTorsion} we have that $\wc_d(\ac)\subseteq \wc^d(\ac)$ and $\wc^d(\ac)\subseteq \wc_d(\ac)$. Hence, $\wc_d(\ac)=\wc^d(\ac)$. We denote this category by $\bc$. By \Cref{Theorem:dCT} it follows that $\wc$ is a $d$-cluster tilting subcategory of $\bc$.
   
   We show that $\bc$ is a wide subcategory of $\ac$.  By \Cref{cor: extclosure} we know that $\bc$ is closed under extensions. We show it is closed under cokernels. Let $B\to B'$ be a morphism in $\bc$, and let $C$ be its cokernel. Since $B\in \wc_d(\ac)$, we can find an epimorphism $W\to B$ with $W\in \wc$. Note that $C$ is also the cokernel of the composite $W\to B\to B'$. Now since $B'\in \wc^d(\ac)$, we can find a monomorphism $B'\to W'$. Let $C'$ denote the cokernel of the morphism $W\to B'\to W'$. We then have a commutative diagram  with right exact rows
   \[
 	\begin{tikzcd}
 	W \arrow[r] \arrow[d, equal] & B' \arrow[r] \arrow[d] & C \arrow[d]\arrow[r] & 0 \\
 	W \arrow[r] & W' \arrow[r]  & C' \arrow[r]  & 0.  
 	\end{tikzcd}
 	\]
  Since the right hand square is bicartesian, we get an exact sequence
\[
0\to B'\to W'\oplus C\to C'\to 0.
\]
Since $\wc$ is closed under $d$-cokernels,  the cokernel $C'$ of $W\to W'$ must be in $\wc^d(\ac)$. Hence, in the exact sequence above the rightmost and leftmost term are both in $\wc^d(\ac)$, and so $W'\oplus C$ is contained in $\wc^d(\ac)$ since $\wc^d(\ac)$ is closed under extensions. Finally, since  $\wc^d(\ac)$ is closed under additive complements by \Cref{Prop:ClosusureAdditiveComplements}, we get that $C\in \wc^d(\ac)=\bc$. This shows that $\bc$ is closed under cokernels. The fact that $\bc$ is closed under kernels is proved dually. 

Finally, assume $\bc'$ is an arbitrary wide subcategory of $\ac$ containing $\wc$ as a $d$-cluster tilting subcategory. Since $\bc'$ is closed under kernels, it follows that $\wc^d(\ac)\subseteq \bc'$. Since $\wc$ is $d$-cluster tilting in $\bc'$, it must be right maximal $d$-rigid in $\bc'$, and hence $\bc'\subseteq \wc^d(\ac)$. Therefore, $\bc=\wc^d(\ac)=\bc'$, which proves the claim.  
\end{proof}

\section{The embedding theorem}\label{Section:EmbeddingThm}

Here we combine the result of \cite{Ebr21} and \cite{EN23} with those of \Cref{Section:Higher extension-closure} to prove that any weakly idempotent complete $d$-exact category is equivalent to a $d$-cluster tilting subcategory of a weakly idempotent exact category.

Fix $\mc$ to be a $d$-exact category. For simplicity, we assume $d\geq 2$. Recall that the Yoneda functor induces a fully faithful functor $Y\colon \mc\to \mathcal{L}(\mc)$ into the category of left exact functors $\mathcal{L}(\mc)$ whose essential image is $d$-rigid and closed under $d$-extensions. See \Cref{Subsection:The category of left exact functors} for details.

 Consider the subcategories
\begin{align*}
	\mc^{d}:=\{F\in \mathcal{L}(\mc) \mid \exists\text{ }0\to F\to \mc(-,X^1)\to \cdots \to \mc(-,X^d)\to 0 \text{ exact in }\mathcal{L}(\mc)\} \\
	\mc_{d}:=\{F\in \mathcal{L}(\mc) \mid \exists\text{ }0\to \mc(-,X_d)\to \cdots \to \mc(-,X_1)\to F\to 0 \text{ exact in }\mathcal{L}(\mc)\}.
	\end{align*}
Note that $\mc^{d}$ and $\mc_{d}$ would be $Y(\mc)^{d}(\mathcal{L}(\mc))$ and $Y(\mc)_{d}(\mathcal{L}(\mc))$ in the notation of \Cref{Subsection:FiniteRes&Cores}, where $Y(\mc)$ denotes the essential image of $Y$ in $\mathcal{L}(\mc)$.

\begin{definition}\label{Definition:E(M)}
    Let $\mathcal{E}(\mc)$ be the subcategory of $\mathcal{L}(\mc)$ consisting of all objects $F$ for which there exists an exact sequence
\[
0\to \mc(-,X_d)\to \cdots \to \mc(-,X_1)\to F\to 0
\]
in $\mathcal{L}(\mc)$ where the cokernel of $\mc(-,X_{i+1})\to \mc(-,X_i)$ is in $\mc^{d}$ for all $1\leq i\leq d-1$.
\end{definition} 
Note that $\ec(\mc)$ can be identified with the subcategory $Y(\mc)_dY(\mc)^{d}(\mathcal{L}(\mc))$ studied in \Cref{Subsection:d-Extension closed and d-cluster tilting}. In particular, we have the following result.

\begin{proposition}\label{Cor:dExactAmbientExactSubcat}
The following hold.
\begin{enumerate}
    \item\label{Cor:dExactAmbientExactSubcat:1} $\mathcal{E}(\mc)$ is weakly idempotent complete and closed under extensions in $\mathcal{L}(\mc)$.
    \item\label{Cor:dExactAmbientExactSubcat:2} The canonical map $$\Ext^d_{\mathcal{E}(\mc)}(F,G)\to \Ext^d_{\mathcal{L}(\mc)}(F,G)$$ is an isomorphism for all $F,G\in Y(\mc)$.
    \item\label{Cor:dExactAmbientExactSubcat:3} Let $F\in \mathcal{L}(\mc)$. Then $F\in \ec(\mc)$ if and only if there exists an exact sequence in $\mathcal{L}(\mc)$
\[
0\to F\to  \mc(-,X^1)\to \cdots \to \mc(-,X^d)\to 0
\]
 where the kernel of $\mc(-,X^{i})\to \mc(-,X^{i+1})$ is in $\mc_{d}$ for all $1\leq i\leq d-1$.
\end{enumerate}
\end{proposition}

\begin{proof}
The subcategory $Y(\mc)$ is closed under $d$-extensions in $\mathcal{L}(\mc)$ by \Cref{Proposition:YonedaImageEbrahimi} \eqref{Proposition:YonedaImageEbrahimi:2}. Hence,  $\ec(\mc)$ is closed under extensions and additive complements in $\mathcal{L}(\mc)$ and the map in \eqref{Cor:dExactAmbientExactSubcat:2} is an isomorphism by \Cref{Prop:ClosusureAdditiveComplements}, \Cref{cor: extclosure}, \Cref{dExtClosedInU^d}, and \Cref{Theorem:Ext^dIso}. Note that $\ec(\mc)$ being closed under additive complements in $\mathcal{L}(\mc)$ implies that it is weakly idempotent complete since $\mathcal{L}(\mc)$ is weakly idempotent complete.  Part \eqref{Cor:dExactAmbientExactSubcat:3} follows from \Cref{Corollary:Equality}.
\end{proof}

\begin{theorem}\label{Cor:WeaklyIdemdExactdCT}
The following hold.
\begin{enumerate}
    \item\label{Cor:WeaklyIdemdExactdCT:1} The weak idempotent completion of $\mc$ is equivalent to a $d$-cluster tilting subcategory of $\ec(\mc)$.
    \item\label{Cor:WeaklyIdemdExactdCT:2} Assume $\mc$ is weakly idempotent complete. Then $Y\colon \mc\to Y(\mc)$ is an equivalence of $d$-exact categories where $Y(\mc)$ is endowed with the induced $d$-exact structure from being a $d$-cluster tilting subcategory. 
\end{enumerate}
\end{theorem}

   \begin{proof}
   Since $\ec(\mc)$ is weakly idempotent complete by \Cref{Cor:dExactAmbientExactSubcat}, the Yoneda embedding $Y\colon \mc\to \ec(\mc)$ must induce a fully faithful functor 
   \[
   \hat{\mc}\to \ec(\mc)
   \]
   where $\hat{\mc}$ is the weak idempotent completion of $\mc$. Note that its essential image can be identified with the closure of the essential image $Y(\mc)$ of $Y$ under additive complements. Since $Y(\mc)$ is closed under $d$-extensions in $\mathcal{L}(\mc)$ by \Cref{Proposition:YonedaImageEbrahimi} \eqref{Proposition:YonedaImageEbrahimi:2}, it must be weakly right and weakly left maximal $d$-rigid in $\ec(\mc)$ by \Cref{Theorem:dCT}. Hence, the closure of $Y(\mc)$ under additive complements must be $d$-cluster tilting in $\ec(\mc)$ by \Cref{Lemma:UClosedUnderDirectSummand}. This proves the first claim. 
   
   The second claim follows immediately from the fact that a complex $X_{d+1}\to \cdots \to X_{0}$ in $\mc$ is admissibly $d$-exact if and only if $0\to Y(X_{d+1})\to \cdots \to Y(X_{0})\to 0$ is acyclic in $\ec(\mc)$ by \Cref{Proposition:YonedaImageEbrahimi} \eqref{Proposition:YonedaImageEbrahimi:3}, and the latter is equivalent to being an admissible $d$-exact sequence in $Y(\mc)$ by \Cref{Theorem:dCTImpliesdExactWeakIdemPotent}.
   \end{proof}

   \begin{remark}\label{Remark:WeakIdemPotentCompletion}
   The weak idempotent completion $\hat{\mc}$ of a $d$-exact category $\mc$ has a natural $d$-exact structure by \cite[Corollary 5.6]{KMS24}. We claim that with this structure the functor  $$\hat{\mc}\to \ec(\mc)$$ 
   in \Cref{Cor:WeaklyIdemdExactdCT} \eqref{Cor:WeaklyIdemdExactdCT:1} induces an equivalence of $d$-exact categories between $\hat{\mc}$ and its essential image $\nc$ (endowed with the $d$-exact structure from being a $d$-cluster tilting subcategory). 
   
   Indeed, we  know that $\hat{\mc}\to \nc$ is an equivalence of additive categories. Since $Y\colon \mc\to \nc$ is $d$-exact, the universal property of the $d$-exact structure on the weak idempotent completion implies that $\hat{\mc}\to \nc$ must also be $d$-exact, see \cite[Theorem 5.5]{KMS24} and \cite[Theorem 2.34]{BS21}. Hence,  we only need to show that if $X_\bullet=(X_{d+1}\to \cdots \to X_{0})$ is a complex in $\hat{\mc}$ such that its image in $\nc$ is an admissible $d$-exact sequence, then $X_\bullet$ must be an admissible $d$-exact sequence in $\hat{\mc}$. First note that $X_\bullet$ must be a $d$-exact sequence, since $\hat{\mc}\to \nc$ is an equivalence of additive categories. Furthermore, by adding two null-homotopic complexes of the form 
   $$\cdots\to  0\to X\xrightarrow{1_X}X\to 0\to \cdots$$
   for suitable choices of $X\in \hat{\mc}$, we may assume $X_0,X_{d+1}\in \mc$. By hypothesis the image $$0\to Y(X_{{d+1}})\to Z_d\to \cdots \to Z_{1}\to Y(X_0)\to 0$$ of $X_\bullet$ in $\nc$ is admissibly $d$-exact, which implies it is acyclic in $\ec(\mc)$.  Hence, it is Yoneda equivalent to $Y(X'_\bullet)$ for an admissible $d$-exact sequence $$X'_\bullet=(X_{d+1}\to X_d'\to \cdots \to X_1'\to X_{0})$$ 
   in $\mc$, by \Cref{Proposition:YonedaImageEbrahimi}. Therefore, by \Cref{Lemma:YonedaEquivMorphism} we have a commutative diagram
    \begin{equation*}
	\begin{tikzcd}[column sep=15]
	0\arrow[r]& Y(X_{d+1})\arrow[r] \arrow[d,equal] &  Z_d \arrow[r]\arrow[d] & \cdots \arrow[r]& Z_1 \arrow[r] \arrow[d] & Y(X_{0}) \arrow[d,equal] \arrow[r] & 0  \\
	0\arrow[r]& Y(X_{d+1})\arrow[r]&  Y(X_d')\arrow[r] & \cdots \arrow[r] & Y(X_1') \arrow[r] & Y(X_0)    \arrow[r] & 0.
	\end{tikzcd} 
	\end{equation*}
  Since the functor $\hat{\mc}\to \nc$ is fully faithful, the diagram gives a weak isomorphism in $\hat{\mc}$ between $X_\bullet$ and $X'_\bullet$. Hence, $X_\bullet$ must be an admissible $d$-exact sequence.
   \end{remark}

For $d=1$ following result was shown in \cite[Appendix A]{Kel90}. 

   \begin{corollary}\label{Corollary:E1IsRedundant}
       The axiom $\operatorname{(E1)}$ for $d$-exact categories is redundant.
   \end{corollary}

   \begin{proof}
    Let $\nc$ be an additive category endowed with a class $\xc$ of $d$-exact sequences which are closed under weak isomorphisms and satisfy axioms (E0), (E1)$^{\operatorname{op}}$, (E2) and (E2)$^{\operatorname{op}}$. Note first that $\mathcal{L}(\nc)$ is still a localizing subcategory of $\operatorname{Mod}\nc$ and \Cref{Proposition:YonedaImageEbrahimi} still holds for $(\nc,\xc)$. Indeed, it can be checked that the proofs of these statements in \cite{Ebr21,EN23} do not use the axiom (E1) (in particular, the "if" direction of \Cref{Proposition:YonedaImageEbrahimi} \eqref{Proposition:YonedaImageEbrahimi:3} relies on the dual of the obscure axiom \cite[Proposition 4.11]{Jas16}, which again only needs \cite[Proposition 4.9]{Jas16}, \cite[Proposition 4.8 (i)$\implies$ (iv)]{Jas16}, and the dual of \cite[Proposition 4.8 (i)$\implies$ (ii)]{Jas16}. None of these statements relies on (E1)).
    
    Since \Cref{Proposition:YonedaImageEbrahimi} holds for $(\nc,\xc)$, so does \Cref{Cor:WeaklyIdemdExactdCT}. Therefore, the weak idempotent completion $\hat{\nc}$ of $\nc$ is equivalent to a $d$-cluster tilting subcategory of $\ec(\nc)$, and therefore 
  inherits the structure of a $d$-exact category. Since the essential image of $\nc\to \mathcal{L}(\nc)$ is closed under $d$-extensions by \Cref{Proposition:YonedaImageEbrahimi} \eqref{Proposition:YonedaImageEbrahimi:2}, the category $\nc$ must be closed under $d$-extensions as a subcategory of $\hat{\nc}$. Hence, $\nc$ inherits a $d$-exact structure from $\hat{\nc}$ by \cite[Theorem E]{Kla26}. By applying \Cref{Proposition:YonedaImageEbrahimi} \eqref{Proposition:YonedaImageEbrahimi:3} to $(\nc,\xc)$ we see that this structure coincides with $\xc$. Hence, axiom (E1) is redundant. 
 \end{proof}

 Next we consider the idempotent completion of a $d$-exact category. Let $\add\ec(\mc)$ denote the smallest subcategory of $\operatorname{Mod}\mc$ which is closed under direct summands and finite direct sums and contains $\ec(\mc)$. Note that $\add\ec(\mc)$ is closed under extensions in $\mathcal{L}(\mc)$ since $\ec(\mc)$ is closed under extensions in $\mathcal{L}(\mc)$.
 
 \begin{corollary}\label{Cor:IdemdExactdCT}
The idempotent completion of $\mc$ is equivalent to a $d$-cluster tilting subcategory of $\add\ec(\mc)$.
\end{corollary}

\begin{proof}
    By \Cref{Cor:WeaklyIdemdExactdCT} \eqref{Cor:WeaklyIdemdExactdCT:1} the weak idempotent completion of $\mc$ is a $d$-cluster tilting subcategory of $\ec(\mc)$. Hence by \Cref{Prop:dCTFor(Weak)IdemComp} \eqref{Prop:dCTFor(Weak)IdemComp:2} the idempotent completion of $\mc$ is a $d$-cluster tilting subcategory of the idempotent completion of $\ec(\mc)$. The claim now follows since $\add\ec(\mc)$ can be identified with the idempotent completion of $\ec(\mc)$.
\end{proof} 

\begin{remark}
    The idempotent completion of $\mc$ is equivalent to $\operatorname{add}\mc$, the smallest subcategory of $\operatorname{Mod}\mc$ which contains the essential image of $Y\colon \mc\to \operatorname{Mod}\mc$ and is closed under direct summands and finite direct sums. Hence, \Cref{Cor:IdemdExactdCT} gives an affirmative answer to a question in \cite[Remark 3.8]{Ebr21}.
\end{remark} 

\begin{remark}
Note that  $$\add\ec(\mc)=\ec(\operatorname{add}\mc)$$
 where $\ec(\operatorname{add}\mc)$ is obtained from \Cref{Definition:E(M)} by replacing $\mc$ with $\operatorname{add}\mc$ and identifying $\mathcal{L}(\mc)$ and $\mathcal{L}(\operatorname{add}\mc)$. Indeed, since $\add\mc$ is $d$-cluster tilting in $\add\ec(\mc)$ by \Cref{Cor:IdemdExactdCT}, it is $d$-extension closed in that subcategory, see \Cref{Corollary:d-CTClosedUnderdExt}. Now by \Cref{Cor:dExactAmbientExactSubcat} the map
 \[
   \Ext^d_{\ec(\mc)}(F,G)\to \Ext^d_{\mathcal{L}(\mc)}(F,G)
 \]
 is an isomorphism for all $F,G\in Y(\mc)$ and hence
 \begin{equation*}
     \Ext^d_{\add\ec(\mc)}(F,G)\to \Ext^d_{\mathcal{L}(\mc)}(F,G)
 \end{equation*}
 is an isomorphism for all $F,G\in \add \mc$. This implies that $\add\mc$ is also $d$-extension-closed in $\mathcal{L}(\mc)$. Applying \Cref{Theorem:UniquenessdCT} with $\ec'=\add\ec(\mc)$ gives the equality.
 \end{remark}

 \begin{remark}\label{Remark:IdemPotentCompletion}
   The idempotent completion of a $d$-exact category has a natural $d$-exact structure \cite[Corollary 4.34]{KMS24}, and with this the equivalence in \Cref{Cor:IdemdExactdCT} becomes an equivalence of $d$-exact categories. The argument for this is identical to the one in \Cref{Remark:WeakIdemPotentCompletion}, using that the $d$-exact structure of the idempotent completion has a universal property by \cite[Theorem 4.39]{KMS24} and \cite[Theorem 2.34]{BS21}.
 \end{remark}

\section{The universal property}\label{Section:UniversalProperty}

 We saw in the previous subsection that any weakly idempotent complete $d$-exact category is equivalent to a $d$-cluster tilting subcategory of a weakly idempotent complete exact category. The goal of this subsection is to show that the exact category satisfies a universal property, and is therefore unique up to exact equivalence. Throughout we fix a weakly idempotent complete exact category $\ec$ with a $d$-cluster tilting subcategory $\mc$. We will use the fact that any right $\mc$-approximation is a deflation, since $\mc$ is generating.

  We first need some preliminary results. Recall that $\operatorname{mod}\mc$ denotes the category of finitely presented functors, see \Cref{Subsection:The category of left exact functors} for the definition. 
  
  \begin{lemma}\label{Lemma:EFinitelyPresInmodM}
The functor $\Hom_{\ec}(-,E)|_{\mc}$ lies in $\operatorname{mod}\mc$ for any $E\in \ec$.
  \end{lemma}

  \begin{proof}
      Let $E\in \ec$ be arbitrary. Choose a right $\mc$-approximation $M_0\to E$. Let $K$ denote its kernel, and let $M_{1}\to K$ be a right $\mc$-approximation of $K$. Then the conflation $$0\to K\to M_0\to E\to 0$$ gives an exact sequence
      \[
      0\to \Hom_{\ec}(-,K)|_{\mc}\to \Hom_{\mc}(-,M_0)\to \Hom_{\ec}(-,E)|_{\mc}\to 0
      \]
      and the right $\mc$-approximation $M_{1}\to K$ gives an epimorphism 
      \[
      \Hom_{\mc}(-,M_{1})\to \Hom_{\ec}(-,K)|_{\mc}\to 0.
      \]
      Combining these, we get an exact sequence
      \[
      \Hom_{\mc}(-,M_{1})\to \Hom_{\mc}(-,M_0)\to \Hom_{\ec}(-,E)|_{\mc}\to 0
      \]
      which proves the claim.
  \end{proof}

A complex $M_{1}\to M_0\to E$ in $\ec$ with $M_{1},M_0\in \mc$ is called an $\mc$\textit{-presentation} of $E$ if
\[
\Hom_{\mc}(-,M_{1})\to \Hom_{\mc}(-,M_0)\to \Hom_{\ec}(-,E)|_{\mc}\to 0
\]
is exact. The association $E\mapsto \Hom_{\ec}(-,E)|_{\mc}$ in \Cref{Lemma:EFinitelyPresInmodM} gives a functor denoted $$\Hom_{\ec}(\mc,-)\colon \ec\to \operatorname{mod}\mc.$$
For the next result we also need the following functors.
\begin{itemize}
    \item If $F\colon \xc\to \yc$ is an additive functor between additive categories, then
    \[
    F_!\colon \operatorname{mod}\xc\to \operatorname{mod}\yc
    \]
    denotes the unique right exact functor extending $F$. See \cite[Lemma 2.6]{Kra98} for details.
    \item If $\ac$ is an abelian category, then $$L\colon\operatorname{mod}\ac\to \ac$$ denotes the unique right exact functor extending the identity functor on $\ac$. See \cite[Theorem 2.2]{Kra15} for details.
\end{itemize}

 Let $\nc$ be a $d$-exact category, and let $\ac$ be an abelian category. A functor $F\colon \nc\to \ac$ is \textit{right exact} 
 if for any admissible $d$-exact sequence $X_{d+1}\to \cdots \to X_{0}$ in $\nc$ the sequence $$F(X_{2})\to F(X_1)\to F(X_{0})\to 0$$ is exact in $\ac$. Note that if $\ac=\operatorname{Ab}$, then such functors can be identified with objects in $\mathcal{L}(\nc^{\operatorname{op}})$. 
 
 \begin{proposition}\label{Proposition:UnivPropRightExact}
Let $\ac$ be an abelian category and $F\colon \mc\to \ac$ a right exact functor. Then there exists a right exact functor $\overline{F}\colon \ec\to \ac$ extending $F$. Furthermore, $\overline{F}$ is unique up to natural isomorphism.
 \end{proposition}
 
 \begin{proof}
Uniqueness of $\overline{F}$ follows from the fact that any right exact functor is uniquely determined up to natural isomorphism by its restriction to a generating subcategory.

We show existence.   Define $\overline{F}\colon \ec\to \ac$ as the composite
    \[
    \ec\xrightarrow{\Hom_{\ec}(\mc,-)} \operatorname{mod}\mc\xrightarrow{F_!} \operatorname{mod}\ac \xrightarrow{L}\ac.
    \]
    Explicitly, if $E\in \ec$ then $\overline{F}(E)$ is the cokernel of $F(M_{1})\to F(M_0)$ where $M_{1}\to M_0\to E$ is a choice of an $\mc$-presentation of $E$. In particular, $\overline{F}(M)\cong F(M)$ naturally in $M\in \mc$, since $0\to M$ is a $\mc$-presentation of $M$. Hence $\overline{F}$ extends $F$.
    
    It remains to show that $\overline{F}$ is right exact. Fix a conflation 
     \begin{equation}\label{Equation:SES}
    0\to X\to Y\to Z\to 0
    \end{equation} 
    in $\ec$. Our strategy is to show that the sequence $\overline{F}(X)\to \overline{F}(Y)\to \overline{F}(Z)\to 0$ is right exact by first making different additional assumptions on the conflation. The right exactness of $\overline{F}$ will be deduced from these special cases.

    First assume that $Y\in \mc$ and $Y\to Z$ is a right $\mc$-approximation in \eqref{Equation:SES}. Choose a right $\mc$-approximation $M\to X$. Then the morphism $M\to Y$ gives a $\mc$-presentation of $Z$, so the sequence
    \[
    \overline{F}(M)\to \overline{F}(Y)\to \overline{F}(Z)\to 0
    \]
    must be right exact by definition of $\overline{F}(Z)$ and the fact that $\overline{F}(M)\cong F(M)$ and $\overline{F}(Y)\cong F(Y)$. Hence, the sequence
     \[
    \overline{F}(X)\to \overline{F}(Y)\to \overline{F}(Z)\to 0
    \]
    must also be right exact. In particular,  $\overline{F}$ sends right $\mc$-approximations to epimorphisms.

    Next assume $Y,Z\in \mc$ in \eqref{Equation:SES}. Choose an exact sequence
    \[
   0\to  M_{d}\to \cdots \to M_1\to X\to 0.
    \]
    Concatenating with the conflation $0\to X\to Y\to Z\to 0$ gives an exact sequence
    \[
    0\to M_d\to \cdots \to M_1\to Y\to Z\to 0
    \]
    in $\ec$ with terms in $\mc$. Hence, it must be an admissible $d$-exact sequence in $\mc$. Applying $\overline{F}$, we get a right exact sequence
    \[
    \overline{F}(M_1)\to  \overline{F}(Y)\to \overline{F}(Z)\to 0
    \]
    since $F$ is right exact and $\overline{F}$ coincides with $F$ on $\mc$. Since $\overline{F}(M_1)\to  \overline{F}(Y)$ factors through $\overline{F}(X)\to \overline{F}(Y)$, the sequence
     \[
    \overline{F}(X)\to \overline{F}(Y)\to \overline{F}(Z)\to 0
    \]
    must also be right exact.

     Now assume $Z\in \mc$ in \eqref{Equation:SES}. Choose a right $\mc$-approximation $M\to Y$, and let $X'$ denote the kernel of the composite $M\to Y\to Z$. We then have a commutative diagram
     \[
		\begin{tikzcd}[column sep=20, row sep=20]
		0\arrow[r] &X' \arrow[r] \arrow[d] & M \arrow[r] \arrow[d] & Z \arrow[d,equal]\arrow[r] & 0 \\
		0\arrow[r] &X \arrow[r, ""] & Y \arrow[r]  & Z \arrow[r]  & 0
		\end{tikzcd}
		\]
    where the rows are conflations.  Applying $\overline{F}$, we get the commutative diagram
 \[
		\begin{tikzcd}[column sep=20, row sep=20]
		 \overline{F}(X') \arrow[d] \arrow[r]  & \overline{F}(M) \arrow[d]\arrow[r]   & \overline{F}(Z) \arrow[r] \arrow[d,equal]   &  0 \\
		\overline{F}(X) \arrow[r]   & \overline{F}(Y) \arrow[r]   & \overline{F}(Z)\arrow[r]   &  0. 
		\end{tikzcd}
		\]
 Since $M,Z\in \mc$, the upper row $\overline{F}(X') \to  \overline{F}(M) \to \overline{F}(Z)\to 0$ must be right exact by the argument above. Similarly, since $M\to Y$ is a right $\mc$-approximation, the map $\overline{F}(M)\to \overline{F}(Y)$ must be an epimorphism. These two facts imply that the lower row
 $$
  \overline{F}(X)\to \overline{F}(Y)\to \overline{F}(Z)\to 0
  $$
  is also right exact.

 Now assume $Y\in \mc$ in \eqref{Equation:SES}. Choose a right $\mc$-approximation $M\to Z$. Then $Y\to Z$ factors through $M\to Z$, so we get a commutative diagram
  \[
		\begin{tikzcd}[column sep=20, row sep=20]
		0\arrow[r] &X \arrow[r] \arrow[d] & Y \arrow[r] \arrow[d] & Z \arrow[d,equal]\arrow[r] & 0 \\
		0\arrow[r] &X' \arrow[r, ""] & M \arrow[r]  & Z \arrow[r]  & 0
		\end{tikzcd}
		\]
    where the rows are conflations. Hence, the sequence
    \[
    0\to X\to Y\oplus X'\to M\to 0
    \]
    must be a conflation, see \cite[Proposition 2.12]{Bue10}. Now since $M\in \mc$ it follows that 
    \[
    \overline{F}(X)\to \overline{F}(Y)\oplus \overline{F}(X')\to \overline{F}(M)\to 0
    \]
    is right exact by the previous case. Hence, the left hand square in the diagram
      \[
		\begin{tikzcd}[column sep=20, row sep=20]
		\overline{F}(X) \arrow[r] \arrow[d] & \overline{F}(Y) \arrow[r] \arrow[d] & \overline{F}(Z) \arrow[d,equal]\arrow[r] & 0 \\
		\overline{F}(X') \arrow[r, ""] & \overline{F}(M) \arrow[r]  & \overline{F}(Z) \arrow[r]  & 0 
		\end{tikzcd}
		\]
   must be cocartesian. Therefore, the cokernel of $\overline{F}(X)\to \overline{F}(Y)$ is isomorphic to the cokernel of $\overline{F}(X')\to \overline{F}(M)$. Now since the $M\to Z$ is a right $\mc$-approximation, the bottom row must be right exact, and so the cokernel must be $\overline{F}(Z)$. This shows that the top row is right exact.

  We now prove the general case. Let $0\to X\to Y\to Z\to 0$ be an arbitrary conflation in $\ec$. Choose a deflation $M_1\to X$ and a deflation $M_2\to Y$ with $M_1,M_2\in \mc$. Then we get a commutative diagram
   \[
		\begin{tikzcd}[column sep=20, row sep=20]
		0\arrow[r] &M_1 \arrow[r] \arrow[d] & M_1\oplus M_2 \arrow[r] \arrow[d] & M_2 \arrow[d]\arrow[r] & 0 \\
		0\arrow[r] &X \arrow[r, ""] & Y \arrow[r]  & Z \arrow[r]  & 0
		\end{tikzcd}
		\]
  where the top row is a split exact sequence, $M_1\to Y$ is given by the composite $M_1\to X\to Y$, and $M_2\to Z$ is given by the composite $M_2\to Y\to Z$. Since the vertical morphisms are deflations, we can take their kernels to get a conflation $0\to X'\to Y'\to Z'\to 0$ by the $(3\times3)$-Lemma, see \cite[Corollary 3.6]{Bue10}. Similarly, we can find a commutative diagram
  \[
		\begin{tikzcd}[column sep=20, row sep=20]
		0\arrow[r] &N_1 \arrow[r] \arrow[d] & N_1\oplus N_2 \arrow[r] \arrow[d] & N_2 \arrow[d]\arrow[r] & 0 \\
		0\arrow[r] &X' \arrow[r, ""] & Y' \arrow[r]  & Z'\arrow[r]  & 0
		\end{tikzcd}
		\]
  with $N_1,N_2\in \mc$, where the vertical morphisms are deflations, and where the top row is a split exact sequence. Combining these, we get the commutative diagram
    \[
		\begin{tikzcd}[column sep=20, row sep=20]
  0\arrow[r] &N_1 \arrow[r] \arrow[d] & N_1\oplus N_2 \arrow[r] \arrow[d] & N_2 \arrow[d]\arrow[r] & 0 \\
		0\arrow[r] &M_1 \arrow[r] \arrow[d] & M_1\oplus M_2 \arrow[r] \arrow[d] & M_2 \arrow[d]\arrow[r] & 0 \\
		0\arrow[r] &X \arrow[r, ""] \arrow[d] & Y \arrow[r] \arrow[d]  & Z \arrow[r] \arrow[d] & 0. \\
  & 0& 0& 0&
		\end{tikzcd}
		\]
  Applying $\overline{F}$, we get the commutative diagram
  \[
		\begin{tikzcd}[column sep=20, row sep=20]
  0\arrow[r] &\overline{F}(N_1) \arrow[r] \arrow[d] & \overline{F}(N_1)\oplus \overline{F}(N_2) \arrow[r] \arrow[d] & \overline{F}(N_2) \arrow[d]\arrow[r] & 0 \\
		0\arrow[r] &\overline{F}(M_1) \arrow[r] \arrow[d] & \overline{F}(M_1)\oplus \overline{F}(M_2) \arrow[r] \arrow[d] & \overline{F}(M_2) \arrow[d]\arrow[r] & 0 \\
		 &\overline{F}(X) \arrow[r, ""] \arrow[d] & \overline{F}(Y) \arrow[r] \arrow[d]  & \overline{F}(Z) \arrow[r] \arrow[d] & 0. \\
  & 0& 0& 0&
		\end{tikzcd}
		\]
  Note that the columns are right exact since $M_1,M_2\in \mc$, and the two top rows are exact since $\overline{F}$ preserves split exact sequences. Hence, the lower sequence must be right exact by the snake lemma. This shows that $\overline{F}$ is right exact. 
 \end{proof}

  Let $\nc$ be a $d$-exact category, and let $\ec'$ be a weakly idempotent complete exact category. A functor $F\colon \nc\to \ec'$ is \textit{exact} if it sends admissible $d$-exact sequence in $\nc$ to acyclic complexes in $\ec'$. By \Cref{Theorem:dCTImpliesdExactWeakIdemPotent} the inclusion of a $d$-cluster tilting subcategory into its ambient weakly idempotent complete exact category is exact.

 \begin{proposition}\label{Proposition:LiftExact}
     Let $F\colon \mc\to \ac$ and $\overline{F}\colon \ec\to \ac$ be as in \Cref{Proposition:UnivPropRightExact}. Then $\overline{F}$ is exact if and only if $F$ is exact.
 \end{proposition}

 \begin{proof}
    Clearly $F$ is exact if $\overline{F}$ is exact, so we only need to show the other implication. We first prove that $F$ sends any acyclic complex
     \[
     M_{\bullet}=(\cdots \to M_{1}\to M_0\to 0)
     \]
     in $\ec$ with components in $\mc$ to an acyclic complex
     \[
    F(M_{\bullet})=(\cdots \to F(M_{1})\to F(M_0)\to 0)
     \]
     in $\ac$. This is equivalent to showing that for any acyclic complex $M_\bullet$ concentrated in degrees $\geq 0$ with components in $\mc$ we have $H_{i}(F(M_\bullet))=0$ for $i\geq 0$. We prove this latter statement by induction on $i$. 
     
     The cases $i=0$ and $i=1$ follows immediately from $\overline{F}$ being right exact. Assume the induction hypothesis holds for $i=0,1,\dots,n$. We show it holds for $n+1$. For this, let $M_{\bullet}$ be any acyclic complex  in $\ec$ with components in $\mc$. Choose a $d$-kernel of $M_1\to M_0$, so we get a complex $$N_{\bullet}=(\cdots \to 0\to N_{d+1}\to \cdots \to N_{2}\to M_{1}\to M_0\to 0\to \cdots)$$ in $\mc$ which is acyclic in $\ec$. Since
     \[
     0\to \Hom_{\ec}(M,N_{d+1})\to \cdots \to \Hom_{\ec}(M,N_{2})\to \Hom_{\ec}(M,M_{1})\to \Hom_{\ec}(M,M_0)
     \]
     is exact for all $M\in \mc$, we can find dashed arrows
 \[
		\begin{tikzcd}[column sep=20, row sep=20]
		\cdots \arrow[r] & M_{d+2} \arrow[d] \arrow[r] &M_{d+1}\arrow[r] \arrow[d,dashed] & \cdots \arrow[r] &M_{2} \arrow[r] \arrow[d,dashed] & M_{1} \arrow[r] \arrow[d,equal] & M_0 \arrow[d,equal]\arrow[r] & 0 \\
		\cdots \arrow[r] & 0\arrow[r] & N_{d+1} \arrow[r] & \cdots \arrow[r] & N_{2} \arrow[r]  & M_{1} \arrow[r]  & M_0\arrow[r]& 0 
		\end{tikzcd}
		\]
     making the diagram commutative. This gives a morphism $M_\bullet\to N_\bullet$ of complexes. Let $C_\bullet$ denote its cone. Note that $C_\bullet$ is acyclic by \cite[Lemma 1.1]{Nee90}, since $M_\bullet$ and $N_\bullet$ are acyclic. Furthermore, $C_\bullet$ is homotopy equivalent to the complex
     \begin{equation}\label{Equation:RightBoundedComplex}
     C_\bullet'= (\cdots \to M_{3}\oplus N_{4}\to  M_{2}\oplus N_{3}\to N_{2}\to 0 \to \cdots )
     \end{equation}
     obtained by removing the columns containing $M_{1}$ and $M_0$ in the diagram above and then taking the total complex of the remaining part. Note that $ M_{i}\oplus N_{i+1}$ is in degree $i+1$. Applying $F$ componentwise, we get a complex
     \begin{equation*}
     F(C_\bullet')=(\cdots \to F(M_{3})\oplus F(N_{4})\to  F(M_{2})\oplus F(N_{3})\to F(N_{2})\to 0\to \cdots).
     \end{equation*}
    Since $C'_\bullet[-2]$ is an  acyclic complex with components in $\mc$ concentrated in degrees $\geq 0$, we know by the induction hypothesis that $H_{j}F(C'_\bullet)=0$ for all $j\leq n+2$. 
 Since the complex  $F(C_\bullet')$ is homotopy equivalent to $F(C_\bullet)$, this implies that 
 $$H_{j}F(C_\bullet)=0$$
     for $j\leq n+2$. Now consider the triangle
     \[
     F(M_\bullet)\to F(N_\bullet)\to F(C_\bullet)\to F(M_\bullet)[1]
     \]
     in $K^b(\ac)$ obtained by applying $F$ to the triangle $M_\bullet\to N_\bullet\to C_\bullet\to M_\bullet[1]$ in $K^b(\mc)$.  By assumption, $F(N_\bullet)$ is acyclic, so from the long exact sequences in homology it follows that 
     $$H_{j+1}F(C_\bullet)\cong H_{j}F(M_\bullet)$$ 
     for all $j\in \mathbb{Z}$. Hence, $H_{n+1}F(M_\bullet)\cong H_{n+2}F(C_\bullet)=0$, which proves the claim.

Now we show that $\overline{F}$ is exact. We know from \Cref{Proposition:UnivPropRightExact} that $\overline{F}$ is right exact, so we only need to show that $\overline{F}$ sends inflations to monomorphisms. To this end, let $X\to Y$ be an inflation, and let $Z$ denote its cokernel. Assume that $Y\in \mc$. By \Cref{Reformulation:d-CT} we can find exact sequences
\[
0\to M_d\to \cdots \to M_1\to X\to 0 \quad \text{and} \quad 0\to Z\to N^1\to \cdots \to N^d\to 0
\]
with $M_i,N^i\in \mc$ for all $i$. Combining these, we get a long exact sequence
\[
0\to M_d\to \cdots \to M_1\to Y\to N^1\to \cdots \to N^d\to 0
\]
with components in $\mc$. Since $\overline{F}$ coincides with $F$ on $\mc$ and $F$ sends acyclic complexes concentrated in degrees $\leq 0$ with components in $\mc$ to exact sequence, the sequence
\[
0\to \overline{F}(M_d)\to \cdots \to \overline{F}(M_1)\to \overline{F}(Y)\to \overline{F}(N^1)\to \cdots \to \overline{F}(N^d)\to 0
\]
must be exact. Since $\overline{F}$ is right exact, $\overline{F}(X)$ is isomorphic to the cokernel of $\overline{F}(M_2)\to \overline{F}(M_1)$. Hence, $\overline{F}(X)\to \overline{F}(Y)$ must be a monomorphism. 

 Now assume $X\to Y$ is an arbitrary inflation in $\ec$. Choose an inflation $Y\to M$ with $M\in \mc$. Then the composite $X\to Y\to M$ must also be an inflation. Hence, the composite $\overline{F}(X)\to \overline{F}(Y)\to \overline{F}(M)$ must be a monomorphism, so $\overline{F}(X)\to \overline{F}(Y)$ is a monomorphism. This shows that $\overline{F}$ is exact, which proves the claim.
 \end{proof}

To prove the universal property of $\ec$, we must relax the assumptions in \Cref{Proposition:LiftExact} so that the functor only takes values in an exact category. This requires the following lemmatas, which tells us how $\ec$ can be obtained as the extension-closure of smaller subcategories. The subcategories in question are 
 \[
\ec_{m,n}\colonequals\{E\in \ec\mid \Ext^i_\ec(E,\mc)=0=\Ext^j_\ec(\mc,E) \text{ for all } 0<i<m \text{ and } 0<j<n\}
\]
where $m,n>0$ and $m+n=d$.

\begin{lemma}\label{Lemma:DescriptionAsKerinD-Exact}
  Let $E\in \ec$, and let $m,n$ be positive integers satisfying $m+n=d$. Then $E$ is contained in $\ec_{m,n}$ if and only if there exists an admissible $d$-exact sequence
  \[
  0\to M_{d+1}\to M_{d}\to \cdots \to M_0\to 0
  \]
where $E$ is isomorphic to the image of $M_{n+1}\to M_{n}$.
\end{lemma}

\begin{proof}
     Note that  $\Ext^j_\ec(\mc,E)=0$ for $0<j<n$  if and only if we can find an exact sequence
     \[
     0\to M_{d-n}\to \cdots \to M_{1}\to M_0\to E\to 0.
     \]
     Similarly,  $\Ext^i_\ec(E,\mc)=0$ for $0<i<m$ if and only if we can find an exact sequence 
     \[
     0\to E\to N^0\to \cdots \to N^{d-m}\to 0.
     \]
    Hence, $E$ belongs to $\ec_{m,n}$ if and only if there exists an exact sequence
    \[
     0\to M_{d-n}\to \cdots \to M_{1}\to M_0\to N^0\to N^1\to \cdots \to N^{d-m}\to 0
    \]
     in $\ec$ with terms in $\mc$, and such that $E$ is the image of $M_{0}\to N^0$. Since $m+n=d$, this sequence must have $d+2$ non-zero terms, and hence be an admissible $d$-exact sequence. This proves the claim.
\end{proof}

An object $E\in \ec$ is \textit{filtered} by a class of objects $\mathcal{S}$ if there exists a sequence of inflations $$0=E^0\to E^1\to \cdots \to E^k\to E^{k+1}=E$$ such that the cokernel of $E^i\to E^{i+1}$ is in $\mathcal{S}$ for all $0\leq i\leq k$. The subcategory of $\ec$ consisting of all objects that can be filtered by objects in $\mathcal{S}$ is denoted $\operatorname{Filt}(\mathcal{S})$. Note that $\operatorname{Filt}(\mathcal{S})$ is the smallest subcategory of $\ec$ closed under extensions and containing $\mathcal{S}$.

\begin{lemma}\label{Lemma:d-CTFiltrations}
   For any $E\in \ec$ there exists $M\in \mc$ such that 
    \[
    E\oplus M\in\operatorname{Filt}(\bigcup_{\substack{m,n>0\\m+n=d}}\ec_{m,n}).
    \]
\end{lemma}

\begin{proof}
We want to prove that $\ec'=\ec$, where
     \[
      \ec'=\{E\in \ec\mid E\oplus M\in \operatorname{Filt}(\bigcup_{\substack{m,n>0\\m+n=d}}\ec_{m,n})\text{ for some }M\in \mc\}.
     \]
      
We first show that $\ec'$ is closed under extensions. Indeed, let 
     \[
     0\to E_3\to E_2\to E_1\to 0
     \]
     be a conflation in $\ec$ with $E_1,E_3\in \ec'$. By definition there exist $M_1,M_3\in \mc$ such that \begin{align*}
     E_1\oplus M_1\in \operatorname{Filt}(\bigcup_{\substack{m,n>0\\m+n=d}}\ec_{m,n}) \quad \text{and} \quad E_3\oplus M_3\in \operatorname{Filt}(\bigcup_{\substack{m,n>0\\m+n=d}}\ec_{m,n}).
    \end{align*} 
    Adding the trivial conflations $0\to M_1\xrightarrow{1}M_1\to 0\to 0$ and  $0\to 0\to M_3\xrightarrow{1} M_3\to 0$ to the conflation above, we get the conflation
    \[
     0\to E_1\oplus M_1\to E_2 \oplus M_1\oplus M_3\to E_3\oplus M_3\to 0.
     \]
     Hence, $E_2 \oplus M_1\oplus M_3$ is a filtration by objects in $\ec_{m,n}$ for $m+n=d$, and so $E_2$ must be contained in $\ec'$
   
   Next we prove by induction on $k$ that the subcategory
     \[
     \mc^{k}(\ec)=\{E\in \ec \mid \exists\text{ }0\to E\to M^{1}\to \dots \to M^{k} \to 0 \text{ exact, }M^i\in \mc \text{ for all }i\}
     \]
     is contained in $\ec'$. The case $k=d$ implies that $\ec=\mc^{d}(\ec)=\ec'$, which proves the lemma. 
     
     Clearly $\mc^1(\ec)=\mc$ and $\mc^2(\ec)=\ec_{d-1,1}$ are both contained in $\ec'$. Assume $\mc^{k}(\ec)$ is contained in $\ec'$ for some $2\leq k<d$. We show that $\mc^{k+1}(\ec)$ is contained in $\ec'$. To this end,  let $E\in \mc^{k+1}(\ec)$ be arbitrary, and let
     \[
     0\to E\to M_{k}\to \dots \to M_0\to 0
     \]
     be an exact sequence with $M_i\in \mc$ for all $i$. Choose an exact sequence
     \[
     0\to M_{d+k}\to \cdots \to M_{k+1}\to E\to 0
     \]
     in $\ec$ with $M_i\in \mc$ for all $i$. Let $M_\bullet$ denote the acyclic complex 
     \[
     0\to M_{d+k}\to \cdots \to M_0\to 0 
     \]
     obtained by combining the two exact sequences and concentrated in degrees $d+k,\dots,0$. Let $F$ be the kernel of $M_{1}\to M_{0}$, and let  
     $0\to  N_{d+1}\to \cdots \to N_{2}\to F\to 0$
     be an exact sequence in $\ec$ with $N_i\in \mc$ for all $i$. Let $N_\bullet$ denote the acyclic complex 
     \[
     0\to N_{d+1}\to \cdots \to N_{2}\to M_{1}\to M_0\to 0
     \]
      concentrated in degrees $d+1,\dots,0$ so that $N_{1}=M_{1}$ and $N_0=M_0$. Since
     \[
     0\to \Hom_{\ec}(M,N_{d+1})\to \cdots \to \Hom_{\ec}(M,N_{2})\to \Hom_{\ec}(M,M_{1})\to \Hom_{\ec}(M,M_0)
     \]
     is exact for all $M$ in $\mc$, we can find dashed arrows
 \[
		\begin{tikzcd}[column sep=15, row sep=15]
		 \cdots  \arrow[r] & 0\arrow[r] \arrow[d] & M_{d+k} \arrow[r] \arrow[d] & \cdots \arrow[r]&M_{d+2}\arrow[r] \arrow[d] & M_{d+1} \arrow[r] \arrow[d,dashed] & \cdots \arrow[r] &M_{2} \arrow[r]  \arrow[d,dashed] & M_{1} \arrow[d,equal] \arrow[r]  & M_0 \arrow[d,equal]\arrow[r] & 0 \\
  \cdots \arrow[r] & 0\arrow[r] & 0 \arrow[r] & \cdots \arrow[r]&0\arrow[r]  & N_{d+1} \arrow[r] & \cdots \arrow[r] &N_{2} \arrow[r] & M_{1} \arrow[r]  & M_0 \arrow[r] & 0
		\end{tikzcd}
		\]
     making the diagram commutative. This gives a morphism $M_\bullet\to N_\bullet$ of complexes. Let $C_\bullet$ be its cone, which is acyclic by \cite[Lemma 1.1]{Nee90}.
     
     We claim that 
     \begin{equation}\label{Equation:InFc}
         Z_{i}(C_\bullet)\in \ec'
     \end{equation}
     for $i\leq k+1$. Indeed, by definition of the cone we have an exact sequence
    \[
    0\to Z_{i}(C_\bullet)\to N_{i}\oplus M_{i-1}\to N_{i-1}\oplus M_{i-2}\to \cdots \to N_{2}\oplus M_{1} \to N_{1}\oplus M_0\to N_0\to 0.
    \]
    Since the rightmost morphisms $M_{1}\to N_{1}$ and $M_0\to N_0$ are identities, we can remove these to get the exact sequence 
    \[
    0\to Z_{i}(C_\bullet)\to  N_{i}\oplus M_{i-1}\to N_{i-1}\oplus M_{i-2}\to \cdots \to N_{2}\to 0.
    \]
    Since $i\leq k+1$, this implies that $Z_{i}(C_\bullet)\in \mc^{k}(\ec)$, and hence it must be contained in $\ec'$ by the induction hypothesis.
    
    Next consider the componentwise split exact sequence
     \[
     0\to N_\bullet\to C_\bullet\to M_\bullet[1]\to 0
     \]
of complexes. This gives commutative diagrams
 \begin{equation}\label{CyclesDiagram}
		\begin{tikzcd}[column sep=20, row sep=20]
  0\arrow[r] &Z_i(N_\bullet) \arrow[r] \arrow[d] & N_i \arrow[r] \arrow[d] & Z_{i-1}(N_\bullet) \arrow[d]\arrow[r] & 0 \\
		0\arrow[r] &Z_i(C_\bullet) \arrow[r] \arrow[d] & C_i \arrow[r] \arrow[d] & Z_{i-1}(C_\bullet) \arrow[d]\arrow[r] & 0 \\
		0\arrow[r] &Z_{i-1}(M_\bullet) \arrow[r, ""] & M_{i-1} \arrow[r]   & Z_{i-2}(M_\bullet) \arrow[r]  & 0
		\end{tikzcd}
	\end{equation}
    for each integer $i$ where the middle column is a split exact sequence. Note that the rows of the diagram are conflations since the complexes $N_\bullet$ and $C_\bullet$ and $M_\bullet$ are acyclic. 
    
    We claim that the left column 
    \begin{align*}
        Z_i(N_\bullet)\to Z_i(C_\bullet)\to Z_{i-1}(M_\bullet)
    \end{align*}
    of \eqref{CyclesDiagram} is a conflation for all $i$. Indeed, by the $(3\times3)$-lemma \cite[Lemma 3.6]{Bue10} it is a conflation if $Z_{i-1}(N_\bullet)\to Z_{i-1}(C_\bullet)\to Z_{i-2}(M_\bullet)$ is a conflation. Hence, it suffices to prove the claim for $i$ sufficiently small. Since $$Z_i(N_\bullet)=Z_i(C_\bullet)=Z_{i-1}(M_\bullet)=0$$
    for $i<0$, the claim holds.

    Now recall that we have $E=Z_{k}(M_\bullet)$. Consider the two leftmost columns of the diagram \eqref{CyclesDiagram}
    \begin{equation*}
		\begin{tikzcd}[column sep=20, row sep=20]
  &Z_{k+1}(N_\bullet) \arrow[r] \arrow[d] & N_{k+1}  \arrow[d] \\
		&Z_{k+1}(C_\bullet) \arrow[r] \arrow[d] & C_{k+1} \arrow[d]  \\
		&E \arrow[r, ""] & M_{k} 
		\end{tikzcd}
	\end{equation*}
    for $i=k+1$. Let $G$ be the pushout of $Z_{k+1}(N_\bullet)\to N_{k+1}$ along $Z_{k+1}(N_\bullet)\to Z_{k+1}(C_\bullet)$. Then we get the commutative diagram 
     \begin{equation}\label{ThreeExactSeqDiag}
		\begin{tikzcd}[column sep=20, row sep=20]
  0\arrow[r] &Z_{k+1}(N_\bullet) \arrow[r] \arrow[d] & N_{k+1} \arrow[r,equal] \arrow[d] & N_{k+1} \arrow[d]\arrow[r] & 0 \\
		0\arrow[r] &Z_{k+1}(C_\bullet) \arrow[r] \arrow[d] & G \arrow[r] \arrow[d] & C_{k+1} \arrow[d]\arrow[r] & 0 \\
		0\arrow[r] &E \arrow[r,equal] & E \arrow[r]   & M_{k} \arrow[r]  & 0 
		\end{tikzcd}
	\end{equation}
 for some object $G$, where the columns are conflations. Since the top left square is cocartesian and $Z_{k+1}(N_\bullet)\to N_{k+1}$ is an inflation with cokernel $Z_{k}(N_\bullet)$, the same must hold for the morphism $Z_{k+1}(C_\bullet)\to G$. Note that $Z_{k}(N_\bullet)\in \ec'$ by \Cref{Lemma:DescriptionAsKerinD-Exact} since $N_\bullet$ is an admissible $d$-exact sequence. Also, $Z_{k+1}(C_\bullet)\in \ec'$ by \eqref{Equation:InFc}. Since  $\ec'$ is closed under extensions, it follows that $G\in \ec'$. Finally, since the lower right square in \eqref{ThreeExactSeqDiag} is bicartesian, we have a conflation
 \[
 0\to G\to C_{k+1}\oplus E\to M_{k}\to 0.
 \]
Since $\ec'$ is closed under extensions, it follows that $C_{k+1}\oplus E\in \ec'$. Since $C_{k+1}\in \mc$, it follows from the definition of $\ec'$ that $E\in \ec'$. This proves the claim. 
\end{proof}

\begin{remark}
    We have seen in \Cref{Reformulation:d-CT} that $\mc$ generates $\ec$ by finite resolutions and coresolutions. \Cref{Lemma:d-CTFiltrations} shows the surprising fact that $\ec$ can also be generated via filtrations and direct summands.
\end{remark}

We can now prove the universal property of the ambient exact category of a $d$-cluster tilting subcategory, under the assumption of weak idempotent completeness.

\begin{theorem}\label{Theorem:UnivProperty}
   The following hold.
    \begin{enumerate}
        \item\label{Theorem:UnivProperty:1} The inclusion $\mc\to \ec$ is exact.
        \item\label{Theorem:UnivProperty:2} Let $\ec'$ be a weakly idempotent complete exact category and $F\colon \mc\to \ec'$ an exact functor. Then there exists a unique (up to natural isomorphism) exact functor $\overline{F}\colon \ec\to \ec'$ extending $F$.
    \end{enumerate} 
\end{theorem}

\begin{proof}
    Part \eqref{Theorem:UnivProperty:1} follows by \Cref{Theorem:dCTImpliesdExactWeakIdemPotent}. 
    
    We prove part \eqref{Theorem:UnivProperty:2}. First note that the category $\ec'$ is equivalent to an extension closed subcategory of $\mathcal{L}(\ec')$ by \cite[Proposition A.2]{Kel90}. For simplicity we identify $\ec'$ with this subcategory, and assume it is closed under isomorphisms in $\mathcal{L}(\ec')$. Let $G$ denote the composite $\mc\xrightarrow{F}\ec'\to \mathcal{L}(\ec')$. By \Cref{Proposition:UnivPropRightExact} and \Cref{Proposition:LiftExact} it has a unique extension to an exact functor $\overline{G}\colon \ec\to \mathcal{L}(\ec')$. It remains to show that the image of $\overline{G}$ lies in $\ec'$. To this end, consider the preimage 
    \[
    (\overline{G})^{-1}(\ec')\colonequals \{E\in \ec\mid \overline{G}(E)\in \mathcal{E}'\}.
    \] 
   Since $\overline{G}$ sends admissible $d$-exact sequences to acyclic complexes in $\ec'$, the subcategories $\ec_{m,n}$ must be contained in $\overline{G}^{-1}(\ec')$ for all integers $m,n$ satisfying $m+n=d$ by \Cref{Lemma:DescriptionAsKerinD-Exact}.  Also, since $\mathcal{E}'$ is closed under extensions and additive complements in $\mathcal{L}(\ec')$, the same must hold for $\overline{G}^{-1}(\ec')$ as a subcategory of $\ec$. Now by \Cref{Lemma:d-CTFiltrations} we know that $\ec$ is the smallest subcategory satisfying these properties. Hence $\overline{G}^{-1}(\ec')=\ec$, which proves the claim. 
\end{proof}

In the following we use concepts for $2$-categories, following \cite{nlab:2-category,nlab:2-functor,nlab:equivalence_of_2-categories,nlab:pseudonatural_transformation}. Let $d\operatorname{-EX}$ be the strict $2$-category whose objects are weakly idempotent complete $d$-exact categories, whose $1$-morphisms are $d$-exact functors, and whose $2$-morphisms are natural transformations. Let $d\operatorname{-CT_{ex}}$ be the strict $2$-category whose objects are pairs $(\mc,\ec)$ where $\ec$ is a weakly idempotent complete exact category and $\mc$ is a $d$-cluster tilting subcategory of $\ec$, whose $1$-morphisms $(\mc,\ec)\to (\mc',\ec')$ are exact functors $\ec\to \ec'$ which sends $\mc$ to $\mc'$, and whose $2$-morphisms are natural transformations. Since any $d$-cluster tilting subcategory inherits a $d$-exact structure, we have a strict $2$-functor 
$$\operatorname{res}\colon d\operatorname{-CT_{ex}}\to d\operatorname{-EX}$$
sending a pair $(\mc,\ec)$ to $\mc$ with its induced $d$-exact structure.

\begin{theorem}\label{Theorem:2equivalence}
The $2$-functor $\operatorname{res}\colon d\operatorname{-CT_{ex}}\to d\operatorname{-EX}$ gives an equivalence of $2$-categories.
\end{theorem}

\begin{proof}
    We construct a pseudofunctor $\ec(-)\colon d\operatorname{-EX}\to d\operatorname{-CT_{ex}}$ as follows. To an object  $\mc$ in $d\operatorname{-EX}$ we associate the pair $(\mc,\ec(\mc))$ where $\mc$ is identified with its essential image under the Yoneda embedding $\mc\to \ec(\mc)$. This is well-defined by \Cref{Cor:WeaklyIdemdExactdCT}. 
    To a $1$-morphism $F\colon \mc\to \nc$ we associate the exact functor $\ec(F)\colon \ec(\mc)\to \ec(\nc)$ obtained by applying the universal property in \Cref{Theorem:UnivProperty} to the exact functor 
    $$\mc\xrightarrow{F} \nc\to \ec(\nc).$$ Finally, to a natural transformation $F\to G$ we associate its unique extension to a natural transformation $\ec(F)\to \ec(G)$. 
    
    We show that the pair $(\operatorname{res},\ec(-))$ form an equivalence of $2$-categories. Clearly the composite 
    $$d\operatorname{-EX}\xrightarrow{\ec(-)} d\operatorname{-CT_{ex}}\xrightarrow{\operatorname{res}}d\operatorname{-EX}$$ is equal to the identity $2$-functor on $d\operatorname{-EX}$. Conversely, if $(\mc,\ec)$ is an object in $d\operatorname{-CT_{ex}}$, then by \Cref{Theorem:UnivProperty} we can find exact functors $\ec\to \ec(\mc)$ and $\ec(\mc)\to \ec$ commuting with the inclusions of $\mc$ into $\ec$ and $\ec(\mc)$. By the uniqueness part of \Cref{Theorem:UnivProperty} the composites $\ec(\mc)\to \ec\to \ec(\mc)$ and $\ec\to \ec(\mc)\to \ec$ are naturally isomorphic to the identity functors. Hence, for each object $(\mc,\ec)$ we have an equivalence $(\mc,\ec)\cong (\mc,\ec(\mc))$ in the $2$-category $d\operatorname{-CT_{ex}}$. Furthermore these equivalences can be extended to a pseudonatural equivalence from
    $$d\operatorname{-CT_{ex}}\xrightarrow{\operatorname{res}}d\operatorname{-EX}\xrightarrow{\ec(-)} d\operatorname{-CT_{ex}}$$
    to the identity $2$-functor on $d\operatorname{-CT_{ex}}$. This can be seen from the natural isomorphism in \Cref{Theorem:UnivProperty} being unique if they are also assumed to be identity on objects in $\mc$. The claim follows.
\end{proof}

 We immediately get the following uniqueness results of the ambient exact category of a $d$-cluster tilting subcategory from \Cref{Theorem:2equivalence}, assuming weak idempotent completeness. For $d$-abelian categories it has already been shown in \cite{Kva22}.  

 \begin{corollary}\label{Theorem:UniquenessAmbientExact}
     Assume we have an equivalence $\mc\xrightarrow{\cong}\nc$ of $d$-exact categories, where $\nc$ is a $d$-cluster tilting subcategory of a weakly idempotent complete exact category $\mathcal{E}'$. Then there exists an exact equivalence $\ec\xrightarrow{\cong} \ec'$ making the diagram
     \[
     \begin{tikzcd}
  \ec \arrow[r,"\cong"] & \ec'  \\
  \mc \arrow[u] \arrow[r,"\cong"]&\nc \arrow[u] & 
 \end{tikzcd}
    \]
    commutative, where the vertical functors are the canonical inclusions.
 \end{corollary}

\begin{corollary}
   We have an exact equivalence $\ec\cong \ec(\mc)$.
\end{corollary}

The following example shows that \Cref{Theorem:UniquenessAmbientExact} does not hold without the assumption of weak idempotent completeness.

\begin{example}\label{Example:NonUniqueness}
Let $\mc$ be any $d$-cluster tilting subcategory of an abelian category $\ac$, and assume there exists a simple object $S\in \ac$ which is not contained in $\mc$ (this holds for example whenever $d\geq 2$ and $\ac=\operatorname{mod}\Lambda$ where $\Lambda$ is a non-semisimple finite-dimensional algebra). Let $\ec$ be the subcategory consisting of all objects in $\ac$ except $S$. Then $\ec$ is clearly extension-closed in $\ac$, and therefore inherits an exact structure. Now the smallest exact structure on $\ac$ which contains the conflations in $\ec$ must necessarily contain all exact sequences in $\ac$. Hence, the weak idempotent completion of $\ec$ is equivalent to $\ac$ as an exact category. In particular,  $$\Ext^i_{\ec}(X,Y)\cong \Ext^i_{\ac}(X,Y)$$
for all $i>0$ by \cite[Remark 1.12.3]{Nee90} and so $\mc$ must be $d$-cluster tilting in $\ec$. However, $\ec$ can't be equivalent to $\ac$, since $\ec$ is not weakly idempotent complete.
\end{example}

\section{Algebraic \texorpdfstring{$(d+2)$}{}-angulated categories}\label{Section:(d+2)-Angulated}

Here we consider algebraic $(d+2)$-angulated categories in the sense of \cite{GKO13} and \cite{Jas16}, and prove that they are $d$-cluster tilting in an algebraic triangulated category, using \Cref{Cor:WeaklyIdemdExactdCT}. We first prove some general results on the relationship between a $d$-cluster tilting subcategory and its ambient exact category. 

Let $\mc$ be a $d$-exact category. Following \cite{Jas16} we call an object $P\in \mc$ \textit{projective} if for every admissible epimorphism $M\to M'$ the map $\mc(P,M)\to \mc(P,M')$ is an epimorphism. The $d$-exact category $\mc$ is \textit{projectively generated} if for every object $M\in \mc$ there exists an admissible epimorphism $P\to M$ with $P$ projective, and $\mc$ \textit{has enough projectives} if for every object $M\in \mc$ there exists an admissible $d$-exact sequence
\[
0\to M'\to P_d\to \cdots \to P_1\to M\to 0
\]
with $P_i$ being projective for all $1\leq i\leq d$. The notion of \textit{injective} object, \textit{injectively cogenerated}, and \textit{having enough injectives} is defined dually.

In the following we show how projectivity can equivalently be characterized in the ambient exact category. Recall that a $d$-cluster tilting subcategory $\mc$ of an exact category $\ec$ is called $d\mathbb{Z}$\textit{-cluster tilting} if $\Ext^i_{\ec}(X,Y)=0$ for all $i\notin d\mathbb{Z}$.

\begin{lemma}\label{Lemma:ProjectivedCTvsExact}
 Let $\mc$ be a $d$-cluster tilting subcategory of a weakly idempotent complete exact category $\ec$. The following hold.
    \begin{enumerate}
\item\label{Lemma:ProjectivedCTvsExact:1} An object is projective in the $d$-exact structure on $\mc$ if and only if it is projective in the exact structure on $\ec$.
        \item\label{Lemma:ProjectivedCTvsExact:2} $\mc$ is projectively generated if and only if $\ec$ has enough projectives.
        \item\label{Lemma:ProjectivedCTvsExact:3} $\mc$ has enough projectives if and only if $\ec$ has enough projectives and $\mc$ is a $d\mathbb{Z}$-cluster tilting subcategory of $\ec$.
    \end{enumerate}
\end{lemma}

\begin{proof}
By definition $P$ is projective in $\mc$ if and only if 
$$\Hom_\ec(P,-)|_{\mc}\colon \mc\to \operatorname{Ab}$$
sends admissible $d$-exact sequences to acyclic complexes of abelian groups. By the dual of \Cref{Proposition:LiftExact} this holds if and only if the functor
\[
\Hom_\ec(P,-)\colon \ec\to \operatorname{Ab}
\]
is exact. Since this is equivalent to $P$ being projective in $\ec$, this proves \eqref{Lemma:ProjectivedCTvsExact:1}.

For \eqref{Lemma:ProjectivedCTvsExact:2}, assume $\ec$ has enough projectives. Then for any $M\in \mc$ there exists a deflation $P\to M$ with $P$ projective, which must be an admissible epimorphism by the description of the $d$-exact structure on $\mc$, see \Cref{Theorem:dCTImpliesdExactWeakIdemPotent}. Hence, $\mc$ is projectively generated. 

Conversely, assume $\mc$ is projectively generated.  For any $E\in \ec$ we can find a deflation $M \to E$ with $M\in \mc$ since $\mc$ is generating in $\ec$. Furthermore, we can find a deflation $P\to M$ with $P$ projective since $\mc$ is projectively generated.  Then the composite $P\to M\to E$ must be a deflation, which shows that $\ec$ has enough projectives. 

Finally, for \eqref{Lemma:ProjectivedCTvsExact:3} note that if $\ec$ has enough projectives, then $\mc$ is $d\mathbb{Z}$-cluster tilting if and only if it is closed under $d$-syzygies, i.e. for any exact sequence 
\[
0\to M'\to P_d\to \cdots \to P_1\to M\to 0
\]
in $\ec$ with $P_i$ being projective for $1\leq i\leq d$ and $M\in \mc$, then $M'\in \mc$, see \cite[Definition-Proposition 2.15]{IJ17}. This is equivalent to $\mc$ having enough projectives, proving part \eqref{Lemma:ProjectivedCTvsExact:3}. 
\end{proof}

Recall that a \textit{Frobenius} $d$\textit{-exact} category is a $d$-exact category which has enough projectives and injectives, and for which the projective and injective objects coincide \cite[Definition 5.5]{Jas16}. 

\begin{proposition}\label{Proposition:FrobdExactImpliesFrobExact}
    Let $\ec$ be a weakly idempotent complete exact category, and let $\mc$ be a $d$-cluster tilting subcategory of $\ec$. Then $\mc$ is a Frobenius $d$-exact category if and only if $\ec$ is a Frobenius exact category and $\mc$ is $d\mathbb{Z}$-cluster tilting in $\ec$.
\end{proposition}

\begin{proof}
 This is an immediate consequence of \Cref{Lemma:ProjectivedCTvsExact} and its dual.   
\end{proof}

\begin{theorem}\label{Theorem:FrobdExactIsdZCTInFrobExact}
    Let $\mc$ be a weakly idempotent complete Frobenius $d$-exact category. Then $\mc$ is equivalent as a $d$-exact category to a $d\mathbb{Z}$-cluster tilting subcategory of a weakly idempotent complete Frobenius exact category.
\end{theorem}

\begin{proof}
    This follows from \Cref{Cor:WeaklyIdemdExactdCT} \eqref{Cor:WeaklyIdemdExactdCT:2} and \Cref{Proposition:FrobdExactImpliesFrobExact}. 
\end{proof}

Let $\mc$ be a $d$-exact category. The stable category $\underline{\mc}$ of $\mc$ is the additive category with the same objects as $\mc$, and whose morphisms spaces are quotients of the morphism spaces in $\mc$ by the ideal of morphisms factoring through projective objects. If $\mc$ is Frobenius, then by \cite[Theorem 5.11]{Jas16} the category $\underline{\mc}$ has the structure of a $(d+2)$-angulated category in the sense of \cite[Definition 1.1]{GKO13}. A $(d+2)$-angulated category $\fc$ is called \textit{algebraic} if  there exists a Frobenius $d$-exact category $\mc$ and an equivalence $\fc\cong \underline{\mc}$ of $(d+2)$-angulated categories \cite[Definition 5.12]{Jas16}.

\begin{lemma}\label{Lemma:WeaklyIdemComFrobModel}
    Let $\fc$ be a weakly idempotent  complete algebraic $(d+2)$-angulated category. Then there exists a weakly idempotent complete Frobenius $d$-exact category $\mc$ and an equivalence $\fc\cong \underline{\mc}$ of $(d+2)$-angulated categories.
\end{lemma}

\begin{proof}
    By assumption there exists a Frobenius $d$-exact category $\nc$ and an equivalence $\fc\cong \underline{\nc}$ of $(d+2)$-angulated categories. Let $\mc=\hat{\nc}$ denote the weak idempotent completion of $\nc$. By \cite[Corollary 5.6]{KMS24} (see also \Cref{Remark:WeakIdemPotentCompletion}) the category $\mc$ can be endowed with a unique $d$-exact structure having the property that a sequence in $\mc$ with components in $\nc$ is an admissible $d$-exact sequence in $\mc$ if and only if it is an admissible $d$-exact sequence in $\nc$. This implies that an object in $\nc$ is projective (resp injective) in $\mc$ if and only if it is projective (resp injective) in $\nc$, and hence $\mc$ must be Frobenius $d$-exact. Finally, since $\underline{\nc}$ is weakly idempotent complete, the inclusion $\nc\to \mc$ must induce an equivalence $\underline{\nc}\to \underline{\mc}$  of $(d+2)$-angulated categories. This proves the claim.
\end{proof}

By \cite[Theorem 1]{GKO13} a $d\mathbb{Z}$-cluster tilting subcategory of a triangulated category inherits the structure of a $(d+2)$-angulated category. This is called the "standard construction". The main theorem of this subsection states that up to equivalence, all weakly idempotent complete algebraic $(d+2)$-angulated categories arise via the standard construction.

\begin{theorem}\label{Theorem:Algd+2AngulatedisdCT}
    Let $\fc$ be a weakly idempotent complete algebraic $(d+2)$-angulated category. Then there exists an equivalence $\fc\cong \mathcal{G}$ of $(d+2)$-angulated categories, where $\mathcal{G}$ is a $d\mathbb{Z}$-cluster tilting subcategory of an algebraic triangulated category $\mathcal{T}$.
\end{theorem}

\begin{proof}
 By \Cref{Lemma:WeaklyIdemComFrobModel} we can find a weakly idempotent complete Frobenius $d$-exact category $\mc'$ and an equivalence $\fc\cong \underline{\mc'}$ of $(d+2)$-angulated categories. By \Cref{Theorem:FrobdExactIsdZCTInFrobExact} there exists an equivalence $\mc'\cong \mc$ of $d$-exact categories where $\mc$ is a $d\mathbb{Z}$-cluster tilting subcategory of a weakly idempotent complete Frobenius exact category $\ec$. Hence, $\underline{\mc}$ must be $d\mathbb{Z}$-cluster tilting in the triangulated category $\underline{\ec}$ and therefore inherit the structure of a $(d+2)$-angulated category. By \cite[Theorem 5.16]{Jas16} this $(d+2)$-angulated category is equivalent to the one obtained from the Frobenius $d$-exact structure on $\mc$. Therefore, $\underline{\mathcal{M}}$ must be equivalent to $\underline{\mc}'$ as a $(d+2)$-angulated category, since $\mc$ and $\mc'$ are equivalent as $d$-exact categories. Combining this with the equivalence $\fc\cong \underline{\mc'}$ proves the claim. 
\end{proof}

We have the following result covering the idempotent complete case.

\begin{lemma}\label{IdempotentCompleteCaseTriang}
    Let $\fc$ be a $d$-cluster tilting subcategory of a triangulated category $\tc$. Then $\fc$ is idempotent complete if and only if $\tc$ is idempotent complete.
\end{lemma}

\begin{proof}
Clearly, if $\tc$ is idempotent complete, then $\fc$ is idempotent complete, so we only need to prove the converse. By \cite[Theorem 1.5]{BS01} the idempotent completion $\tilde{\tc}$ of $\tc$ has the structure of a triangulated category such that $\tc\to \tilde{\tc}$ is a triangulated functor. Note that $\fc$ is functorially finite in $\tilde{\tc}$ since it is functorially finite in $\tc$. Furthermore, $\fc$ is closed under direct summands in $\tilde{\tc}$ since $\fc$ is idempotent complete. Using this, we see that
\begin{align*}
 	\fc &=\{E\in \tilde{\tc}\mid \Hom_{\tilde{\tc}}(E,\fc[i])=0 \text{ for all }0<i<d\} \\
 		&=\{E\in \tilde{\tc}\mid \Hom_{\tilde{\tc}}(\fc,E[i])=0 \text{ for all }0<i<d\}
 \end{align*}
since $\fc$ satisfies the same properties as a subcategory of $\tc$. This implies $\fc$ is $d$-cluster tilting in $\tilde{\tc}$. Hence, for any $T\in \tilde{\tc}$ we can find triangles
 \[
 T_{i+1}\to M_{i}\to T_{i}\to T_{i+1}[1] \quad \text{for }0< i< d-1
 \]
 and
 \[
 M_{d}\to M_{d-1}\to T_{d-1}\to M_{d}[1]
 \]
 where $T_1=T$ and $M_i\in \mc$ for all $i$ by \cite[Corollary 3.3]{IY08}. Since $\tc\to \tilde{\tc}$ is triangulated, the subcategory $\tc$ must be closed under cones in $\tilde{\tc}$, and hence by downwards induction it follows that $T_i\in \tc$ for all $i$. In particular, $T=T_1\in \tc$, which shows that $\tc=\tilde{\tc}$, so $\tc$ must be idempotent complete. 
\end{proof}

We end this section by showing that the notion of algebraic $(d+2)$-angulated category given in \cite[Definition 3.2.5]{JKM22}  coincides with the notion of idempotent complete algebraic $(d+2)$-angulated category given in \cite[Definition 5.12]{Jas16} and used above. Note that the definition in \cite{JKM22} uses differential graded (DG) categories. We recall the necessary terminology we need below. For more information on DG categories we refer to \cite[Section 3.1]{JKM22} or \cite{Kel06}.

For a small DG category $\mathfrak{A}$ we let $D(\mathfrak{A})$ denote its derived category,  $D^c(\mathfrak{A})$ the thick subcategory of $D(\mathfrak{A})$ consisting of the compact objects, and $H^0(h)\colon H^0(\mathfrak{A})\to D^c(\mathfrak{A})$ the fully faithful functor given by taking the $0$th homology of the DG Yoneda embedding $h$. The DG category $\mathfrak{A}$ is called \textit{pre-}$(d+2)$\textit{-angulated} if $H^0(h)\colon H^0(\mathfrak{A})\to D^c(\mathfrak{A})$ induces an equivalence between $H^0(\mathfrak{A})$ and a $d\mathbb{Z}$-cluster tilting subcategory of $D^c(\mathfrak{A})$. In this case $H^0(\mathfrak{A})$ has a $(d+2)$-angulated structure by \cite[Theorem 1]{GKO13}. 

\begin{proposition}\label{(d+2)AngulatedDefCoincides}
    Let $\fc$ be a small idempotent complete $(d+2)$-angulated category. Then $\fc$ is algebraic if and only if there exists a pre-$(d+2)$-angulated DG category $\mathfrak{A}$ and an equivalence $H^0(\mathfrak{A})\cong \fc$ of $(d+2)$-angulated categories.
\end{proposition}

\begin{proof}
    Assume $\fc$ is algebraic. By \Cref{Theorem:Algd+2AngulatedisdCT} and \Cref{IdempotentCompleteCaseTriang} we can find an idempotent complete algebraic triangulated category $\tc$, a $d\mathbb{Z}$-cluster tilting subcategory $\mathcal{G}$ of $\tc$, and an equivalence $\fc\cong \mathcal{G}$ of $(d+2)$-angulated categories. Furthermore, we can find a small DG category $\mathfrak{A}$ and an equivalence $F\colon \tc\to D^c(\mathfrak{A})$ of triangulated categories which restricts to an equivalence $\mathcal{G}\cong H^0(\mathfrak{A})$ by \cite[Theorem 4.3]{Kel94}. This implies that $\mathfrak{A}$ is pre-$(d+2)$-angulated in such a way that $\mathcal{G}\cong H^0(\mathfrak{A})$ is an equivalence of $(d+2)$-angulated categories. Combining this with the equivalence $\fc\cong \mathcal{G}$ proves the "only if" direction of the claim.

    Now assume we have a pre-$(d+2)$-angulated DG category $\mathfrak{A}$ and an equivalence $H^0(\mathfrak{A})\cong \fc$ of $(d+2)$-angulated categories. It is well-known that $D(\mathfrak{A})$ is algebraic, see e.g. \cite[Section 7.5 Example (2)]{Kra07}. Since $D^c(\mathfrak{A})$ is a thick subcategory of $D(\mathfrak{A})$, it must also be algebraic, see e.g. \cite[Section 7.5 Lemma (3)]{Kra07}. Let $\ec$ be a Frobenius exact enhancement of $D^c(\mathfrak{A})$, and let $\mc$ be the full subcategory of $\ec$ consisting of all objects which are isomorphic in $D^c(\mathfrak{A})$ to objects in $H^0(\mathfrak{A})$. Since $H^0(\mathfrak{A})$ is $d\mathbb{Z}$-cluster tilting in $D^c(\mathfrak{A})$, it follows that $\mc$ is a $d\mathbb{Z}$-cluster tilting subcategory of $\ec$, see e.g. \cite[Theorem 6.1]{Kva21}. Hence $\mc$ is a Frobenius $d$-exact category by \Cref{Proposition:FrobdExactImpliesFrobExact}, and so $\underline{\mc}$ has the structure of a $(d+2)$-angulated category such that we have an equivalence $\underline{\mc}\cong H^0(\mathfrak{A})$ of $(d+2)$-angulated categories \cite[Theorem 5.16]{Jas16}. Therefore $\fc\cong H^0(\mathfrak{A})$ must be algebraic, which proves the claim.
\end{proof}

\section{The non-weakly idempotent complete case}\label{Section:NonWeaklyIdemComp}
We know that a $d$-cluster tilting subcategory of a weakly idempotent complete exact category inherits a $d$-exact structure by \Cref{Theorem:dCTImpliesdExactWeakIdemPotent}. However, if the exact category is not weakly idempotent complete, then the proposed structure in that theorem is not $d$-exact, since acyclicity is not preserved by weak isomorphisms.
  
  In this section we show that the closure under homotopy equivalences of the class of $d$-exact sequences in \Cref{Theorem:dCTImpliesdExactWeakIdemPotent} is a $d$-exact structure for any exact category, see \Cref{Theorem:dCTImpliesdExactNotWeakIdemPotent}. Note that this is the same as the closure under weak isomorphisms, see \Cref{Remark:EquivDescrdExactStructure}. In \Cref{Example:CounterExdExNotdCTNotValid} we explain why \cite[Example 2.5]{Ebr21} is not a counterexample to \Cref{Question:NonweaklyIdemComplete}, since it relies on the incorrect description in \cite[Theorem 4.14]{Jas16}.

\begin{theorem}\label{Theorem:dCTImpliesdExactNotWeakIdemPotent}
    Let $\mc$ be a $d$-cluster tilting subcategory of an exact category $\ec$. Then $\mc$ has the structure of a $d$-exact category, where a complex 
    \[
0\to M_{d+1}\to \cdots \to M_{0}\to 0
    \]
   in $\mc$ is an admissible $d$-exact sequence if and only if it is homotopy equivalent to an acyclic complex in $\ec$. 
\end{theorem}

\begin{proof}
 Let $\xc$ be the class of complexes described in the theorem, and let $\hat{\mc}$ and $\hat{\ec}$ denote the weak idempotent completions of $\mc$ and $\ec$, respectively. By \Cref{Prop:dCTFor(Weak)IdemComp} we know that $\hat{\mc}$ is a $d$-cluster tilting subcategory of $\hat{\ec}$, and hence by \Cref{Theorem:dCTImpliesdExactWeakIdemPotent} it is a $d$-exact category. We denote its $d$-exact structure by $\hat{\xc}$. Our strategy  is to prove that $\mc$ is closed under $d$-extensions in $\hat{\mc}$ in the sense of \cite{Kla26}, then apply \cite[Corollary 5.15]{Kla26} to conclude that $\mc$ inherits a $d$-exact structure from $\hat{\mc}$ which is equal to $\xc$.

As a first step we show that for any $\hat{E}\in \hat{\mathcal{E}}$ there exists $M\in \mc$ such that $\hat{E}\oplus M\in \ec$. Indeed, there exists an object $F\in \ec$ such that $\hat{E}\oplus F\in \ec$ by the construction of the weak idempotent completion. Now since $\mc$ is cogenerating, we can find a conflation 
 \[
 0\to F\to M\to G\to 0
 \]
 in $\ec$ where $M\in \mc$. Adding the trivial conflation $0\to \hat{E}\xrightarrow{1_{\hat{E}}}\hat{E}\to 0\to 0$, we get the conflation 
 \[
 0\to \hat{E}\oplus F\to \hat{E}\oplus M\to G\to 0
 \]
 in $\hat{\ec}$. Since $\ec$ is closed under extensions in $\hat{\ec}$ (e.g. since $\ec\to \hat{\ec}$ induces a derived equivalence by \cite[Remark 1.12]{Nee90}) it follows that $\hat{E}\oplus M\in \ec$, which proves the claim.
 
 Next we show that for any $\hat{\xc}$-admissible $d$-exact sequence
 \[
 \hat{M}_{d+1}\to \cdots \to \hat{M}_{0}
 \]
 there exists objects $N_i\in \mc$ for $0\leq i\leq d$ such that 
  \begin{multline}\label{Equation:XTildeToX}
  (\cdots \to 0\to \hat{M}_{d+1}\to \cdots \to \hat{M}_{0}\to 0\to \cdots)\oplus \bigoplus_{i=0}^{d} (\cdots \to 0\to N_i\xrightarrow{1_{N_i}}N_i\to 0 \to \cdots) \cong \\
  (\cdots \to 0 \to \hat{M}_{d+1}\oplus N_{d}\to \hat{M}_d\oplus N_{d}\oplus N_{d-1}\to \cdots \to \hat{M}_{1}\oplus N_{1}\oplus N_0\to \hat{M}_{0}\oplus N_0\to 0\to \cdots)
 \end{multline}
is an acyclic complex in $\ec$. Indeed, for $1\leq i\leq d$ let $\hat{K}_i$ denote the kernel of $\hat{M}_{i}\to \hat{M}_{i-1}$, so that we have conflations
 \[
 0\to \hat{K}_{i
}\to \hat{M}_i\to \hat{K}_{i-1}\to 0
 \]
 in $\hat{\ec}$, where $\hat{K}_d\colonequals \hat{M}_{d+1}$. We also set $\hat{K_0}=N_0$. Using the previous claim, we choose objects $N_i\in \mc$ such that $N_i\oplus \hat{K}_i\in \ec$ for all $i$. Adding the trivial conflations 
 \[
 0\to N_{i}\xrightarrow{1_{N_{i}}}N_{i}\to 0\to 0 \quad \text{and} \quad 0 \to 0 \to N_{i-1}\xrightarrow{1_{N_{i-1}}}N_{i-1}\to 0
 \]
 to the conflation above, we get a conflation 
 \begin{equation*}
    0\to \hat{K}_{i}\oplus N_{i}\to \hat{M}_i \oplus N_{i}\oplus N_{i-1}\to \hat{K}_{i-1}\oplus N_{i-1}\to 0 
 \end{equation*}
 where all the terms are in $\ec$ since $\ec$ is closed under extensions in $\hat{\ec}$. Now the complex \eqref{Equation:XTildeToX} is obtained by gluing together these conflations, and hence it must be acyclic in $\ec$. This proves the claim. 

 Finally, if $\hat{M}_0,\hat{M}_{d+1}\in \mc$ in the previous claim, then we can set $N_0=N_d=0$ to get that any $d$-exact sequence in $\hat{\xc}$ with end terms in $\mc$ is both Yoneda equivalent and homotopic to a $d$-exact sequence in $\hat{\xc}\cap \xc$. This implies in particular that $\mc$ is $d$-extension-closed in $\hat{\mc}$, and also that a $d$-exact sequence in $\hat{\xc}$ with components in $\mc$ must lie in $\xc$, since $\xc$ is closed under homotopy equivalence. Also any $d$-exact sequence in $\xc$ must lie in $\hat{\xc}$ since any complex in $\xc$ is acyclic in $\hat{\ec}$. Therefore, $\xc$ is precisely the class of $d$-exact sequences in $\hat{\xc}$ with components in $\mc$. Hence, $\xc$ is a $d$-exact structure on $\mc$ by \cite[Corollary 5.15]{Kla26}. 
 
\end{proof}

\begin{remark}\label{Remark:AdmissibleEpiVSDeflation}
We claim that a morphism $X\to Y$ in $\mc$ is an admissible monomorphism in $\mc$ if and only if it is an inflation in the weak idempotent completion $\hat{\ec}$. Indeed we only need to show that if $X\to Y$ is an inflation in $\hat{\ec}$, then it is an admissible monomorphism in $\mc$, since the other direction is clear. First note that $X\to Y$ must be an admissible monomorphism in $\hat{\mc}$, see \Cref{AdmMonoEpiWeakIdemComp}. Hence, we can find an admissible $d$-exact sequence in $\hat{\mc}$
 \[
 X\xrightarrow{}Y \to \hat{M}_{d-1}\to \cdots \to \hat{M}_{0}.
  \]
Applying the second claim in the proof of \Cref{Theorem:dCTImpliesdExactNotWeakIdemPotent} and setting $N_0=0$ in the construction, we get an acyclic complex 
 \[
X\xrightarrow{}Y\oplus N_{d-1} \to \hat{M}_{d-1}\oplus N_{d-1}\oplus N_{d-2}\to \hat{M}_{d-2}\oplus N_{d-2} \oplus N_{d-3}\to \cdots\to \hat{M}_1\oplus N_{1}\oplus N_0 \to \hat{M}_{0} \oplus N_0
 \]
 in $\ec$ with terms in $\mc$. This complex is homotopy equivalent to the complex
  \[
X\xrightarrow{}Y \to \hat{M}_{d-1}\oplus N_{d-2}\to \hat{M}_{d-2}\oplus N_{d-2} \oplus N_{d-3}\to \cdots\to \hat{M}_1\oplus N_{1}\oplus N_0 \to \hat{M}_{0} \oplus N_0.
 \]
Therefore the latter must be an admissible $d$-exact sequence in $\mc$, which implies that $X\to Y$ is an admissible monomorphism in $\mc$. 

It follows from the claim that if a morphism in $\mc$ is an inflation in $\ec$, then it is an admissible monomorphism in $\mc$. The converse does not hold in general though. 
\end{remark}

\begin{remark}\label{Remark:EquivDescrdExactStructure}
  The $d$-exact structure in \Cref{Theorem:dCTImpliesdExactNotWeakIdemPotent} can also be characterized in the following two equivalent ways, cf. \cite[Theorem 3.7]{Kla24}:
\begin{enumerate}
    \item The smallest class of $d$-exact sequences in $\mc$, closed under homotopy equivalence, and which contains all complexes
    $0\to X_{d+1}\to \cdots \to X_0\to 0$
     in $\mc$ which are acyclic in $\ec$.
    \item The smallest class of $d$-exact sequences in $\mc$, closed under weak isomorphisms, and which contains all complexes
    $0\to X_{d+1}\to \cdots \to X_0\to 0$
     in $\mc$ which are acyclic in $\ec$.
\end{enumerate}
  Indeed, this is clear for $d=1$, so we assume $d\geq 2$. Now we know that the $d$-exact structure on $\mc$ is closed under weak isomorphisms and homotopy equivalences, and hence must contain the two classes of $d$-exact sequences described above. So we only need to prove the converse inclusions. 
  
  From the second claim in the proof of \Cref{Theorem:dCTImpliesdExactNotWeakIdemPotent} we see that the $d$-exact structure on $\mc$ consists of the complexes 
    $
    X_\bullet=(X_{d+1}\to \cdots \to X_0)
    $
    in $\mc$ for which there exists objects $N_i\in \mc$ for $0\leq i\leq d$ such that 
  \[
  X_\bullet\oplus \bigoplus_{i=0}^d (\cdots \to 0\to N_i\xrightarrow{1_{N_i}}N_i\to 0 \to \cdots) 
 \]
is an acyclic complex in $\ec$ with components in $\mc$. Since this is clearly homotopy equivalent to $X_\bullet$, this shows that the $d$-exact structure is equal to the class of $d$-exact sequences described in (1). To see that it is equal to the class (2), note that the homotopy equivalence can be realized as a composite of inclusions
\[
X_\bullet\oplus \bigoplus_{i=0}^{n-1} (\cdots \to 0\to N_i\xrightarrow{1_{N_i}}N_i\to 0 \to \cdots) \to X_\bullet\oplus \bigoplus_{i=0}^n (\cdots \to 0\to N_i\xrightarrow{1_{N_i}}N_i\to 0 \to \cdots) 
\]
for $0\leq n\leq d$. Since each such inclusion is a weak isomorphism, this proves the claim.
\end{remark}

In the last part we show that the $d$-exact structure from \cite[Example 2.5]{Ebr21} arises from a $d$-cluster tilting subcategory of an exact category, contradicting a claim in that  example. The confusion stems from the differences in the $d$-exact structure for weakly idempotent complete categories, see \Cref{Theorem:dCTImpliesdExactWeakIdemPotent}, and non-weakly idempotent complete categories, see \Cref{Theorem:dCTImpliesdExactNotWeakIdemPotent}. 

\begin{example}\label{Example:CounterExdExNotdCTNotValid}
  Fix a field $k$. Following  \cite[Example 2.5]{Ebr21}, consider the non-weakly idempotent complete category $\mathcal{V}$ of finite-dimensional $k$-vector spaces of dimension $\neq 1$. It is an extension-closed subcategory of the category of all $k$-vector spaces, and therefore inherits an exact structure which is trivial in the sense that $\Ext^i_{\vc}(V,V')=0$ for all $i>0$ and $V,V'\in \vc$.  Hence, $\vc$ is a $d$-cluster tilting subcategory of itself for any $d>0$.  Its $d$-exact structure is given by all sequences
  \[
V_{d+1}\to \cdots \to V_0
  \]
  which are homotopic to acyclic complexes in $\vc$, by \Cref{Theorem:dCTImpliesdExactNotWeakIdemPotent}. These are precisely the complexes which are acyclic in the weak idempotent completion of $\vc$, i.e. in the category of all finite-dimensional $k$-vector spaces. Since all $d$-exact sequences must satisfy this, all $d$-exact sequences must be admissible. Hence, we recover the $d$-exact structure discussed in \cite[Example 2.5]{Ebr21}, which contradicts the claim in the last line in that example.
\end{example}

\begin{remark}\label{Remark:NonUniqueness}
    Assume we have an equivalence $\mc\cong \nc$ of $d$-exact categories where $\nc$ is a $d$-cluster tilting subcategory of an exact category $\ec$. If $\mc$ is not weakly idempotent complete, then neither is $\nc$, and therefore $\ec$ is not weakly idempotent complete. Hence, \Cref{Theorem:UniquenessAmbientExact} cannot be applied, so it is not clear in which way the pair $(\nc,\ec)$ can be unique if $\mc$ is fixed. In fact, if we do not assume any additional assumptions on $\ec$, then it will not be unique, see \Cref{Example:NonUniqueness}.
\end{remark}

\section*{Acknowledgements}

The author would like to thank Jenny August, Ramin Ebrahimi, Johanne Haugland, Karin Marie Jacobsen, Yann Palu and Hipolto Treffinger for useful discussions. Part of this work were carried out during the author's stay at the Centre for Advanced Study of the Norwegian
Academy of Science and Letters from November 2022 to January 2023.

	\bibliographystyle{alpha}
	\bibliography{Mybibtex}
	
\end{document}